\documentclass[11pt,leqno]{amsart}
\usepackage{amsmath,amsthm,epsfig,graphicx,color}
\usepackage{comment}
\usepackage{ascmac}
\definecolor{darkblue}{rgb}{.2, 0.2,.8}
\definecolor{carageen}{rgb}{0,0.5,0.3}
\definecolor{darkred}{rgb}{.8, .1,.1}
\newcommand{\red}{\color{darkred}}

\newcommand{\levy}{L\'evy}
\newcommand{\slln}{strong law of large numbers}

\newcommand{\ex}{{\rm e}\,}

\newcommand{\asy}{asymptotic}

\newtheorem{lemma}{Lemma}[section]

\newtheorem{theorem}[lemma]{Theorem}

\newtheorem{proposition}[lemma]{Proposition}
\newtheorem{definition}[lemma]{Definition}
\newtheorem{corollary}[lemma]{Corollary}
\newtheorem{example}[lemma]{Example}
\newtheorem{exercise}[lemma]{Exercise}
\newtheorem{remark}[lemma]{Remark}
\newtheorem{fig}[lemma]{Figure}
\newtheorem{tab}[lemma]{Table}

\newcommand{\bfi}{{\bf i}}

\newcommand{\bfu}{{\bf u}}
\newcommand{\bfv}{{\bf v}}

\newcommand{\bfU}{{\bf U}}

\newcommand{\bth}{\begin{theorem}}
\newcommand{\ethe}{\end{theorem}}

\newcommand{\bre}{\begin{remark}\em }
\newcommand{\ere}{\end{remark}}

\newcommand{\ble}{\begin{lemma}}
\newcommand{\ele}{\end{lemma}}

\newcommand{\pp}{point process}
\newcommand{\bde}{\begin{definition}}
\newcommand{\ede}{\end{definition}}
\newcommand{\bco}{\begin{corollary}}
\newcommand{\eco}{\end{corollary}}

\newcommand{\bpr}{\begin{proposition}}
\newcommand{\epr}{\end{proposition}}

\newcommand{\bexer}{\begin{exercise}}
\newcommand{\eexer}{\end{exercise}}

\newcommand{\bexam}{\begin{example}}
\newcommand{\eexam}{\end{example}}

\newcommand{\bfig}{\begin{fig}}
\newcommand{\efig}{\end{fig}}

\newcommand{\btab}{\begin{tab}}
\newcommand{\etab}{\end{tab}}

\newcommand{\lhs}{left-hand side}

\newcommand{\rv}{random variable}

\newcommand{\var}{{\rm var}}

\newcommand{\cov}{{\rm cov}}

\newcommand{\as}{{\rm a.s.}}

\newcommand{\Pois}{\rm Poisson}

\newcommand{\rhs}{right-hand side}

\newcommand{\beao}{\begin{eqnarray*}}
\newcommand{\eeao}{\end{eqnarray*}\noindent}

\newcommand{\beam}{\begin{eqnarray}}
\newcommand{\eeam}{\end{eqnarray}\noindent}

\newcommand{\beqq}{\begin{equation}}
\newcommand{\eeqq}{\end{equation}\noindent}

\newcommand{\bce}{\begin{center}}
\newcommand{\ece}{\end{center}}

\newcommand{\barr}{\begin{array}}
\newcommand{\earr}{\end{array}}
\newcommand{\cadlag}{c\`adl\`ag}

\newcommand{\stp}{\stackrel{\P}{\rightarrow}}
\newcommand{\std}{\stackrel{d}{\rightarrow}}
\newcommand{\stas}{\stackrel{\rm a.s.}{\rightarrow}}

\newcommand{\eqd}{\stackrel{d}{=}}

\newcommand{\vague}{\stackrel{\lower0.2ex\hbox{$\scriptscriptstyle
                    \it{v} $}}{\rightarrow}}
\newcommand{\weak}{\stackrel{\lower0.2ex\hbox{$\scriptscriptstyle
                    \it{w} $}}{\rightarrow}}
\newcommand{\what}{\stackrel{\lower0.2ex\hbox{$\scriptscriptstyle
                    \it{\hat{w}} $}}{\rightarrow}}

\newcommand{\bdis}{\begin{displaymath}}
\newcommand{\edis}{\end{displaymath}\noindent}

\newcommand{\R}{\mathbb{R}}

\newcommand{\nto}{n\to\infty}

\newcommand{\ov}{\overline}
\newcommand{\wt}{\widetilde}
\newcommand{\wh}{\widehat}
\newcommand{\vep}{\varepsilon}

\newcommand{\la}{\lambda}

\newcommand{\bbr}{{\mathbb R}}

\newcommand{\con}{convergence}

\newcommand{\st}{such that}
\newcommand{\fif}{if and only if}
\newcommand{\wrt}{with respect to}
\newcommand{\chf}{characteristic function}
\newcommand{\fct}{function}

\newcommand{\ds}{distribution}

\newcommand{\rep}{representation}

\newcommand{\seq}{sequence}

\newcommand{\pro}{probabilit}

\newcommand{\ms}{measure}

\newcommand{\bfx}{{\bf x}}
\newcommand{\bfX}{{\bf X}}

\newcommand{\bfY}{{\bf Y}}

\newcommand{\bfZ}{{\bf Z}}

\newcommand{\bfa}{{\bf a}}
\newcommand{\bfb}{{\bf b}}

\newcommand{\bft}{{\bf t}}

\newcommand{\bfs}{{\bf s}}

\newcommand{\E }{{\mathbb E}}
\renewcommand{\P }{{\mathbb P}}

\newcommand{\1}{{\mathbf 1}}

\allowdisplaybreaks

\begin{document}
\today
\bibliographystyle{plain}
\title[Distance covariance for random fields]{Distance covariance
for random fields}

%\thanks{Muneya Matsui's research is partly supported by the JSPS Grant-in-Aid for Scientific Research C
%(19K11868). 
%Thomas Mikosch's research is partially supported by Danmarks Frie Forskningsfond Grant No 9040-00086B}

\author[M. Matsui]{Muneya Matsui}
\address{Department of Business Administration, Nanzan University, 18
Yamazato-cho, Showa-ku, Nagoya 466-8673, Japan.}
\email{mmuneya@gmail.com}

\author[T. Mikosch]{Thomas Mikosch}
\address{Department  of Mathematics,
University of Copenhagen,
Universitetsparken 5,
DK-2100 Copenhagen,
Denmark}
\email{mikosch@math.ku.dk}
\author[R. Roozegar]{Rasool Roozegar}
\address{Department of Statistics, Yazd University, P.O. Box 89195-741, Yazd, Iran}
\email{rroozegar@yazd.ac.ir}
\author[L. Tafakori]{Laleh Tafakori}
\address{Department of Mathematical Sciences,
School of Science,
RMIT University, Melbourne, VIC, 
Australia}
\email{laleh.tafakori@rmit.edu.au} 

\begin{abstract}{
We study an independence test based on distance correlation for 
random fields $(X,Y)$. We consider the situations when $(X,Y)$ is observed 
on a lattice with equidistant grid sizes and when $(X,Y)$ is observed at 
random locations. We provide \asy\ theory for the sample distance correlation
in both situations and show bootstrap consistency. The latter fact allows 
one to build a test for independence of $X$ and $Y$ based on the considered 
discretizations of these fields. We illustrate the performance of the bootstrap test
by simulations, and apply the test to Japanese meteorological data observed over the entire area of Japan.}\\
%We illustrate the performance of the bootstrap test 
%in a simulation study involving fractional Brownian and infinite variance stable fields.
%The independence test is applied to Japanese meteorological data, which are observed over the entire area of Japan. }\\
Keywords: Empirical \chf , distance covariance, random field, independence test 
\end{abstract}
%\keywords{Empirical \chf , distance covariance, random field, test of independence}
\subjclass{Primary 62E20; Secondary 62G20 62M99 60F05 60F25}
\maketitle

\section{Introduction to model}
\subsection{Distance covariance in Euclidean space and literature review}
It is well known that  two $q$- and $r$-dimensional random vectors $\bfX$ and
$\bfY$, respectively,  are independent \fif\ their joint  \chf\ 
factorizes, i.e.,
\beao%\label{eq:indch}
\varphi_{\bfX,\bfY}(\bfs,\bft)&=&
\E\big[\exp(i\,\bfs^\top \bfX+i\bft^\top \bfY)\big]\\
&=&
\E\big[\exp(i\,\bfs^\top \bfX)\big]\,\E\big[\exp(i\,\bft^\top \bfY)\big]\\&=&
\varphi_\bfX(\bfs)\,\varphi_\bfY(\bft)\,,\qquad
\bfs\in\bbr^q\,,\bft\in\bbr^r\,.
\eeao
%\beao%\label{eq:indch}
%\varphi_{\bfX,\bfY}(\bfs,\bft)=
%\E\big[\exp(i\,\bfs^\top \bfX+i\bft^\top \bfY)\big]
%=
%\E\big[\exp(i\,\bfs^\top \bfX)\big]\,\E\big[\exp(i\,\bft^\top \bfY)\big]=
%\varphi_\bfX(\bfs)\,\varphi_\bfY(\bft)\,,%\qquad
%\bfs\in\bbr^q\,,\bft\in\bbr^r\,.
%\eeao
%where $\bfs\in\bbr^q\,,\bft\in\bbr^r$. 
However, this identity is difficult to check if one has data 
at the disposal; a replacement of the corresponding \chf s 
by empirical versions does not lead to powerful statistical tools for
detecting independence between $\bfX$ and $\bfY$. 
In the univariate case, Klebanov and Zinger 
\cite{klebanov:zinger:1990},  Feuerverger \cite{feuerverger:1993}, and later Sz\'ekely et al. \cite{szekely:rizzo:bakirov:2007,szekely:rizzo:2009,szekely:rizzo:2013,szekely:rizzo:2014} in the general multivariate case recommended to
use a weighted $L^2$-distance between  $\varphi_{\bfX,\bfY}$ and 
$\varphi_\bfX\,\varphi_\bfY$: for $\beta\in (0,2)$, the {\em distance covariance between $\bfX$ and $\bfY$} is given by
\beao
T_\beta(\bfX,\bfY)= c_qc_r\int_{\bbr^{q+r}} \big|\varphi_{\bfX,\bfY}(\bfs,\bft)-\varphi_\bfX(\bfs)\varphi_\bfY(\bft)\big|^2 
|\bfs|^{-(q+\beta)}|\bft|^{-(r+\beta)}\,d\bfs\, d\bft\,,
\eeao
where the constants $c_d$ for $d\ge 1$  are chosen \st\
\beam\label{eq:may9}
c_d\,\int_{\bbr^d} (1-\cos(\bfs'\bfx))\,|\bfx|^{-(d+\beta)}d\bfx= |\bfs|^\beta\,.
\eeam
Here and in what follows, we suppress the dependence of the 
Euclidean norm $|\cdot|$ on the dimension; it will always be clear
from the context what the dimension is.
\par
The quantity $T_\beta(\bfX,\bfY)$ is finite under suitable moment conditions on $\bfX,\bfY$.
The corresponding {\em distance correlation} is given by
\beam\label{eq:august20a}
R_\beta(\bfX,\bfY)= \dfrac{T_\beta(\bfX,\bfY)}{\sqrt{T_\beta(\bfX,\bfX)}\sqrt{T_\beta(\bfY,\bfY)}}\,.
\eeam
Of course, $\bfX$ and $\bfY$ are independent \fif\ $R_\beta(\bfX,\bfY)=T_\beta(\bfX,\bfY)=0$.
\par 
Thanks to the choice of the weight \fct\ $|\bfs|^{-(q+\beta)}|\bft|^{-(r+\beta)}$
and \eqref{eq:may9},
$T_\beta(\bfX,\bfY)$ has an explicit form: assuming that $(\bfX_i,\bfY_i)$, $i=1,2,\ldots,$ are iid copies of $(\bfX,\bfY)$, we have
\beam\label{eq:may8}
T_\beta(\bfX,\bfY)&=& \E [|\bfX_1-\bfX_2|^\beta|\bfY_1-\bfY_2|^\beta]+  
\E [|\bfX_1-\bfX_2|^\beta]\E[|\bfY_1-\bfY_2|^\beta]\nonumber\nonumber\\&& -
2\,\E [|\bfX_1-\bfX_2|^\beta|\bfY_1-\bfY_3|^\beta]\,,
\eeam 
and $R_\beta(c\bfX,c\bfY)=R_\beta(\bfX,\bfY)$ for $c\in\bbr$, i.e., $R_\beta$
is scale-invariant. 

The definition already points at some 
nice properties of the distance covariance/correlation. It vanishes
if and only if $\bfX,\bfY$ are independent, hence one can capture  
arbitrary dependence structures such as non-linear and non-monotone onces. 
By a handy choice of the weight \fct\ as in \eqref{eq:may9} 
an explicit interpretation through  moments is 
available. It translates to the empirical moments (cf. \eqref{eq:may8b},\eqref{eq:may8c}) which can be easily calculated. Moreover, the empirical distance covariance 
is approximated by a (degenerate) $U$-statistic whose theoretical properties
have been studied in great detail in the literature. 
Due to the aforementioned advantages the empirical distance correlation
has found applications in a rather wide area where dependence modeling is crucial: time series  \cite{zhou:2012,dmmw:2018,fokianos:pitsillou:2018}, 
functional time series \cite{hlavka:huskova:meintanis:2021,meintanis:huskova:hlavka:2021}, continuous-time stochastic processes \cite{matsui:mikosch:samorodnitsky:2017,dehling:matsui:mikosch:samorodnitsky:tafakori:2019}.
Moreover, taking advantage of the arbitrariness of dimensions, 
applications in high-dimensional data have been studied intensively in recent years; see 
\cite{szekely:rizzo:2013,yao:zhang:shao:2018,Gao:Fan:Lv:Shao:2020,zhu:zhang:yao:shao:2020}. 
\par
There are numerous other applications of distance correlation 
in addition to those mentioned above. However, to the best of our knowledge,
applications  to continuous-parameter random fields (which have 
infinite dimension) have not been considered.
Some work has been conducted on independence tests for continuous-parameter  
random fields, for example  
\cite{huang:huang:tsay:pan:2021},  where a high-dimensional 
case is tractable, but  a dimension reduction approach was used. In this paper we consider an independence test based 
on distance correlation for random fields, exploiting  
the arbitrariness of dimensions and making use of techniques developed for 
tests on stochastic processes.
Our approach will be clarified in the following text.

\subsection{Distance covariance for 
random fields on a lattice in $[0,1]^d$}\label{subsec:latticecase}
Sz\'ekely et al. \cite{szekely:rizzo:2013} showed that distance 
correlation fails for high-dimensional vectors $\bfX,\bfY$ if
their components are independent. Therefore Matsui et al. \cite{matsui:mikosch:samorodnitsky:2017}
and  Dehling et al. \cite{dehling:matsui:mikosch:samorodnitsky:tafakori:2019}
required some dependence structure on the components. They
extended distance covariance and distance correlation to stochastic 
processes. In particular,  \cite{dehling:matsui:mikosch:samorodnitsky:tafakori:2019} studied discretizations of stochastic processes $X,Y$ on $[0,1]$.
Instead of using the vectors 
$(X(t_i))_{i=1,\ldots,p}$ and $(Y(t_i))_{i=1,\ldots,p}$ for 
partitions $t_0=0<t_1<\cdots<t_p=1$, $\Delta_i=(t_{i-1},t_i]$ \st\ $p=p_n\to\infty$ and 
$\delta_n=\max_i |\Delta_i|=\max_i (t_i-t_{i-1})\to 0$ as $\nto$, one introduces the 
weighted vectors
\beam\label{eq:discret}
\bfX^{(p)}= \big( \sqrt{t_i-t_{i-1}}\,X(t_i))_{i=1,\ldots,p}\quad \mbox{and}\quad 
\bfY^{(p)}= \big( \sqrt{t_i-t_{i-1}}\,Y(t_i))_{i=1,\ldots,p}.
\eeam
\par
 Here we introduce a direct analog of this approach to $[0,1]^d$
for $d>1$. 
We assume that any random field 
$Z=(Z(\bfu))_{\bfu\in[0,1]^d}$ of interest is observed on a 
lattice with constant mesh size. We start by partitioning 
$[0,1]$ into equidistant points 
$t_i=i/q$, $i=0,1,\ldots,q$, for some positive integer $q$. From them we construct a
lattice in  $[0,1]^d$ via the points
\beao
\bft_{\bfi}=(t_{i_1},\ldots,t_{i_d}),\qquad \bfi=(i_1,\ldots,i_{d})\in \{0,1,\ldots,q\}^d\,.
\eeao
We have $p:=q^d$ lattice points in $(0,1]^d$.
We discretize $Z$ on the cells $\Delta_{\bfi}= (\bft_{\bfi-\1},\bft_\bfi]$
with volume
\beam\label{partition}
 |\Delta|=|\Delta_{\bfi}|=p^{-1}\,.
\eeam
Consider the vector
\beao
 \bfZ_{p}= \big(|\Delta|^{1/2}
 Z(\bft_{\bfi})\big)\,,\qquad \bfi\in \Pi_d=\{1,\ldots,q\}^d\,,
\eeao
and a step-\fct\  approximation to  $Z$: 
\beam\label{eq:march5a}
  Z^{(p)}(\bft)= \sum_{\bfi\in \Pi_d} Z(\bft_\bfi) \1(\bft\in \Delta_\bfi),\qquad \bft\in[0,1]^d\,. 
\eeam
For a measurable bounded square-integrable field 
$Z$ on $[0,1]^d$ we have for a.e. sample path 
\beam\label{eq:may8a}
 |\bfZ_{p}|^2 &=& \dfrac 1 p\,
\sum_{\bfi\in\Pi_d} Z^2(\bft_\bfi)
 = \sum_{\bfi\in\Pi_d} Z^2(\bft_{\bfi})\, |\Delta_{\bfi}|
 = \| Z^{(p)}\|^2_2\nonumber\\ 
& \to& \int_{[0,1]^d} Z^2(\bft)\,d\bft
=\|Z\|_2^2\,,\qquad p\to \infty,  
\eeam
where $\|\cdot \|_2$ denotes the $L_2$-norm on $[0,1]^d$.
Motivated by \eqref{eq:may8} and this approximation, we define the distance covariance 
 $T_\beta(X,Y)$,  $\beta\in (0,2)$,
between
two random fields $X,Y$ on some bounded Borel set $B\subset \bbr^d$ 
of finite positive Lebesgue \ms\  
 by 
\beam\label{eq:dcovfield}
T_\beta(X,Y)&=& \E\big[
\|X_1-X_2\|_2^\beta\, \|Y_1-Y_2\|_2^\beta
\big] + \E\big[
\|X_1-X_2\|_2^\beta \big] \,\E\big[\|Y_1-Y_2\|_2^\beta
\big] \nonumber\\
& &-2\,\E\big[\|
X_1-X_2
\|_2^\beta \|Y_1-Y_3\|_2^\beta \big]\,,
\eeam
where $(X_i,Y_i)$, $i=1,2,\ldots,$ are iid copies of $(X,Y)$,
and the distance correlation $R_\beta(X,Y)$ is defined correspondingly.
(The $L^2$-norm $\|\cdot\|_2$ is to be understood on $B$.)
A further motivation is the following fact:
\ble\label{lem:ax}
Assume that $\beta\in(0,2)$, $X,Y$ are random fields on $B$
which are square-integrable, bounded, stochastically continuous and 
$\E[\|X\|_2^\beta+\|Y\|_2^\beta+\|X\|_2^\beta\|Y\|_2^\beta]<\infty.$ Then 
$T_\beta(X,Y)=0$ \fif\ $X,Y$ are independent.
\ele
For $\beta\le1$ the statement follows from Lyons \cite{lyons:2013},
who also addressed the more general problem of distance covariance in 
metric spaces, and for $\beta\in (0,2)$ 
a straightforward modification of the proofs in Dehling et al. \cite{dehling:matsui:mikosch:samorodnitsky:tafakori:2019} for random fields
on general Borel sets $B$
yields the result.
\par
Sample versions of $T_\beta$ and $R_\beta$ are given by
\beam
\qquad T_{n,\beta}(X,Y) &=& \frac{1}{n^2} \sum_{k,l=1}^n \|X_k-X_l\|_2^\beta \|Y_k-Y_l\|_2^\beta + \frac{1}{n^2} \sum_{k,l=1}^n \|X_k-X_l\|_2^\beta \frac{1}{n^2}
 \sum_{k,l=1}^n \|Y_k-Y_l\|_2^\beta \label{eq:may8b}\\
        &  &\qquad -2 \frac{1}{n^3} \sum_{k,l,m=1}^n \|X_k-X_l\|_2^\beta \|Y_k-Y_m\|_2^\beta\,,\nonumber \\
\qquad R_{n,\beta}(X,Y)&=&\dfrac{T_{n,\beta}(X,Y)}{\sqrt{
T_{n,\beta}(X,X) T_{n,\beta}(Y,Y)
}}\,.\label{eq:may8c}
\eeam
Since $T_{n,\beta}$ is a $V$-statistic, under suitable moment conditions
$T_{n,\beta}(X,Y)$ and  $R_{n,\beta}(X,Y)$
are consistent estimators of $T_\beta(X,Y)$ and $R_\beta(X,Y)$, respectively, 
and if $X,Y$ are independent, one also has
\beao
n\,R_{n,\beta}(X,Y)\std \sum_{i=1}^\infty \la_i\,(N_i^2-1) +{\rm const}, 
\eeao
for a square-summable real \seq\ $(\la_i)$, iid $N(0,1)$ \rv s $(N_i)$.
These \asy\ results 
in combination with the fact that $T_\beta(X,Y)=0$ \fif\ $X,Y$ are independent
encourage one to build a statistical test about independence of $X,Y$
on the quantities $T_{n,\beta}(X,Y)$. 
\par
Unfortunately, a sample of paths of $(X,Y)$ is rarely at our disposal, and so one has to think about discretizations
of $(X,Y)$. Motivated by the Riemann sum approximation \eqref{eq:may8a}
possible choices are the lattice discretizations 
$(X^{(p)},Y^{(p)})$ in \eqref{eq:march5a} and, for an iid sample
$(X_i,Y_i)$, $i=1,\ldots,n$, the sample version $T_{n,\beta}(X^{(p)},Y^{(p)})$
as replacement of the test statistic
$T_{n,\beta}(X,Y)$.
In agreement with 
\cite{dehling:matsui:mikosch:samorodnitsky:tafakori:2019} in the case 
$d=1$ we will show that this idea can be made to work for $d>1$ if 
$p=p_n\to\infty$.

\subsection{Distance covariance for random fields
at random locations}\label{subsec:randomlocation}
In this paper, we follow a second  path of research. 
\begin{itemize}
\item
We introduce distance covariance $T_\beta(X,Y)$  and $R_\beta(X,Y)$ for 
random fields $X,Y$ on some Borel set $B\subset \bbr^d$ of
positive Lebesgue \ms .
\item
We define $T_\beta(X^{(p)},Y^{(p)})$ and $R_\beta(X^{(p)},Y^{(p)})$ for non-lattice
based
discretizations $X^{(p)},Y^{(p)}$
of $X,Y$ on $B$. In contrast to \eqref{eq:discret}, we 
choose a random number $N_p$ of random locations $(\bfU_i)$ where the 
processes $X,Y$ are observed. Typically, these locations are uniformly 
distributed on $B$ and $N_p\stas \infty$ as 
$p$ increases with the sample size $n$ to infinity. 
\end{itemize}
The second idea has already been advocated in Matsui et al. \cite{matsui:mikosch:samorodnitsky:2017}. There it was 
assumed that $(\bfU_i)$ are the points of a Poisson process with 
constant intensity $p$.
For statistical purposes this means that one would have to observe  
a sample  of iid copies $(X_i^{(p)},Y_i^{(p)})$ of $(X^{(p)},Y^{(p)})$
at random locations that change across $i$ and their number would 
change as well. 
The \asy\ analysis of this setting is not very elegant and 
does not lead beyond consistency of $T_{n,\beta}(X^{(p)},Y^{(p)})$.
\par
To be precise, we consider bounded 
stochastically continuous square-integrable random fields
$X,Y$ on a bounded Borel set $B\subset \bbr^d$ of positive Lebesgue \ms .
We define the {\em distance covariance} $T_\beta(X,Y)$ 
between $X,Y$ as in \eqref{eq:dcovfield}
but the norm is now given by the (standardized) $L^2$-norm on $B$:
\beam\label{eq:norm2}
\|f\|_2= \Big(|B|^{-1}\int_B f^2(\bfu)\,d\bfu\Big)^{1/2}\,.
\eeam
For ease of notation,  we assume without loss of generality that the Lebesgue
\ms\ $|B|$ of $B$ is one. 
The {\em distance correlation} $R_\beta(X,Y)$  is defined correspondingly.
\par
Consider the \pp
\beam\label{eq:august17a} 
N^{(p)}(\cdot)=\sum_{i=1}^{N_p} \vep_{\bfU_i}(\cdot)= \#\{i\le N_p: \bfU_i\in\cdot\}\,,\qquad p>0\,,
\eeam 
on the state space $B$
where $(\bfU_i)$ is iid uniform on $B$ independent of the counting number 
$N_p\stas \infty$ as $p\to\infty$. If $N_p$ is Poisson distributed and $p$ is fixed then $N^{(p)}$ constitutes a homogeneous Poisson process on $B$
; see p. 132 in Resnick \cite{resnick:1987}. In view of the order statistics property of any  
homogeneous  Poisson process $N$ on $B$ we also know that the points of $N$,
conditionally on $N(B)$, are iid uniform on $B$. However, the 
\rep\   \eqref{eq:august17a} goes beyond the order statistics 
property since we require the point \seq\ $(\bfU_i)$ to be the same for 
all $p>0$.
\par
Recycling the notation $X^{(p)},Y^{(p)}$, the discretizations of $X,Y$ are now 
given by
\beao
X^{(p)}(\bfu) = \sum_{i=1}^{N_p} X(\bfu) \1 _{\{ {\bfU_i}\}}(\bfu)\quad
\text{and}\quad Y^{(p)}(\bfu) = \sum_{i=1}^{N_p} Y(\bfu) \1_{{\{\bfU_i\}}} 
(\bfu)\,,\qquad \bfu\in B\,,
\eeao
%\beao
%X^{(p)}= \dfrac 1 N_p\sum_{i:\Gamma_i\in B} X(\Gamma_i)\,,\quad Y^{(p)}= \dfrac 1 N_p\sum_{i:\Gamma_i\in B} Y(\Gamma_i)\,,\qquad i=1,\ldots,N_p,
%\eeao
where $(N^{(p)})_{p>0}$ and $X,Y$ are independent. 
If $N_p=0$, $X^{(p)}=Y^{(p)}=0$ by convention.
 %Writing for any process $Z$ on $B$
%(hereby abusing notation)
%\beao
%\bfZ^{(p)}=  \big(p^{-1/2}\,Z(\Gamma_i)\big)_{i=1,\ldots,p},
%\eeao
Again recycling the symbol $\|\cdot\|_2$, we write 
\beam\label{eq:lln}
\|X^{(p)}-Y^{(p)}\|_2^2 := \dfrac 1 N_p\sum_{i:\bfU_i\in B} (X-Y)^2(\bfU_i)
&=& \dfrac 1 N_p\, \int_B (X-Y)^2\,d\,N^{(p)}\,,%\stp \|X-Y\|_2^2\,,\qquad p\to\infty\,.\nonumber\\
\eeam
with the convention that the \rhs\ is zero if $N_p=0$. 
We notice that for a.e. realization of $(N^{(p)})_{p>0}$, 
$\bfU_1,\ldots,\bfU_{N_p}$ are  getting arbitrarily tight in $B$ 
since we assume 
$N_p\stas \infty$ as $p\to\infty$. 
If we assume that $X,Y$ are path-wise square-integrable on $B$ and we condition 
on $(N^{(p)})_{p>0}$ then for a.e. realization of $X,Y$
we have the Riemann sum approximation $\|X^{(p)}-Y^{(p)}\|_2^2 \to  \|X-Y\|_2^2$.
\par
In what follows, we define the sample versions $T_{n,\beta}(X,Y),R_{n,\beta}(X,Y)$ 
in the natural way by applying the $L^2$-norm \eqref{eq:norm2} on $B$. 
We assume that the iid \seq\ $((X_i,Y_i))$ is independent of 
$(N^{(p)})_{p>0}$. Then the random locations $(\bfU_i)$ are the same for 
any  sample size $n$, and $((X_i^{(p)},Y_i^{(p)}))$ are iid, conditional on 
$(N^{(p)})_{p>0}$.
Moreover, we define $T_{n,\beta}(X^{(p)},Y^{(p)}),R_{n,\beta}(X^{(p)},Y^{(p)})$ 
accordingly by using the notation in \eqref{eq:lln}. 
\par
The random discretizations of $X,Y$
have some advantages over the lattice case: 
\begin{itemize}
\item
they can be defined 
on quite general sets $B$,
\item
the random fields can be observed on 
irregularly spaced locations, 
\item
the smoothness of the field
does not play a significant role for the \asy\ theory of 
the sample distance covariance.
\end{itemize}

Of course, one needs preliminary confirmation about the Poisson property 
of the counting variable $N_p$ and the
uniformity of the locations $(\bfU_i)$. There exists a large literature 
on tests for Poisson point process models 
(see e.g. \cite{ripley:1979,moller:waagepetersen:2003,diggle:2013} and references therein). 
This research topic is still progressing, and tests on arbitrary 
multi-dimensional subsets of $\R^d$ or 
even on sets of unknown support have been developed e.g. in \cite{bcv:2006,bcv:2012,ebner:nestmann:schulte:2020}.

\subsection{Organization of the paper}
We provide the necessary \asy\ theory for the 
two aforementioned methods:
\begin{enumerate} 
\item[(i)] in Section~\ref{sec:randomlocation} 
for random locations and increasing intensity $p$ on a bounded 
Borel set $B\subset \bbr^d$ of positive Lebesgue measure,
\item[(ii)] in Section~\ref{sec:lattice}
for the lattice case on $B=[0,1]^d$ for increasing 
$p$ which is the total (deterministic) 
number of grid points.
\item[(iii)] in Section~\ref{sec:onlyrandomlocation} we combine the two methods by averaging
the observations in each cell of the regular lattice on  $B=[0,1]^d$. 
\end{enumerate}
Since the (non-Gaussian) limit \ds s of the considered discretized 
sample distance covariances are not tractable, in
Section~\ref{asymptotic} we consider some modifications of 
these quantities. Roughly speaking, we find approximating
degenerate $U$-statistics to the sample distance covariances and 
apply bootstrap techniques tailored for these $U$-statistics. We show the consistency of the bootstrap. 
In the case of random locations,
this part is quite delicate and rather different from the theory developed in Dehling et al. \cite{dehling:matsui:mikosch:samorodnitsky:tafakori:2019}
 in the lattice case for $d=1$. Indeed, we show bootstrap consistency for these U-statistics {\it conditional} on the random locations 
and their number.
%We show the consistency of the bootstrap. 
In Section~\ref{sec:simulations} we illustrate 
how the \con\ of the sample distance correlation depends on the sample
size and the number of discretization points. We also 
show how the bootstrap performs for 
the sample distance correlation on selected discretized random 
fields, in particular for fractional Brownian and infinite variance stable 
\levy\ sheets. In Section~\ref{SecEmp}
we apply the sample distance correlation for testing the independence of Japanese meteorological data observed at 
stations all over Japan. We interpret their number as random and their location as uniformly distributed.
%\textcolor{red}{In Section~\ref{SecEmp} we illustrate the use of empirical distance correlation for testing independence of fields 
%through Japanese meteorological observations over the whole area of Japan.} 

In the remaining sections we provide the proofs of the main 
results. 

\section{Random field at random locations}\label{sec:randomlocation}\setcounter{equation}{0}
\subsection{Technical conditions}
We use the notation and assumptions  of Section~\ref{subsec:randomlocation}.
In particular,
$X,Y$ are measurable bounded stochastically continuous 
random fields on a bounded Borel set $B\subset \bbr^d$ of positive Lebesgue \ms\
with Riemann square-integrable sample paths on $B$. Also recall that $p=p_n\to\infty$
and $N_p\stas \infty$ as $\nto$.
\par
In what follows, $c$ denotes any positive constant whose value is not of 
interest.
\par
The following result gives some \asy\ results for the sample 
distance covariance. They are the basis for proving Theorem~\ref{thm:main}.
\bpr\label{prop::asymptotics1}
Choose $\beta\in (0,2)$ and assume that $X,Y$ are independent.
Consider the following conditions:
\begin{enumerate} 
\item[\rm (1a)]
$X,Y$ have finite second moment and 
$\int_B \E\big[X^2(\bfu)+Y^2(\bfu)\big]\,d\bfu<\infty$.
\item[\rm (1b)] 
$X,Y$ have infinite second moment but
$\E\big[\sup_{\bfu \in B}|X|^\beta(\bfu)+ \sup_{\bfu \in B}|Y|^\beta(\bfu)\big]<\infty$.
\item[\rm (2a)] $X$ has finite second moment, $\big[\beta\in(0,1]$ and 
$\int_B \E[X^2(\bfu)]\,d\bfu<\infty\big]$ or 
$\big[\beta\in (1,2)$
and $\int_B\E[|X|^{2\beta}(\bfu)]\,d\bfu<\infty\big]$.
\item[\rm (2b)]  $X$ has infinite second moment, $\beta\in (0,1)$
and $\E[\sup_{\bfu \in B} |X|^{2\beta}(\bfu)]<\infty$.
\end{enumerate}
If {\rm (1a)} or {\rm (1b)} hold then
\beam\label{app1distacecov1}
 \E\big[\big|T_{n,\beta}(X^{(p)},Y^{(p)})-T_{n,\beta}(X,Y)\big|\;\big|\;N^{(p)}\big] 
\stas 0\,,\qquad \nto\,.
\eeam
If {\rm (2a)} or {\rm (2b)} hold then
\beam\label{eq:august19a}
\P\big(|T_{n,\beta}(X^{(p)},X^{(p)})-T_{n,\beta}(X,X)|>\vep \;\big|\;N^{(p)}\big)\stas 0\,,\qquad \nto\,. 
\eeam
\epr
The proof of this proposition is given in Section~\ref{sec:proofprop41}.
\subsection{Some examples}
We consider some examples of random fields and discuss the fulfillment 
of the conditions in Proposition~\ref{prop::asymptotics1}.
\bexam\label{exam:fbm} %{\red check this example}
\rm Let $X(\bfu)=B^H(\bfu)$, $\bfu\in [0,1]^d$,
be a {\it fractional Brownian sheet} with Hurst parameter $H=(H_1,\ldots,H_d)\in
 (0,1)^d$ introduced by Kamont \cite{Kamont:1996}; see also Ayache and Xiao \cite{ayache:xiao:2005}. It is a centered 
Gaussian random field with continuous sample paths and
 covariance \fct\
\beam\label{eq:fbma}
 \cov(B^H(\bfs), B^H( \bft))=\prod_{i=1}^d \frac{1}{2}(
 |s_i|^{2H_i}+|t_i|^{2H_i}-|s_i-t_i|^{2H_i}),\qquad \bfs,\bft\in [0,1]^d\,.
\eeam
%for $s=(s_1,\ldots,s_d)$ and $t=(t_1,\ldots,t_d)$.
Moreover, for all $\gamma>0$, $\E[|B^H(\bfu)|^{\gamma}]$ is a continuous 
\fct\ of $\bfu$. Therefore all conditions on $X$ are satisfied.
\eexam
\bexam\label{exam:levybm}\rm 
Let $H\in (0,1)$ and $X^H$ be a centered Gaussian 
random field on $[0,1]^d$ with covariance \fct
\beao
\cov(X^H(\bfs), X^H( \bft))= \dfrac 1 2 \big(|\bfs|^{2H}+|\bft|^{2H}-
|\bfs-\bft|^{2H}\big)\,,\qquad \bfs,\bft\in [0,1]^d\,.
\eeao
This process is called {\em \levy\ fractional Brownian field}; see Samorodnitsky and Taqqu \cite{samorodnitsky:taqqu:1994}, p. 393.
It has stationary increments and continuous sample paths and moments 
of any order.
\eexam
\bexam\label{exam:levy}\rm
Let $X$ be a centered {\em  \levy\ sheet} on $[0,1]^d$ with
 L\'evy-Khintchine triplet $(\mu,\sigma^2,\nu)$; see Khoshnevisan and Xiao \cite{Khoshnevisan:xiao:2004}, Example 2.1. This means that, 
for disjoint intervals $A_i \in [0,1]^d$, $i=1,\ldots,k$, and $k\ge 1$,  
the increments  $X(A_1),\ldots,X(A_k)$ are independent and 
for each interval $A=(\bfs,\bft]$, the characteristic function of 
the increment 
\[
 X(A) = \sum_{k_1\in \{0,1\}}\cdots \sum_{k_d \in \{0,1\}} (-1)^{d-\sum_{l=1}^dk_l}X(s_1+k_1(t_1-s_1),\ldots,s_d+k_d(t_d-s_d)), 
\]
is given by 
\begin{align}
\label{chf:levysheet}
  \E[\ex^{i u X(A)}] = \exp(-|A|\,\Psi(u))\,,
\end{align}
where
\[
 \Psi(u) = i \mu u + u^2\sigma^2/2+ \int_\R (\ex^{iu x}-1- ux \1_{\{|x| \le 1\}}) \nu (dx).   
\]
%It has independent stationary  increments and \cadlag\ sample paths.  
If $X$ has finite $\gamma$th moment for some $ \gamma>0$
then $ \E[|X(\bfu)|^{\gamma}]$ is a continuous \fct\ of $\bfu$ and 
the \fct\  $\E[|\sup_{u\in [0,1]^d}X(\bfu)|^{\gamma}]$ is finite as well.
The field $X$ 
has \cadlag\ sample paths which are bounded and square-integrable on 
$ [0,1]^d$.
\eexam

\bexam\label{exam:stab}\rm 
Consider a symmetric {\it $\alpha$-stable \levy\ sheet} on $[0,1]^d$
for some $\alpha\in (0,2)$; see for example Ehm \cite{ehm:1981}, Samorodnitsky and Taqqu \cite{samorodnitsky:taqqu:1994}. 
This field has infinite variance 
and is a \levy\ random sheet whose increments have a symmetric
$\alpha$-stable \ds , i.e. the \chf\ of $X(A)$ is given 
by \eqref{chf:levysheet} with $\Psi(u)=c^\alpha |u|^\alpha$ for some $c>0$.  
It has moments of order $\gamma\in (0,\alpha)$. 
\eexam
\subsection{Main result}
The following statement is the main result  in the case of random locations.
\bth\label{thm:main}
Choose $\beta\in (0,2)$ and assume that $X,Y$ are independent, 
\begin{enumerate}
\item[\rm 1.] If {\rm (1a)} or {\rm (1b)} of Proposition~\ref{prop::asymptotics1} hold then for all $\vep>0$ along a.e. sample path of $(N^{(p)})_{p>0}$, 
as $\nto$
and $p=p_n\to\infty$,
\beao
\P\big(\big|T_{n,\beta}(X^{(p)},Y^{(p)})\big|>\vep\mid N^{(p)}\big)\to 0\,.
\eeao
\item[\rm 2.] If $\big[${\rm (1a)} and {\rm (2a)} both for $X,Y\big]$ or 
$\big[${\rm (1b)} and {\rm (2b)} both for $X,Y\big]$ of Proposition~\ref{prop::asymptotics1} hold then we also have for all $\vep>0$
along a.e. sample path of $(N^{(p)})_{p>0}$, as $\nto$
and $p=p_n\to\infty$,
\beao
\P\big(\big|R_{n,\beta}(X^{(p)},Y^{(p)})\big|>\vep\mid N^{(p)}\big)\to 0\,.
\eeao
\end{enumerate}
\ethe
\begin{proof} {\bf 1.} From Proposition~\ref{prop::asymptotics1}, 
in particular from \eqref{app1distacecov1}, we conclude
that $T_{n,\beta}(X^{(p)},Y^{(p)})-T_{n,\beta}(X,Y)\to 0$ as $\nto$ in \pro y
conditional on $N^{(p)}$.
The rest is analogous
to the derivations in Dehling et al. \cite{dehling:matsui:mikosch:samorodnitsky:tafakori:2019}: $T_{n,\beta}(X,Y)$
is a $V$-statistic and satisfies the \slln\ $T_{n,\beta}(X,Y)\stas T_\beta(X,Y)=0$
under the required moment conditions on $X,Y$.\\ 
{\bf 2.} 
 Under the additional conditions (2a), (2b) both for $X,Y$ we 
also have $T_{n,\beta}(X^{(p)},X^{(p)})-
T_{n,\beta}(X,X)\to 0$ and $T_{n,\beta}(Y^{(p)},Y^{(p)})-
T_{n,\beta}(Y,Y)\to 0$ as $\nto$ in \pro y conditional on $N^{(p)}$.
Moreover, by the \slln\ 
$T_{n,\beta}(X,X)\stas T_\beta(X,X)$ and $T_{n,\beta}(Y,Y)\stas T_\beta(Y,Y)$, 
hence $R_{n,\beta}(X,Y)\to R_\beta(X,Y)=0$ in \pro y conditional on $N^{(p)}$.
\end{proof}
In Section~\ref{asymptotic} we provide much stronger \asy\ results
 under stronger conditions. In particular,  we show that $n\,T_{n,\beta}(X^{(p)},Y^{(p)})$ conditional on $(N^{(p)})_{p>0}$ 
has the same weak limit
as $n\,T_{n,\beta}(X^{(p)},Y^{(p)})$, and we also show consistency of a suitable
bootstrap procedure.

\section{Random field at a lattice}\label{sec:lattice}\setcounter{equation}{0}
In the present and next sections we consider sampling schemes 
of $(X,Y)$ on a  deterministic lattice on $B=[0,1]^d$. We will assume conditions on the smoothness of $(X,Y)$
from which we can derive convergence rates of $T_{n,\beta}(X^{(p)},Y^{(p)})$ 
to $T_{n,\beta}(X,Y)$.
Here we closely follow Dehling et al. \cite{dehling:matsui:mikosch:samorodnitsky:tafakori:2019} who dealt with the lattice case for $d=1$.
The following conditions and results are adaptations of those in \cite{dehling:matsui:mikosch:samorodnitsky:tafakori:2019}. 

\subsection{Technical conditions}
We use the notation and assumptions  of Section~\ref{subsec:latticecase}.
In what follows, we introduce moment and smoothness conditions on the fields $X,Y$ and require certain rates for $p=p_n\to\infty$ as $n\to\infty$. 
The conditions are separated in two parts depending on whether  $(X,Y)$
have finite second moments or not.
\par
If $X,Y$ have finite second moments we will work under the following 
conditions.\\[2mm]
{\bf Condition (A): finite second moments}
\begin{enumerate}
\item[\rm (A1)] {\em Smoothness of increments.} There exist
$\gamma_X,\gamma_Y>0$ and $c>0$ \st\ %{\blue the blue part is not needed}
\beao
\var\big(X(\bft)-X(\bfs)\big)\le c\,|\bft-\bfs|^{\gamma_X} %{\blue ???\le c p^{-\gamma_X/d}}
\qquad\mbox{and}\qquad \var\big(Y(\bft)- Y(\bfs)\big)\le
	     c\,|\bft-\bfs|^{\gamma_Y}\,.% {\blue ??? \le c p^{-\gamma_Y/d}}\,.
\eeao

\item[\rm (A2)] {\em Additional moment conditions.}
If $\beta\in (1,2)$ %$1<\beta<2$ 
we have
\beao%\label{eq:w6}
&&\max_{\bft \in [0,1]^d} \E[|X( \bft )|^{2(2\beta-1)}]+\max_{\bft \in
	     [0,1]^d} \E[|Y( \bft )|^{2(2\beta-1)}]<\infty\,.
%&&\max_{i=1,\ldots,p} \max_{t\in\Delta_i} \E[|X(t,t_i]|^{2\beta}]+\max_{i=1,\ldots,p} \max_{t\in\Delta_i} \E[|Y(t,t_i]|^{2\beta}]
%\to 0\,,\quad p\to\infty\,.\label{eq:w7}
%Then there is $c'$ \st\ for any  $n\ge 1$, 
% \beao
% \E\big[ |T_{n,\beta}(X^{(p)},X^{(p)})-T_{n,\beta} (X,X) |\big] \le c'\,
%     p^{-(\gamma_X\wedge \gamma_Y) \,(\beta\wedge 1)/2}. 
\eeao
\item[\rm (A3)] {\em Growth condition on $p=p_n\to\infty$.} We have
\beao
%\delta_n\
p^{-1}n^{2/\big( d^{-1}(\gamma_X\wedge\gamma_Y)\,(\beta\wedge 1)\big)} \to 0\,,\qquad\nto\,.
\eeao
\end{enumerate}
If $X,Y$ possibly have infinite 
second moments we will work under the following conditions:\\[2mm]
{\bf Condition (B): infinite second moments} Assume $\beta\in (0,2)$.
\begin{enumerate}
\item[\rm (B1)] {\em Finite $\beta$th moment.}
\beao%\label{eq:k}
\E\big[\max_{\bft \in (0,1]^d} |X(\bft)|^{\beta}\big]<\infty\;\;\mbox{and}\;\; 
\E\big[\max_{\bft \in (0,1]^d} |Y(\bft)|^{\beta}\big]<\infty\,. 
\eeao
\item[\rm (B2)]
{\em Smoothness of increments.}
There exist $\gamma_X,\gamma_Y>0$ and $c>0$ \st\  
\beao
\max_{\bfi\in\Pi_d}\E\big[\max_{ \bft \in \Delta_{\bfi}}|
	     X(\bft)-X(\bft_\bfi)|^{\beta}\big]\le c\, %\delta_n^{\gamma_X}\ 
p^{-\gamma_X/d}\ 
\mbox{and}\ 
\max_{\bfi\in \Pi_d}\E\big[\max_{t\in \Delta_\bfi}|
	     Y(\bft)-Y(\bft_\bfi)|^{\beta}\big]\le c\, %\delta_n^{\gamma_Y}
p^{-\gamma_Y/d}
\,.
\eeao
\item[\rm (B3)]  {\em Additional moment and smoothness conditions.}
\beao%\label{eq:k}
\E\big[\max_{\bft \in (0,1]^d} |X(\bft)|^{2\beta}\big]<\infty\;\;\mbox{and}\;\; 
\E\big[\max_{\bft \in (0,1]^d} |Y(\bft)|^{2\beta}\big]<\infty\,. 
\eeao
If $\beta\in (0,1)$ there also exist $\gamma_X',\gamma_Y'>0$ and $c>0$ \st\
\beao
\max_{\bfi\in\Pi_d}\E\big[\max_{\bft\in \Delta_\bfi}|
	     X(\bft)-X(\bft_\bfi)|^{2\beta}\big]\le c\, p^{-\gamma_X'/d}
%\delta_n^{\gamma_X'}\
	     \mbox{ and }\ 
\max_{\bfi\in\Pi_d}\E\big[\max_{\bft\in \Delta_\bfi}|
	     Y(\bft)-Y(\bft_\bfi)|^{2\beta}\big]\le c\, p^{-\gamma_Y'/d}\,.%\label{eq:may8e}
%\delta_n^{\gamma_Y'}\,.
\eeao
\item[\rm (B4)] {\em Growth condition on $p=p_n\to\infty$.} We have $\beta/2+
(\gamma_X\wedge\gamma_Y)/d>1$ and %{\blue If this condition holds the exponent
%in the next display would be negative ??? So this condition is not needed?}
\beao
p^{-1}n^{(\beta/(\beta\wedge 1))\big/(\beta/2+(\gamma_X\wedge\gamma_Y)/d-1)}\to 0\,,\quad\nto\,.
\eeao
\end{enumerate}
%\textcolor{red}{M: the condition is obtained by that $n\times$ the righthand side of \eqref{eq:may10b} should go to zero, so above is correct one.}
%\textcolor{red}{M: the condition (B4) with $\beta\in(0,1)$ is very severe, 
%if $d\ge 2$ we could not use this condition. If possible it is better to have bound for 
%$\E\big[\max_{\bfi\in\Pi_d}\max_{\bft\in \Delta_\bfi}|
%	     X(\bft)-X(\bft_\bfi)|^{2\beta}\big]$. } {\blue isn't the condition incorrect? We need that  $\beta/2+
%(\gamma_X\wedge\gamma_Y)/d<1$ ? Otherwise, B4 is always true?}
%In order to obtain \eqref{eq;kl}, we need 
%The proof of the theorem is based on the following proposition wherein 
%we derive the rate at which 
% $\E[|T_{n,\beta}(X^{(p)},Y^{(p)})-T_{n,\beta} (X,Y)|]\to 0$ and show 
% $T_{n,\beta}(X^{(p)},X^{(p)})-T_{n,\beta} (X,Y)\stp0$. 
In the lattice case the following result is the analog of Proposition~\ref{prop::asymptotics1}.
\bpr\label{prop:1}
Assume the following conditions.
\begin{enumerate}
\item [\rm (1)]
$X,Y$ are independent stochastically continuous bounded processes on $[0,1]^d$ defined
on the same \pro y space. 
\item[\rm (2)]
If  $X,Y$ have finite expectations, then these are assumed to be equal
to $0$.
\item[\rm (3)] $\beta\in (0,2)$.
\item[\rm (4)] $p=p_n\to\infty$ as $\nto$.
\end{enumerate} 
Then the  following statements hold.
\begin{enumerate}
\item[\rm 1.]
If also {\rm  (A1)} holds then there is a constant $c$ \st\ for any  $n\ge 1$, 
 \beam\label{eq:may10a}
 \E\big[ |T_{n,\beta}(X^{(p)},Y^{(p)})-T_{n,\beta} (X,Y) |\big] \le c\,
     p^{-d^{-1}(\gamma_X\wedge \gamma_Y) \,(\beta\wedge 1)/2}\,. 
\eeam
\item[\rm 2.]
If also {\rm (B1),(B2)} hold then there is a constant $c$ \st\ for any  $n\ge 1$, %{\blue something is wrong with the blue exponent, it's a contradiction  to B4}
\beam\label{eq:may10b}
\E\big[ |T_{n,\beta}(X^{(p)},Y^{(p)})-T_{n,\beta} (X,Y) |\big] \le c\, \big(
p^{1-(\beta/2+(\gamma_X\wedge \gamma_Y)/d)}\big)^{(\beta\wedge 1)/\beta}\,.
\eeam
\item[\rm 3.] If also {\rm (A1), (A2)} hold  
in the finite variance case and $\big[\mbox{$\beta\in
    (0,1)$, {\rm (B3)}}\ and\  1<\beta+(\gamma_X' \wedge
	     \gamma_Y')/d \big]$ in the infinite variance
	     case, then 
\beao
T_{n,\beta}(X^{(p)},X^{(p)})-T_{n,\beta}(X,X) \stp 0\,,\qquad \nto\,,
\eeao
and the analogous result holds for $Y$. 
\end{enumerate}
%\textcolor{red}{M: Notice that the cases 1 and 2 are basically the same as Proposition 7.1 of \citep{dehling:matsui:mikosch:samorodnitsky:tafakori:2019} and, 
%the case 3. is the same as Lemma S.1 in Section S (Supplement) of \citep{supplement}. This secures our results.}
\epr
The proof of this result is completely analogous to the case $d=1$ given in
\cite{dehling:matsui:mikosch:samorodnitsky:tafakori:2019} and therefore omitted.
\par
A comparison of Propositions~\ref{prop::asymptotics1} and   \ref{prop:1} shows 
that in the former result one does not need conditions on the smoothness
of the sample paths such as (A1), (B2) or (B3). In the latter result, the 
smoothness parameters $\gamma_X,\gamma_Y$ appear explicitly in the 
approximation rates in \eqref{eq:may10a} and \eqref{eq:may10b}.
For example, consider two iid Brownian sheets $X,Y$. According to Example~\ref{exam:BH} below, 
$\gamma_X=\gamma_Y= 1$ since $H_i=0.5$. Then the \rhs\ of \eqref{eq:may10a}
turn into $c\, p^{-(\beta\wedge 1)/(2d)}$. In Proposition~\ref{prop::asymptotics1}
the parameter $p$ has a distinct meaning. Assume that $(N_p)_{p>0}$ are Poisson
variables, e.g. $N_p=M(p\,B)$ for a unit rate homogeneous Poisson process 
$M$ on $\bbr^d$. When restricted to $B$, $M(p\,\cdot)$ has intensity $p|B|$ 
and we have $N_p/p\stp |B|$ and  therefore $p|B|$ 
is a rough approximation of $N_p$.
We observe that \eqref{eq:may10a} is the analog of 
\eqref{app1distacecov1}. In the latter case, we do not have a \con\ rate.
This is perhaps not surprising since
knowledge of $X,Y$ on the grid points of a lattice in combination 
with the smoothness of the sample paths of $X,Y$ often give us additional
information about the sample path inside a cell $\Delta_\bfi$.
In Section~\ref{sec:onlyrandomlocation} %of \citep{supplement:2} 
we will moderate between
the two sampling schemes: we will average the observations at random locations
in each cell of the lattice. This property allows one to use smoothness
properties of the random field also for observations at random locations.
\par
Thus, if the observations are given on a lattice it is preferable to
use distance covariance on it. However, the random location case has the 
advantage that smoothness does not matter and it works on quite general
bounded sets $B$.

\bexam\label{exam:BH}\rm Let $X(\bfu)=B^H(\bfu)$, $\bfu\in [0,1]^d$,
be a fractional Brownian sheet with Hurst parameter $H=(H_1,\ldots,H_d)\in
 (0,1)^d$; see Example~\ref{exam:fbm}. From 
\eqref{eq:fbma} we have 
\begin{align*}
& \var(B^H(\bft)-B^H(\bfs)) \\
&= \Big(\prod_{i=1}^d |t_i|^{2H_i} -\prod_{i=1}^d 
0.5 (|t_i|^{2H_i}+|s_i|^{2H_i}-|t_i-s_i|^{2H_i})\Big)\\
&\quad + \Big(\prod_{i=1}^d |s_i|^{2H_i} -\prod_{i=1}^d 
0.5 (|t_i|^{2H_i}+|s_i|^{2H_i}-|t_i-s_i|^{2H_i})\Big)\\
% &= \big( \prod_{i=1}^d (a_i+b_i) -\prod_{i=1}^d b_i \big) +  \big( \prod_{i=1}^d (c_i+b_i) -\prod_{i=1}^d b_i \big), 
%\
&= \sum_{i=1}^d \Big(\prod_{j=1}^{i-1} |t_j|^{2H_j}\,\dfrac 1 2 \big(|t_i|^{2H_i}-|s_i|^{2H_i} + |t_i-s_i|^{2H_i}\big) \\
&\quad + \prod_{j=1}^{i-1} |s_j|^{2H_j}\,\dfrac 1 2 \big(|s_i|^{2H_i}-|t_i|^{2H_i} + |t_i-s_i|^{2H_i}\big)\Big)
\prod_{k=i+1}^d \frac{1}{2} (|t_k|^{2H_k}+|s_k|^{2H_k}-|t_k-s_k|^{2H_k}).
\end{align*} 
%where
%\[
% a_i=|t_i|^{2H_i} -b_i,\,b_i= 1/2(|t_i|^{2H_i}+|s_i|^{2H_i}-|t_i-s_i|^{2H_i}),\%,c_i=|s_i|^{2H_i}-b_i.
%\]
%Thus $\var(B^H(\bft)-B^H(\bfs))$ is expressed with polynomials such that all terms include $a_i$ or $c_i,\,i=1,\ldots,d$. 
%Since $a_i$ and $c_i$ are bounded with 
%\[
% c \{|t_i|^{2H_i}-|s_i|^{2H_i}+|t_i-s_i|^{2H_i}\},\,i=1,\ldots,d,\,t_i,s_i \in [0,1],
%\]
Thus we have for some constant $c>0$,
\[
 \var(B^H(\bft)-B^H(\bfs)) \le c \sum_{i=1}^d|t_i-s_i|^{(2H_i) \wedge 1} \le c |\bft-\bfs|^{(2\min_i H_i)\wedge 1}. 
\] 
%It has stationary increments and
%therefore it follows from \eqref{eq:fbma} that {\red check}
%\beao
%\var (B^H (\bft)-B^H(\bfs))=\var (B^H (\bft-\bfs)) =
%\prod_{i=1}^d|s_i-t_i|^{2H_i}\le 
%|\bft-\bfs|^{2\,(H_1+\cdots +H_d)}\,.
%\eeao 
%Hence (A1) is satisfied with 
%$\gamma_X=2\, (H_1+\cdots +H_d)$. Moreover, for any $\gamma>0$, 
%$\E[|B^H(\bfu)|^\gamma]$ is continuous. Hence (A2) holds.
Hence (A1) is satisfied with 
$\gamma_X= (2\min_{i=1,\ldots,d} H_i)\wedge 1$. Moreover, for any $\gamma>0$, 
$\E[|B^H(\bfu)|^\gamma]$ is continuous. Hence (A2) holds.
\eexam
\bexam\rm Let $X^H(\bfu)$, $\bfu\in[0,1]^d$, for some $H\in (0,1)$ 
be a \levy\ fractional
Brownian field; see Example~\ref{exam:levybm}. It has all moments and 
stationary increments. Hence
\beao
\var(X^H(\bft)-X^H(\bfs))= |\bft-\bfs|^{2H}\,.
\eeao
Hence (A1) is satisfied with $\gamma_X=2H$. Moreover, for any $\gamma>0$, 
$\E[|X^H(\bfu)|^\gamma]$ is continuous. Hence (A2) holds.
\eexam
\bexam\label{exam:levyb}\rm 
Let $X$ be a \levy\ sheet on $[0,1]^d$; see Example~\ref{exam:levy}. 
Recall that $X(\bft)=X([\bf0,\bft])$. Therefore %{\blue this is not an increment, an increment would be $X(\bfs,\bft]$ for $\bfs\le \bft$.}
\begin{align}
\label{increment:levy}
 X(\bft)-X(\bfs) = X([{\bf0},\bft]\setminus ([{\bf0},\bfs]\cap [{\bf0},\bft]))- X([\bf0,\bfs]\setminus ([\bf0,\bfs]\cap [\bf0,\bft])), 
%\sum_{i=1}^d X(s_1,\ldots,s_{i-1},t_i,\ldots,t_d)-X(s_1,\ldots,s_i,t_{i+1},\ldots,t_d),
\end{align}
where the two random variables on the \rhs\ are defined on disjoint sets, hence they are independent.
We observe that the set $\big([\bf0,\bfs]\cup  [\bf0,\bft]\big)\backslash
\big([\bf0,\bfs]\cap  [\bf0,\bft]\big)$ for $\bfs,\bft\in [0,1]^d$ 
is the union of a finite number (only depending on $d$) 
of disjoint sets $(\bfa,\bfb]$ \st\ for some $1\le j\le d$ and a 
constant $c>0$, $|(\bfa,\bfb]\big|\le c\,|t_j-s_j|$. Thus,
if $X$ has finite second moment, we have for some constants $c_1,c_2$,
\beao
\var(X(\bft)-X(\bfs))&\le & c_1 \max_{i=1,\ldots,d}|t_i-s_i| \le c_2 |\bft-\bfs|\,.
\eeao
Hence (A1) holds with $\gamma_X=1$. 
If the moment condition $\E[|X( \bfu )|^{2(2\beta-1)}]<\infty$ holds for some $\bfu$, 
it is finite for all $\bfu$ and also continuous. Hence
(A2) holds for $X$.
\eexam
\bexam\rm
Consider a symmetric $\alpha$-stable \levy\ sheet on $[0,1]^d$
for some $\alpha\in (0,2)$; see Example~\ref{exam:stab}. 
It has moments of order $\beta\in (0,\alpha)$. 
%We again rely on the expression \eqref{increment:levy}, which yields
Since $X$ is symmetric and has independent increments an application of
Levy's maximal inequality yields
\beao
\P\Big(\max_{\bft \in \Delta_{\bfi}}|X(\bft)-X(\bft_{\bfi})|>x\Big)\le 2\,
\P\big(|X(\bft_{\bfi})-X(\bft_{\bfi-\bf1})|>x\big)\,,\qquad x>0\,.
\eeao
% \max_{\bft \in \Delta_{\bfi}}|X(\bft)-X(\bft_{\bfi})|^\beta &= \max_{\bft \in \Delta_{\bfi}} \big|
%\sum_{j=1}^d X(t_1,\ldots,t_j,t_{i_{j+1}},\ldots,t_{i_d})-X(t_1,\ldots,t_{j-1},t_{i_j},\ldots,t_{i_d}) \big|^\beta \\
%&\le \max_{\bft \in \Delta_{\bfi}}
%\sum_{j=1}^d |X(t_1,\ldots,t_j,t_{i_{j+1}},\ldots,t_{i_d})-X(t_1,\ldots,t_{j-1},t_{i_j},\ldots,t_{i_d}) |^\beta 
%\end{align*}
%Since increments by disjoint intervals are independent, an application of a maximal inequality, stationary increments and self-similarity 
%imply
%for such $\beta$ that uniformly for $\bfi\in\Pi_d$,}
%\begin{align*}
% \E\big[\max_{ \bft \in \Delta_{\bfi}}| X(\bft)-X(\bft_\bfi)|^{\beta}\big] 
%&\le \sum_{j=1}^d \E\big[
%\max_{ \bft \in \Delta_{\bfi}} | 
%X(t_1,\ldots,t_j-t_{i_j},\ldots,t_{i_d})
%|^\beta
%\big] \\
%& \le c d \E[| X(1,\ldots,1,1/q,1,\ldots,1) |^\beta] \\
%&\le c(q^{-1})^{\beta/\alpha} \le c p^{-\beta/(\alpha d)}. 
%\end{align*}
%An application of a maximal inequality, stationary increments and self-similarity 
%imply
%for such $\beta$ that uniformly for $\bfi\in\Pi_d$,
%\beao
%\E\big[\max_{ \bft \in \Delta_{\bfi}}|
%	     X(\bft)-X(\bft_\bfi)|^{\beta}\big]
%&=& \E\big[\max_{ \bft \in \Delta_{\1}}|
%	     X(\bft)|^{\beta}\big]\\
%&\le &c\,\E\big[|X(\1/q)|^{\beta}\big]=
%c\,(q^{-d})^{\beta/\alpha}= c\,p^{-\beta/\alpha}\,.
%\eeao
By definition of a symmetric  $\alpha$-stable sheet
\beao
\E\big[|X(\bft_{\bfi})-X(\bft_{\bfi-\bf1})|^\beta\big]
&=& c\,\big|[{\bf0},\bft_{\bfi}]\backslash [{\bf0},\bft_{\bfi-\bf1}]\big|^{\beta/\alpha}
\le c\,p^{-\beta/(\alpha\,d)}\,.
\eeao
Thus (B2) holds with $\gamma_X=\beta/\alpha$.
The same
argument yields (B3) %\eqref{eq:may8e}
with $\gamma_X'=2\beta/\alpha $
for $\beta\in (0,0.5\alpha)$.
\eexam
\subsection{Main result in the lattice case}
Now we proceed with the main theorem. In the case $d=1$ it  
corresponds to Theorem 3.1 in Dehling et al. \cite{dehling:matsui:mikosch:samorodnitsky:tafakori:2019}.
\bth\label{thm:1} Assume the following conditions:
\begin{enumerate}
\item [\rm (1)]
$X,Y$ are independent stochastically continuous bounded processes on $[0,1]^d$. 
\item[\rm (2)]
If  $X,Y$ have finite expectations, then they are centered.
%\item[\rm ] $\delta_n\to0$ as $\nto$.
\item[\rm (3)] $\beta\in (0,2)$.
\item[\rm (4)] $p=p_n\to\infty$ as $\nto$.
\end{enumerate} 
Then the  following statements hold.
\begin{enumerate}
\item[\rm 1.]
If either {\rm (A1)}  or $\big[\mbox{{\rm (B1),(B2)} and 
$1<\beta/2+ (\gamma_X\wedge \gamma_Y)/d$} \big]$ are
satisfied then 
%\[
% T_{n,\beta}(X^{(p)},Y^{(p)})-T_{n,\beta} (X,Y) \stackrel{p}{\to} 0
%\]
%and hence 
\[
 T_{n,\beta}(X^{(p)},Y^{(p)}) \stp 0
\]
holds. 
%\eqref{eq:pl} $($and, hence, \eqref{eq:slln}$)$ hold. 
\item[\rm 2.] If either {\rm (A1),(A3)} or {\rm (B1),(B2),(B4)} hold
then 
\beam\label{eq:may8fa}
n\,T_{n,\beta}(X^{(p)},Y^{(p)}) \std \sum_{i=1}^\infty \la_i (N_i^2-1)+ c,  
\eeam
for an iid \seq\ of standard normal \rv s $(N_i)$, a constant $c$, 
and a square summable
\seq\ $(\la_i)$. 
\item[\rm 3.] If either {\rm (A1),(A2)} or $\big[\mbox{$\beta\in (0,1)$, {\rm (B1)-(B3)}}\ and\ 1< \beta+
(\gamma_X'\wedge \gamma_Y')/d \big]$ hold then 
\beao
R_{n,\beta}(X^{(p)},Y^{(p)})\stp 0\,.
\eeao
\item[\rm 4.] If either {\rm (A1)-(A3)} or $\big[\mbox{$\beta\in
    (0,1)$, {\rm (B1)-(B4)}}\ and\ 1<\beta+(\gamma_X' \wedge
	     \gamma_Y')/d \big]$ hold then 
$n\,R_{n,\beta}(X^{(p)},Y^{(p)})$ converges to a scaled version of the limit
in \eqref{eq:may8fa}.
\end{enumerate}

\ethe
\begin{proof}
{\bf 1.}
Under the assumptions, 
%{\rm (A1)} or $\big[$(B1),(B2) and 
%$1<\beta/2+\textcolor{red}{\gamma_X\wedge \gamma_Y/d}$$\big]$, 
Proposition \ref{prop:1}, yields that 
$T_{n,\beta}(X,Y)-T_{n,\beta}(X^{(p)},Y^{(p)})$\\$\stp 0$ while 
$T_{n,\beta}(X,Y)\stas T_\beta(X,Y)=0$ by the \slln\ for $V$-statistics.
The statement follows.\\
{\bf 2.} 
 Under the additional conditions, %{\rm (A1),(A3)} or {\rm (B1),(B2),(B4)},
 where $1<\beta/2+ (\gamma_X\wedge \gamma_Y)/d$ is implied by $\rm (B4)$, we
 have from Proposition \ref{prop:1}, that  
 $n|T_{n,\beta}(X,Y)-T_{n,\beta}(X^{(p)},Y^{(p)})|
 \stp {0}$. By degeneracy of the $V$-statistics
 $nT_{n,\beta}(X,Y)$ converges in distribution to a sum of independent weighted
 $\chi^2$ random variables, so does $nT_{n,\beta}(X^{(p)},Y^{(p)})$.\\
{\bf 3. {\rm and} 4.} 
 They follow by combining (1) and (2) with Proposition \ref{prop:1},
which allows one to switch from $T_{n,\beta}$ to $R_{n,\beta}$.
We omit further  details. 
\end{proof}

\section{Random field at random locations grouped in lattice cells}\label{sec:onlyrandomlocation}\setcounter{equation}{0}
In this section we consider a combination of the two previous sampling schemes for 
$(X,Y)$. For the sake of argument we restrict ourselves to fields on
$B=[0,1]^d$. We assume that the fields are observed at the locations
$(\bfU_i)_{i=1,\ldots,p}$ which are iid uniformly distributed on $B$ and 
$p=p_n\to\infty$ is a deterministic integer \seq . 
This time we average the randomly scattered $\bfU_i$ in each cell of a regular lattice grid. 
\par
Similarly to Section \ref{subsec:latticecase} 
%\cite[Subsection 1.2]{mmrt:main} %\ref{sec:lattice} 
we partition $B$ into $\tilde p=[p/\log p]$ disjoint cells $\Delta_i,i=1,\ldots,\tilde p$ with 
side length $\tilde p^{-1/d}$ and volume $\tilde p^{-1}$. 
The main idea of this approach is to average the observations 
$(X(\bfU_j),Y(\bfU_j))$ in a given cell. Since $p$ increases slightly 
faster than the number of cells $\tilde p$ the probability that 
there is no observation in a given cell decreases at a certain rate.
%see Lemma~\ref{lem:uniform}.
\par
For any random field $Z$ on $B$ we consider the discretization (again abusing notation),
\beam\label{eq:dsicret}
 Z^{(p)}(\bfu)% \sum_{j=1}^{\tilde p} \1(\bfu \in \Delta_j,\,
% \# (i:\bfU_i \in \Delta_j)\ge 1)\, \sum_{k:\Gamma_k\in \Delta_j}
% \frac{Z(\bfU_k)}{ \# (i:\bfU_i \in \Delta_j) } \\
= \mbox{$\sum'$}\1_{\Delta_j}(\bfu)
 \ov Z_j, 
\eeam
where $\sum'$ denotes summation over those $j=1,\ldots,\tilde p$
\st\ $\#\Delta_j=\#\{i\le p:\bfU_i \in \Delta_j\}\ne 0$ and 
\beao
\ov Z_j=\sum_{k:\bfU_k\in \Delta_j}
\dfrac{Z(\bfU_k)}{ \# \Delta_j }\,.
\eeao
We also recycle the notation
\beam\label{def:normapp2}
 \|Z^{(p)}\|_2^2 &:=& 
\int_{[0,1]^d}
\mbox{$\sum'$} \1_{\Delta_j}(\bfu)\, \ov Z_j^2 \,d\bfu 
= {\tilde p}^{-1}\mbox{$\sum'$} \ov Z_j^2\, . 
\eeam
\subsection{Technical conditions}
The results in this section parallel those in the lattice case \ref{sec:lattice}. 
%\citep[Section 3]{mmrt:main}.
The conditions are similar to {\bf (A)} and {\bf (B)}; we also 
use the notation from  {\bf (A)} and {\bf (B)} in the sequel.
\par
If $X,Y$ have finite second moments we will need the following 
condition.
\begin{enumerate}
\item[\rm (A3')] {\em Growth condition on $p=p_n\to\infty$.} 
\beao
(p^{-1}\vee \tilde p^{-(\gamma_X \wedge \gamma_Y)/d} )n^{2/(\beta\wedge 1)} \to 0\,,\qquad\nto\,.
\eeao
\end{enumerate}
If $X,Y$ possibly have infinite 
second moments we will need the following condition.\\ Assume $\beta\in (0,2)$\
and
\begin{enumerate}
\item[\rm (B4')] {\em Growth condition on $p=p_n\to\infty$.} 
We have 
\beao
(p^{-1}\vee \tilde p^{-(\gamma_X \wedge \gamma_Y)/d}) p^{1-\beta/2} n^{\beta/(1\wedge \beta)}\to 0\,,\quad\nto\,.
\eeao
\end{enumerate}
If we remove $p^{-1}$ in (A3') and (B4') and 
replace $\tilde p$ by $p$ then we 
recover {\rm (A3)} and {\rm (B4)}. 
\par
The following result is an analog of Proposition 2.1 %in \citep[Section 2]{mmrt:main} 
for the discretizations $X^{(p)},Y^{(p)}$ of $X,Y$  defined via \eqref{eq:dsicret}.
In contrast to the latter case, we have explicit rates of \con\
for $\E\big[|T_{n,\beta}(X^{(p)},Y^{(p)})-T_{n,\beta} (X,Y) |\big]\to 0$.
These rates are achieved due to the averaging of values in the cells
$\Delta_j$. Then one can also exploit the smoothness of the field over
these small cells.
\bpr\label{prop:2}
Assume the following conditions.
\begin{enumerate}
\item [\rm (1)]
$X,Y$ are independent stochastically continuous bounded processes on $[0,1]^d$ defined
on the same \pro y space. 
\item[\rm (2)]
If  $X,Y$ have finite expectations, then these are assumed to be equal
to $0$.
\item[\rm (3)] $\beta\in (0,2)$.
\end{enumerate} 
Then the  following statements hold.
\begin{enumerate}
\item[\rm 1.]
If also {\rm  (A1)} holds then there is a constant $c$ \st\ for all  $n\ge 1$, 
 \beao%\label{eq:may10a}
 \E\big[ |T_{n,\beta}(X^{(p)},Y^{(p)})-T_{n,\beta} (X,Y) |\big] \le c\,
    \big( \tilde p^{-(\gamma_X\wedge \gamma_Y)/d}+ p^{-1} \big)^{(\beta\wedge 1)/2}\,. 
\eeao
\item[\rm 2.]
If also {\rm (B1),(B2)} hold then there is a constant $c$ \st\ for all  $n\ge 1$, 
\beao%\label{eq:may10b}
\E\big[ |T_{n,\beta}(X^{(p)},Y^{(p)})-T_{n,\beta} (X,Y) |\big] \le c\, \big(
p^{1-\beta/2} ( \tilde p^{-(\gamma_X\wedge \gamma_Y)/d} +p^{-1}) \big)^{(\beta\wedge 1)/\beta}\,.
\eeao
\item[\rm 3.] Under the additional conditions {\rm (A1), (A2)} in the finite
	     variance case and under $\big[\beta\in (0,1)$, 
{\rm (B1)-(B3)} and $1<\beta+\gamma_X'/d$
	     $\big]$ in the infinite variance
	     case, 
\beao
T_{n,\beta}(X^{(p)},X^{(p)})-T_{n,\beta}(X,X) \stp 0\,,\qquad \nto\,.
\eeao

\end{enumerate}
\epr
The proof of Proposition \ref{prop:2} is rather technical and given in 
Section \ref{App:B}.
%{\color{red}The proof is given in supplementary file.}
\subsection{Main result}
The following is the main result of this section.

\bth\label{thm:random:location} Assume the following conditions:
\begin{enumerate}
\item [\rm 1.]
$X,Y$ are independent stochastically continuous bounded processes on $[0,1]^d$. 
\item[\rm 2.]
If  $X,Y$ have finite expectations, then they are centered.
\item[\rm 3.] $\beta\in (0,2)$ and $p_n\to \infty$ as $\nto$.
\end{enumerate} 
Then the  following statements hold.
\begin{enumerate}
\item[\rm (1)]
If either {\rm (A1)}  or $\big[\mbox{{\rm (B1),(B2)} and 
$1<\beta/2+(\gamma_X\wedge \gamma_Y)/d$}\big]$ are
satisfied then 
\[
 T_{n,\beta}(X^{(p)},Y^{(p)}) \stp 0
\]
hold. 
%\eqref{eq:pl} $($and, hence, \eqref{eq:slln}$)$ hold. 
\item[\rm (2)] If either {\rm (A1),(A3')} or {\rm (B1),(B2),(B4')} hold
then 
\beao\label{eq:nov26}
n\,T_{n,\beta}(X^{(p)},Y^{(p)}) \std \sum_{i=1}^\infty \la_i (N_i^2-1)+ c 
\eeao
for an iid \seq\ of standard normal \rv s $(N_i)$, a constant $c$, 
and a square summable
\seq\ $(\la_i)$. 
\item[\rm (3)] If either {\rm (A1),(A2)} or $\big[\mbox{$\beta\in (0,1)$, {\rm (B1)-(B3)}}\ and\ 1< \min(\beta+(\gamma_X'
	     \wedge \gamma_Y')/d,\beta/2+(\gamma_X\wedge\gamma_Y)/d)
\big]$ hold then 
\beao
R_{n,\beta}(X^{(p)},Y^{(p)})\stp 0\,.
\eeao
\item[\rm (4)] If either {\rm (A1),(A2),(A3')} or $\big[\mbox{$\beta\in
    (0,1)$, {\rm (B1)-(B3),(B4')}}\ and\ 1<\beta+ (d^{-1}(\gamma_X' \wedge
	     \gamma_Y'))\wedge 1)\big]$ hold then 
$n\,R_{n,\beta}(X^{(p)},Y^{(p)})$ converges to a scaled version of the limit
in \eqref{eq:nov26}.
\end{enumerate}
\ethe
One can follow the lines of the proof of Theorem \ref{thm:1}.

\section{The bootstrap for $T_{n,\beta}$}\label{asymptotic}\setcounter{equation}{0}
In this section we introduce a bootstrap procedure for $T_{n,\beta}(X,Y)$
and $T_{n,\beta}(X^{(p)},Y^{(p)})$.
\subsection{$T_{n,\beta}(X,Y)$ as a degenerate $V$-statistic}\label{subsec:appa}
We recall some facts from the Appendix in Dehling et al. \cite{dehling:matsui:mikosch:samorodnitsky:tafakori:2019} (see also Lyons \cite{lyons:2013}) and adapt them to the situation of a random field
on $B$.
We assume that $Z_i=(X_i,Y_i)$, $i=1,2,\ldots,$ is an iid \seq\ with generic element $Z=(X,Y)$ whose components are bounded
Riemann square-integrable random fields on $B$, and 
$\E[\|X\|_2^\beta+\|Y\|_2^\beta + \|X\|_2^\beta\|Y\|_2^\beta]<\infty$ for some $\beta\in (0,2)$.  Under these assumptions, $T_{n,\beta}(X,Y)$ 
has \rep\ as a $V$-statistic of order 4 with symmetric degenerate kernel 
of order 1.
\par
We start with the kernel
\begin{align*}
& f((x_1,y_1),(x_2,y_2),(x_3,y_3),(x_4,y_4))\ (=: f(z_1,z_2,z_3,z_4)) \\
& \quad = \|x_1-x_2\|_2^\beta\|y_1-y_2\|_2^\beta 
  +\|x_1-x_2 \|_2^\beta \|y_3-y_4 \|_2^\beta  -2\|x_1-x_2\|_2^\beta \|y_1-y_3\|_2^\beta.
\end{align*}
From this representation, 
\[
  T_{n,\beta}(X,Y) =\frac{1}{n^4} \sum_{1\leq i,j,k,l \leq n} f(Z_i,Z_j,Z_k,Z_l). 
\]
Then one can define the corresponding symmetric kernel via 
the usual symmetrization as
\begin{equation} \label{e:h4}
 h(z_1,z_2,z_3,z_4)=\frac{1}{24} \sum_{ (l_1,l_2,l_3,l_4) \mbox{ permutation of  }(1,2,3,4) }
 f(z_{l_1},z_{l_2},z_{l_3},z_{l_4}).
\end{equation}
The kernel $h$ is at least $1$-degenerate: under the null hypothesis of
  independence between $X$ and~$Y$,
\begin{align*}
&  \E [f(z_1,Z_2,Z_3,Z_4)] + \E [f(Z_2,z_1,Z_3,Z_4)] +  \E
 [f(Z_2,Z_3,z_1,Z_4)] \\
&\qquad + 
\E [f(Z_2,Z_3,Z_4,z_1)]=0\,. 
\end{align*}
Still under the null  hypothesis of   independence between $X$ and $Y$, 
\beam
\label{condiexp_h_2}
\lefteqn{\E [h(z_1,z_2,(X_3,Y_3),(X_4,Y_4))]}\nonumber\\
 &=& \frac{1}{6} \big( \| x_1-x_2\|_2^\beta + \E[ \|
X_1-X_2\|_2^\beta] - \E [\| x_1-X\|_2^\beta] -  \E[ \| x_2-X\|_2^\beta]\big)\nonumber \\
&&\ \ \times\bigl( \| y_1-y_2\|_2^\beta + \E [\|
Y_1-Y_2\|_2^\beta] - \E[ \| y_1-Y\|_2^\beta] -  \E[ \|
y_2-Y\|_2^\beta]\big)\,, 
\eeam
and the \rhs\ is not constant. Hence, the kernel $h$ is
precisely $1$-degenerate. One can follow the lines of the proof in 
the Appendix  in \cite{dehling:matsui:mikosch:samorodnitsky:tafakori:2019} to show the following result (the only necessary change is the replacement of $[0,1]$ by $B$).
\ble\label{eq:may13}
If $X,Y$ are independent and $\E[\|X\|_2^\beta+\|Y\|_2^\beta]<\infty$  for some $\beta \in (0,2)$ then $T_{n,\beta}(X,Y)$ has \rep\ as a
$V$-statistic of order $4$ with symmetric $1$-degenerate kernel $h$. %of order $1$-degenerate}.
The corresponding $U$-statistic $\wt T_{n,\beta}(X,Y)$ 
is obtained from  $T_{n,\beta}(X,Y)$ by restricting the summation to indices 
$(i_1,i_2,i_3,i_4)$ with mutually distinct components. Then
\beao
\qquad n\,\big(T_{n,\beta}(X,Y)-\wt T_{n,\beta}(X,Y)\big)\stp
\E[\|X_1-X_2\|_2^\beta]\E[\|Y_1-Y_2\|_2^\beta] \,,\qquad \nto\,.
\eeao 
\ele
An immediate con\seq\ of this result is that, up to an additive 
constant, $(n\,T_{n,\beta}(X,Y))$ and $(n\,\wt T_{n,\beta}(X,Y))$
have the same limit \ds\ which is indicated in \eqref{eq:august20a}.
In what follows, we will focus on $\wt T_{n,\beta}(X,Y)$.
\subsection{Bootstrapping $T_n(X,Y)$}
For the sake illustration of the method we restrict ourselves to
$\beta=1$ and suppress $\beta$ in the notation. We also assume 
that 
\beam\label{eq:august20b}
\int_B\E[X^2(\bfu)+Y^2(\bfu)]d\bfu<\infty\,.
\eeam
 A generic element $Z=(X,Y)$ has trajectory $z=(x,y)$ on $B\subset \bbr^d$.
\par
Under the assumptions of Lemma~\ref{eq:may13}, 
$\wt T_{n}(X,Y)$ has \rep\  
as $U$-statistic of order 4 with a 1-degenerate symmetric 
kernel $h(x_1,x_2,x_3,x_4)$. 
Applying the Hoeffding decomposition to $\wt T_{n}(X,Y)$, the limit \ds s
(modulo a change of location/scale)  
of $nT_{n}(X,Y)$ and the following normalized version of $\wt T_n(X,Y)$
coincide:
\beao
U_n(Z)= \dfrac 1 {n}\sum_{1\le i\ne j\le n} h_2(Z_i,Z_j;F_Z)
\eeao
where $F_Z=F_X\times F_Y$ and $h_2$ is defined by 
\begin{align*}
h_2(z_1,z_2;F_Z)=&\E[h(z_1,z_2,Z_3,Z_4)]-\E [h(z_1,Z_2,Z_3,Z_4)]\\ &-\E [h(Z_1,z_2,Z_3,Z_4)]+ \E[h(Z_1,Z_2,Z_3,Z_4)]\,.
\end{align*}
%\beao
%h_2(z_1,z_2;F_Z)&=&\E[h(z_1,z_2,Z_3,Z_4)]-\E [h(z_1,Z_2,Z_3,Z_4)]\nonumber\\
%&& -\E [h(Z_1,z_2,Z_3,Z_4)]+ \E[h(Z_1,Z_2,Z_3,Z_4)]\,.
%\eeao
We write $F_n$  for the empirical \ds\ of the sample $Z_1,\ldots,Z_n$ and 
$(Z_{n1}^\ast,\ldots,Z_{nn}^\ast)$ for a bootstrap sample, i.e.,
given $Z_1,\ldots,Z_n$ it is iid with \ds\ $F_n$.
Arcones and Gin\'e \cite{arcones:gine:1992} proved that the correct
bootstrap version of $U_n(Z)$ is   
\beao
U_n(Z^\ast)= \dfrac 1 {n}\sum_{1\le i\ne j\le n} h_2(Z_{ni}^\ast,Z_{nj}^\ast;F_{n})\,.
\eeao
The fact that the limiting distributions of $U_n(Z)$
and $U_n(Z^\ast)$ coincide follows from Dehling and Mikosch \cite{dehling:mikosch:1994}; see also Corollary 5.2 in \cite{dehling:matsui:mikosch:samorodnitsky:tafakori:2019} in the case  $B=[0,1]$.
\bpr\label{cor:dehlmik}
Under the aforementioned conditions, and if also 
$\E[|h(Z_{i_1},\ldots,Z_{i_4})|^2]<\infty$ for all indices $1\leq i_1 \leq \cdots \leq i_4\leq 4$,
we have 
\beao
d_2\big({\mathcal L}\big(U_n(Z) \big)\;,
{\mathcal L}\big(U_n(Z^\ast)
\big) \big)\to 0\,,
\qquad \nto\,,\eeao
for almost all realizations of $(Z_i)$. Here $d_2$ denotes the Wasserstein
distance of order $2$.
\epr
The additional second moment assumption on $h$ 
is satisfied for our kernel. Note that it suffices to consider the non-symmetric kernel $f$, and to show that
$ \E [(f(Z_{i_1},Z_{i_2},Z_{i_3},Z_{i_4}))^2]<\infty$ 
for all indices $1\leq i_1,\ldots,i_4\leq 4$. For our specific kernel, this condition reads as 
\beao
 \E \big[\big(\|X_{i_1}-X_{i_2}\| \big[ \|Y_{i_1}-Y_{i_2}\| +
\| Y_{i_3}-Y_{i_4} \|
 -2 \|Y_{i_1}-Y_{i_3}\|  \big] \big)^2\big] <\infty,
\eeao
and this holds under the moment conditions in this paper.
\par
The Wasserstein distance $d_2$ metrizes weak \con\ and moment \con\
up to the second order. Therefore Proposition~\ref{cor:dehlmik} proves bootstrap consistency
for $U_n(Z^\ast)$. In particular, $U_n(Z^\ast)$ and $n\,T_{n}(X,Y)$ have the
same limit \ds\ up to some change of scale/location. 
From now, we will work with
$U_n(Z)$ as a surrogate of the normalized sample distance covariance $n\,T_n(X,Y)$ 
and  with its bootstrap version $U_n(Z^\ast)$.
\subsection{Bootstrapping $U_n(Z^{(p)})$}\label{sec:bootstr}
Our goal is to show
that we are allowed to replace $Z=(X,Y)$  in $U_n(Z)$
by the corresponding 
discretizations $Z^{(p)}=(X^{(p)},Y^{(p)})$ as well as the
corresponding result for $U_n(Z^\ast)$. We restrict ourselves to the case
of random locations studied in Section~\ref{sec:randomlocation}. The 
corresponding bootstrap consistency results 
in the lattice case of Section~\ref{sec:lattice} follow by a straightforward
adaptation of the results in Dehling et al. \cite{dehling:matsui:mikosch:samorodnitsky:tafakori:2019} who considered the case $d=1$ and $B=[0,1]$.
\par
We start by showing that $U_n(Z)$ and $U_n(Z^{(p)})$ are close conditionally on $N^{(p)}$. 
\ble\label{eq:additonal}
Assume that $X,Y$ are independent, \eqref{eq:august20b} holds, and $N_p\to\infty$ \as
as $p=p_n\to\infty$.
Then
\beam
\label{squaredist:Uzzp}
\E[(U_n(Z)-U_n(Z^{(p)}))^2\mid N^{(p)}]\to 0\,,\qquad \nto\,,
\eeam
%If also $\sum_{n=1}^\infty p_n^{-1/4}<\infty$ 
%then $|U_n(Z^\ast)-U_n(Z^{\ast (p)})|\to 0$ 
for a.e. realization of $(N^{(p)})$. 
\ele

\begin{proof}
We consider
\beam\label{eq:august23a}
U_n(Z)-U_n(Z^{(p)}\mid N^{(p)}):= \frac{1}{n} \sum_{1\le i\neq j \le n}
 \big(h_2(Z_i,Z_j)-h_2(Z_i^{(p)},Z_j^{(p)}\mid N^{(p)})\big)\,,
\eeam
where 
\beao
 && h_2(z_1,z_2; F_Z^{(p)}\mid N^{(p)}) \\
 &&:= \E[h(z_1,z_2,Z_3^{(p)},Z_4^{(p)})\mid N^{(p)}]-
 \E[h(z_1,Z_2^{(p)},Z_3^{(p)},Z_4^{(p)})\mid N^{(p)}] \\
 &&\quad -\E[h(Z_1^{(p)},z_2,Z_3^{(p)},Z_4^{(p)})\mid N^{(p)}]+
 \E[h(Z_1^{(p)},Z_2^{(p)},Z_3^{(p)},Z_4^{(p)})\mid N^{(p)}]\,. 
\eeao
By construction, given $N^{(p)}$, 
$U_n(Z)- U_n(Z^{(p)}\mid N^{(p)})$ only depends on the sample
$Z_1,\ldots,Z_n$ and is a $U$-statistic of order $2$ with symmetric 
$1$-degenerate kernel.
\begin{comment}
 for any $p$. The kernel is 
\[
 \wh h(Z_1,\ldots,Z_4) = h(Z_1,\ldots,Z_4)-h(Z_1^{(p)},\ldots,Z_4^{(p)}\mid ),
\]
which is symmetric in $Z$, and moreover 
\beao
&& \E[\wh h(z_1,Z_2,Z_3,Z_4)] \\
&&= \E[h(z_1,Z_2,Z_3,Z_4)]-\E[h(z_1^{(p)},Z_2^{(p)},Z_3^{(p)},Z_4^{(p)})] \\
&&= -\E\big[\E[f(z_1^{(p)},Z_2^{(p)},Z_3^{(p)},Z_4^{(p)})\mid N]
+\E[f(Z_2^{(p)},z_1^{(p)},Z_3^{(p)},Z_4^{(p)})\mid N] \\
&&\quad 
+\E[f(Z_2^{(p)},Z_3^{(p)},z_1^{(p)},Z_4^{(p)})\mid N]
+\E[f(Z_2^{(p)},Z_3^{(p)},Z_4^{(p)},z_1^{(p)})\mid N]\big]\\
&&=0
\eeao
and 
\beao
 \E[\wh h(z_1,z_2,Z_3,Z_4)] = \E[h(z_1,z_2,Z_3,Z_4)]-\E\big[ \E [
 h(z_1^{(p)},z_2^{(p)},Z_3^{(p)},Z_4^{(p)}\mid N] \big].  
%& &= \E[h(z_1,z_2,Z_3,Z_4)] \\
%&&\quad -\frac{1}{6} \E\Big[ 
%\Big(
%\|x_1^{(p)}-x_2^{(p)}\|_2+ 
%\E[\|X_1^{(p)}-X_2^{(p)}\|_2 \mid N ] 
%- \E[\|x_1^{(p)}-X^{(p)}\|_2 \mid N ]
%- \E[\|x_2^{(p)}-X^{(p)}\|_2 \mid N ] 
%\Big) \\
%&&\qquad \times \Big(
%\|y_1^{(p)}-y_2^{(p)}\|_2+ 
%\E[\|Y_1^{(p)}-Y_2^{(p)}\|_2 \mid N ] 
%- \E[\|y_1^{(p)}-Y^{(p)}\|_2 \mid N ]
%- \E[\|y_2^{(p)}-Y^{(p)}\|_2 \mid N ] 
%\Big)
%\Big] 
\eeao
In view of \eqref{condiexp_h_2}, we observe that the right-hand side is
 not constant in most combinations $(Z_i)=(X_i,Y_i)$, and hence the
 kernel $\wh h$ is at least $1$-degenerate. }
\end{comment}
By Lemma A in Serfling \cite[p. 183]{serfling:1980}, we have
\beam\lefteqn{
 \E\big[\big(U_n(Z)-U_n(Z^{(p)}\mid N^{(p)})\big)^2\mid N^{(p)} \big]}  \nonumber\\ & \le  & c\, 
\E\big[\big(h_2(Z_1,Z_2) -h_2(Z_1^{(p)},Z_2^{(p)}\mid N^{(p)})\big)^2\mid N^{(p)}\big] \nonumber\\
&\le & c\,
 \E\big[\big(h(Z_1,Z_2,Z_3,Z_4)-h(Z_1^{(p)},Z_2^{(p)},Z_3^{(p)},Z_4^{(p)})\big)^2\mid 
N^{(p)}\big]\nonumber\\
&\le & c\E \big[
\big(
f(Z_1,Z_2,Z_3,Z_4)-f(Z_1^{(p)},Z_2^{(p)},Z_3^{(p)},Z_4^{(p)})
\big)^{2}\mid N^{(p)}\big]\nonumber \\
&\le&  c\ \E\big[
\big(\|X_1-X_2\|_2\|Y_1-Y_2\|_2
 -\|X_1^{(p)}-X_2^{(p)}\|_2\|Y_1^{(p)}-Y_2^{(p)}\|_2\big)^2 \nonumber\\
&&\qquad +(\|X_1-X_2\|_2\|Y_3-Y_4\|_2
 -\|X_1^{(p)}-X_2^{(p)}\|_2\|Y_3^{(p)}-Y_4^{(p)}\|_2\big)^2 \nonumber\\
&&\qquad +(\|X_1-X_2\|_2\|Y_1-Y_3\|_2
 -\|X_1^{(p)}-X_2^{(p)}\|_2\|Y_1^{(p)}-Y_3^{(p)}\|_2\big)^2 \mid N^{(p)}
\big]\label{eq:august23b} \,.
\eeam
Now we can proceed as in the proof of Proposition~\ref{prop::asymptotics1}
to show that the \rhs\ converges to zero as $\nto$ along a.e. sample
path of $(N^{(p)})_{p>0}$. We illustrate this for the first term on the 
\rhs . We have by independence of $(X_i)$ and $(Y_i)$,
\beam\label{eq:august23c}\lefteqn{
\E\big[
\big(\|X_1-X_2\|_2\|Y_1-Y_2\|_2
 -\|X_1^{(p)}-X_2^{(p)}\|_2\|Y_1^{(p)}-Y_2^{(p)}\|_2\big)^2\mid N^{(p)}\big]}\nonumber \\
&\le &
c\,\E\big[
\big( \|X_1-X_2\|_2-\|X_1^{(p)}-X_2^{(p)}\|_2\big)^2 \mid N^{(p)}\big]
\, \E[\|Y_1-Y_2\|_2^2]\nonumber\\
&& +c\,\E\big[\|X_1^{(p)}-X_2^{(p)}\|_2^2\mid N^{(p)}\big]\, 
\E\big[\big(\|Y_1-Y_2\|_2-\|Y_1^{(p)}-Y_2^{(p)}\|_2\big)^2\mid N^{(p)}\big]\,.
\eeam
The expectation  $\E[\|Y_1-Y_2\|_2^2]$ is finite by \eqref{eq:august20b}.
For the same reason and by Riemann-sum approximation along a.e. sample
path of $(N^{(p)})_{p>0}$,
\beao
\E\big[\|X_1^{(p)}-X_2^{(p)}\|_2^2\mid N^{(p)}\big]
&=& \dfrac 1 {N_p} \sum_{i=1}^{N_p} \E\big[(X_1-X_2)^2(\bfU_i)\mid \bfU_i\big]\\
&\to & \int_B \E\big[(X_1-X_2)^2(\bfu)\big]\,d\bfu\,.
\eeao
We have by dominated \con %: since measure $\E[\cdot | N_p]$ also is moving, is it correct to apply DCT in a usual sense?},
\beao
&& \E\big[ \big( \|X_1-X_2\|_2-\|X_1^{(p)}-X_2^{(p)}\|_2\big)^2 \mid (N^{(q)})_{q>0}\big] \\
&&\le  \E\big[\big|\int_B(X_1-X_2)^2(\bfu) d\bfu - 
\|X_1^{(p)}-X_2^{(p)}\|_2^2 \big|\mid (N^{(q)})_{q>0} \big]\\
&& \to 0\,,\qquad p\to\infty\,,
\eeao
and the corresponding result holds if we replace $(X_1,X_2)$ by
$(Y_1,Y_2)$. Hence the \rhs\ in \eqref{eq:august23c} converges to zero
along a.e. sample path of $(N^{(p)})$.
\end{proof}

Our next goal is to show that 
$U_n(Z^\ast)$ and $U_n(Z^{(p)\ast}\mid N^{(p)})$ are asymptotically close
where the latter quantity is defined as in \eqref{eq:august23a} if we 
replace the \ds\ of $Z$ by $F_n$.
We write $\var^\ast$ for the variance \wrt\ bootstrap \pro y \ms .
\par
We will need some stronger assumptions to achieve this goal.
\ble\label{lem:boot} We assume that $(N_p)$ are Poisson variables \st\
$ \E[N_p]=p=p_n$ and $N_p\to\infty$ a.s. as $\nto$, $\sum_np_n^{-1/2}<\infty$ and
\beao
\int_B \E[X^4(\bfu)+Y^4(\bfu)]\,d\bfu<\infty\,.
\eeao
Then for a.e. sample paths of $(N^{(p)})$ and $(Z_i)$,
\beao
\P\big(U_n(Z^\ast)-U_n(Z^{\ast (p)}\mid N^{(p)})\to 0\,,\nto\mid (Z_i),(N^{(p)})\big)=1\,.
\eeao
\ele
%\begin{comment}
\begin{proof}
We observe that given $N^{(p)}$, 
\beao
U_n(Z^\ast)-U_n(Z^{\ast (p)}\mid N^{(p)})= \frac{1}{n} 
\sum_{1\le i\neq j \le n}
 \big(h_2(Z_i^\ast,Z_j^\ast)-h_2(Z_i^{\ast(p)},Z_j^{\ast (p)}\mid N^{(p)})\big)
\eeao
is a $U$-statistic of order $2$ with a $1$-degenerate kernel. Again applying 
the variance formula in Lemma A of \cite{serfling:1980}, p. 183, we obtain
\beao
&& \var^\ast(U_n(Z^\ast) -U_n(Z^{\ast (p)}\mid N^{(p)})\mid N^{(p)}) \\
&&\le  c\,
 \var^\ast(h_2(Z_1^\ast,Z_2^\ast)-h_2(Z_1^{\ast (p)},Z_2^{\ast (p)}\mid N^{(p)})\mid N^{(p)})\\
&& = c\,\frac{1}{n^2} \sum_{k,\ell=1}^n\big(
h_2(Z_k,Z_\ell)-h_2(Z_k^{(p)},Z_\ell^{(p)}\mid N^{(p)})
\big)^2.
\eeao
Taking expectations on both sides, we have  
\beao
 \E\big[\var^\ast(U_n(Z^\ast)-U_n(Z^{\ast(p)}\mid N^{(p)})\mid N^{(p)})\big] &\le& c\, \E\big[\big(
h_2(Z_1,Z_2)-h_2(Z_1^{(p)},Z_2^{(p)}\mid N^{(p)})
\big)^2\big]\,.
\eeao
Now appeal to \eqref{eq:august23b} and take expectations 
on both sides. Keeping in mind 
\eqref{eq:august23c}, we have to bound expressions of the 
following type:
\beam\label{eq:nov27}
\lefteqn{c\,\E\Big[\E\big[
\big( \|X_1-X_2\|_2-\|X_1^{(p)}-X_2^{(p)}\|_2\big)^2 \mid N^{(p)}\big]
\, \E[\|Y_1-Y_2\|_2^2]}\nonumber\\
&& +c\,\E\big[\|X_1^{(p)}-X_2^{(p)}\|_2^2
\,\big(\|Y_1-Y_2\|_2-\|Y_1^{(p)}-Y_2^{(p)}\|_2\big)^2\mid  N^{(p)}\big]\Big]\nonumber\\
&\le&c \,\E\big[\big( \|X_1-X_2\|_2-\|X_1^{(p)}-X_2^{(p)}\|_2\big)^2\big]\nonumber\\
&&+c\,\big(\E\big[\|X_1^{(p)}-X_2^{(p)}\|_2^4\big]\big)^{1/2}\,
\big(\E\big[\big(\|Y_1-Y_2\|_2-\|Y_1^{(p)}-Y_2^{(p)}\|_2\big)^4\big]\big)^{1/2}\,.
\eeam
We write $A=X_1-X_2$. Then we have for the first expression on the \rhs 
\beam\label{eq:nov27a}
c\,\E\big[\big( \|A\|_2-\|A^{(p)}\|_2\big)^2\big]
&\le & c\,\E\Big[\E\Big[N_p^{-1} \Big|\sum_{i=1}^{N_p}\Big( A^2(\bfU_i)-\int_B A^2(\bfu)d \bfu\Big)\Big|\; \Big| A,N_p\Big]\Big]\nonumber\\
&\le &c\,\E\Big[\Big(\E\Big[\Big(N_p^{-1} \sum_{i=1}^{N_p}\Big( A^2(\bfU_i)-\int_B A^2(\bfu)d \bfu\Big)\Big)^2\; \Big|A,N_p\Big]\Big)^{1/2}\Big]\nonumber\\
&=& c\, \E\big[N_p^{-1/2}\1(N_p>0)\big]\, \E\Big[\big(\var\big( A^2(\bfU_1)\mid A\big)\big)^{1/2}\Big]\nonumber\\
&\le &c\, \E\big[N_p^{-1/2}\1(N_p>0)]\, \Big(\int_B \E[(X_1-X_2)^4(\bfu)] d\bfu\Big)^{1/2}\,.
\eeam
The second factor is finite by the moment assumptions on $X$.  
Using the Poisson structure of $N_p$, we can apply Lemma~4.1 %in \citep{supplement:2} 
to the first factor and conclude that it has the \asy\ order  $O(p^{-1/2})$.
Now we consider the first expectation in \eqref{eq:nov27}: 
\beao
\E\big[\|A^{(p)}\|_2^4\big]= \E\Big[\Big(N_p^{-1}\sum_{i=1}^{N_p} A^2(\bfU_i)\Big)^2 \Big]
&\le &\E\Big[N_p^{-1}\sum_{i=1}^{N_p} A^4(\bfU_i) \Big]\\
&=& \E[\1(N_p>0)]\, \int_B \E\big[(X_1-X_2)^4(\bfu)\big]d\bfu\,.
\eeao
The \rhs\ is finite in view of the moment conditions on $X$.
Next we turn to the second expression in \eqref{eq:nov27}.
Write $C=Y_1-Y_2$. It remains to bound
%\beao
%&& +c\,\Big(\E\Big[\Big[\Big(N_p^{-1}\sum_{i=1}^{N_p} (X_1-X_2)^2(\bfU_i)\Big)^2\Big|X_1-X_2\Big]\Big)^{1/2}\\
%&&\times \Big(
%\E\Big[\E\Big[ \Big(N_p^{-1}\Big|\sum_{i=1}^{N_p}\Big((Y_1-Y_2)^2(\bfU_i)-\int_B(Y_1-Y_2)^2(\bfu)d\bfu\Big|\Big) \Big)^2\Big|Y_1-Y_2\Big]\Big]\Big)^{1/2}\,.
%\eeao
\beao
\big(\E\big[\big(\|C\|_2-\|C^{(p)}\|_2\big)^4\big]\big)^{1/2}
&\le &
c\,\Big(\E\Big[\E\Big[\Big(N_p^{-1}\sum_{i=1}^{N_p}\Big( C^2(\bfU_i)-\int_B C^2(\bfu)d \bfu\Big)\Big)^2\; \Big|C,N_p\Big]\Big]\Big)^{1/2}\\
&\le & c\,\big(\E\big[N_p^{-1}\1(N_p>0)]\big)^{1/2}\, 
\Big(\int_B \E[ C^4(\bfu)]\, d\bfu\Big)^{1/2}\\
&=&O(p^{-1/2})\,,\qquad n\to\infty\,.
\eeao
Here we again used Lemma~4.1. %in \citep{supplement:2}.
Summarizing the bounds above, we conclude that for $\vep>0$,
\beao
&& \sum_n \P(|U_n(Z^\ast)- U_n(Z^{\ast (p))}\mid N^{(p)})|>\vep) \\
&&\le  \vep^{-2}\sum_n\E\Big[\var^\ast\big( U_n(Z^\ast)- U_n(Z^{\ast (p)}) \mid N^{(p)})\mid N^{(p)}
\big)\Big]\\
&&\le c\,\sum_np_n^{-1/2}<\infty\,.
\eeao
Therefore the Borel-Cantelli lemma implies that
\beao
U_n(Z^\ast)-U_n(Z^{\ast (p)}\mid N^{(p)})\to 0\,,\qquad \nto\,.
\eeao
for a.e. sample path of $(Z_i)$ and $(N^{(p)})_{p>0}$.
\end{proof}
\bco
Under the conditions of Lemma~\ref{lem:boot}, $(n\,T_n(X,Y))$ and 
$(U_n(Z^{\ast (p)})$ (conditional on $(N^{(p)})$) have the same limit \ds\
in \eqref{eq:may8fa} (up to changes of scale/location).
\eco
This result means that the modified bootstrap for 
$(n\,T_n(X^{(p)},Y^{(p)}))$ is 
consistent.
\begin{proof}
Following the discussion in Sections~\ref{subsec:appa} and \ref{sec:bootstr},
$(nT_n(X,Y))$, $(U_n(Z))$ and $(U_n(Z^\ast))$ (given the data) have the same
limit \ds\ (up to possible changes of scale/location). In this section
we proved that $(U_n(Z))$ and $(U_n(Z^{(p)}))$ (given $(N^{(p)})$) and 
$(U_n(Z^\ast))$ and $(U_n(Z^{\ast(p)}))$ (given the data and $(N^{(p)})$) have the 
same limit \ds . This finishes the proof.
\end{proof}
It remains to show the bootstrap consistency for $(n\,R_n(X^{(p)},Y^{(p)}))$.
We consider the following modified bootstrap version of the latter \seq :
\beao
\dfrac{n\,U_n(Z^{\ast (p)})}{\sqrt{T_n(X^{(p)},X^{(p)})\,T_n(Y^{(p)},Y^{(p)})}}. 
\eeao
Since under our moment conditions, $T_n(X,X) \stas T(X,X)$ and $T_n(Y,Y)\stas
T(Y,Y)$ by the \slln\ for $V$-statistics it remains to show that
\beam\label{fin}
T_n(X^{(p)},X^{(p)})-T_n(X,X)\stas 0\,,\qquad \nto\,,
\eeam 
and the corresponding result for $(T_n(Y^{(p)},Y^{(p)}))$. This is the content
of the following lemma.
\ble
Under the conditions of Lemma~\ref{lem:boot} and summability of 
$(p_n^{-1/4})$, \eqref{fin} and the corresponding 
result for $(T_n(Y^{(p)},Y^{(p)}))$ hold.
\ele
\begin{proof} We have the decomposition
\beao\lefteqn{
T_n(X,X)-T_n(X^{(p)},X^{(p)})}\\&=&
\dfrac 1 {n^2} 
\sum_{i,j=1}^n  \big(\|X_i-X_j\|_2^2- \|X_i^{(p)}-X_j^{(p)}\|_2^2)\\&&
+\Big[ \Big(\dfrac 1{n^2}\sum_{i,j=1}^n \|X_i-X_j\|_2\Big)^2-\Big(\dfrac 1 {n^2}\sum_{i,j=1}^n
\|X_i^{(p)}-X_j^{(p)}\|_2\Big)^2\Big]\\&&
-2 \dfrac 1 {n^3} \sum_{i,j,k=1}^n\Big[ 
(\|X_i-X_j\|_2-\|X_i^{(p)}-X_j^{(p)}\|_2)\,\|X_i-X_k\|_2\\&& 
+  \|X_i^{(p)}-X_j^{(p)}\|_2\,(\|X_i-X_k\|_2-\|X_i^{(p)}-X_k^{(p)}\|_2)\Big]\\
&=:& I_1+I_2-2 (I_{31}+I_{32})\,.
\eeao 
Writing $A=X_1-X_2$, we see that
\begin{align*}
\E[|I_1|] \le  \E\big[|\|A\|_2^2-\|A^{(p)}\|_2^2|\big]=
\E\Big[\E\Big[N_p^{-1} \Big|\sum_{i=1}^{N_p}\Big( A^2(\bfU_i)-\int_B A^2(\bfu)d \bfu\Big)\Big|\; \Big| A,N_p\Big]\Big]\,.
\end{align*}
Now we can proceed as for \eqref{eq:nov27a} and conclude that the \rhs\
is $O(p^{-1/2})$. Since we assume summability of $(p_n^{-1/4})$ the Borel-Cantelli lemma yields $I_1\stas 0$.
\par
For $I_2\stas 0$ it suffices to prove that
\beao
I_4:= \dfrac 1{n^2}\sum_{i,j=1}^n \big(\|X_i-X_j\|_2-\|X_i^{(p)}-X_j^{(p)}\|_2\big)\stas 0\,,\qquad \nto\, 
\eeao 
since, by the \slln\ for $U$-statistics,  $n^{-2}
\sum_{i,j=1}^n \|X_i-X_j\|_2\stas \E[\|X_1-X_2\|_2 ]$.
We have
\beao
\E[|I_4|]&\le & \E\big[\big|\|A\|_2-\|A^{(p)}\|_2\big|\big]\le  \big(\E\big[\big|\|A\|_2-\|A^{(p)}\|_2\big|^2\big]\big)^{1/2}. 
\eeao
Following the proof in \eqref{eq:nov27a}, the \rhs\ is bounded
by $c\,(\E[N_p^{-1/2}\1(N_p>0)])^{1/2}=O(p^{-1/4})$. The summability of 
$(p_n^{-1/4})$ and the Borel-Cantelli lemma prove that $I_4\stas 0$.
\par
We focus on showing $I_{32}\stas 0$; the case $I_{31}\stas 0$ is analogous.
We have
\beao
\E[|I_{32}|]&\le & \E\big[
\|X_1^{(p)}-X_2^{(p)}\|_2\,\big|\|X_1-X_3\|_2-\|X_1^{(p)}-X_3^{(p)}\|_2\big|\big]\\
&\le & \big(\E\big[ \|A^{(p)}\|_2^2\big]\big)^{1/2}\,\big(\E\big[
\big(\|A\|_2-\|A^{(p)}\|_2\big)^2\big]\big)^{1/2}\,.
\eeao
The first expected value is bounded uniformly for $p$. Another application
of \eqref{eq:nov27a} shows that the \rhs\ is of the order $O(p_n^{-1/4})$.
Another application of the Borel-Cantelli lemma proves $I_{32}\stas 0$.
This finishes the proof.
\end{proof}
\section{A Monte Carlo study}\label{sec:simulations}\setcounter{equation}{0}
\subsection{Finite sample study of the \con }
In this section we conduct a Monte Carlo study of 
the finite sample behavior of the 
sample distance correlation for $\beta=1$; in what follows we suppress $\beta$ in the notation. We consider a fractional Brownian sheet
(fB for short) with Hurst parameter
 $H=(H_1,H_2)\in (0,1)^2$, in particular, $H_i\in \{1/4,1/2,3/4\}$.
As a heavy-tailed alternative we choose 
symmetric  $1.8$-stable \levy\ sheets. The observations
are given either on a lattice or at random locations in $[0,1]^2$. 
In the lattice case we take $q$ equidistant grid points on 
each side of $[0,1]^2$, resulting in a lattice of  
$p=q\times q$ points. We choose sample sizes $n\in\{100,200,300\}$ 
and repeat the Monte Carlo simulations 500 times
for each process.

\begin{figure}[h!]
\begin{center}
\includegraphics[width=0.9\textwidth]{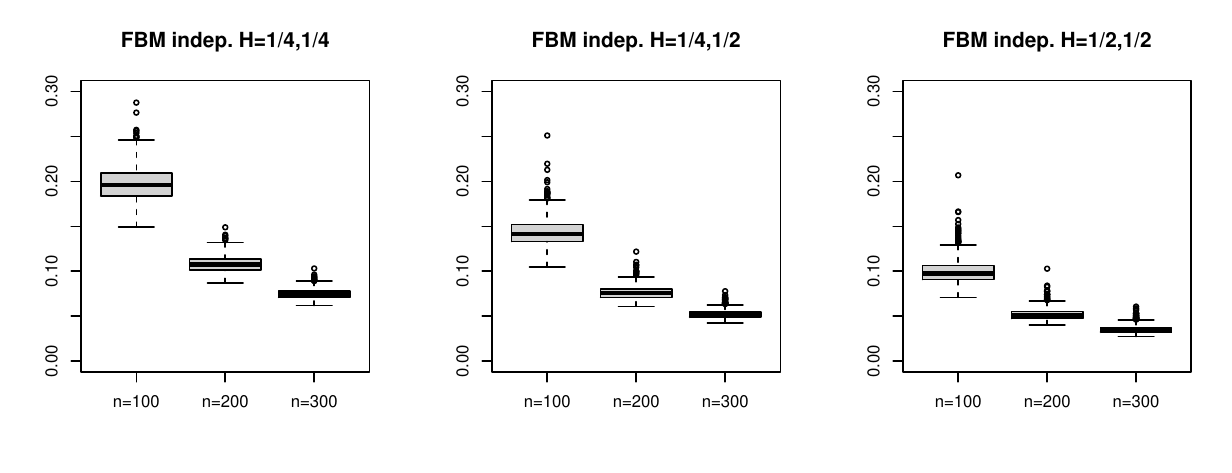}
\includegraphics[width=0.9\textwidth]{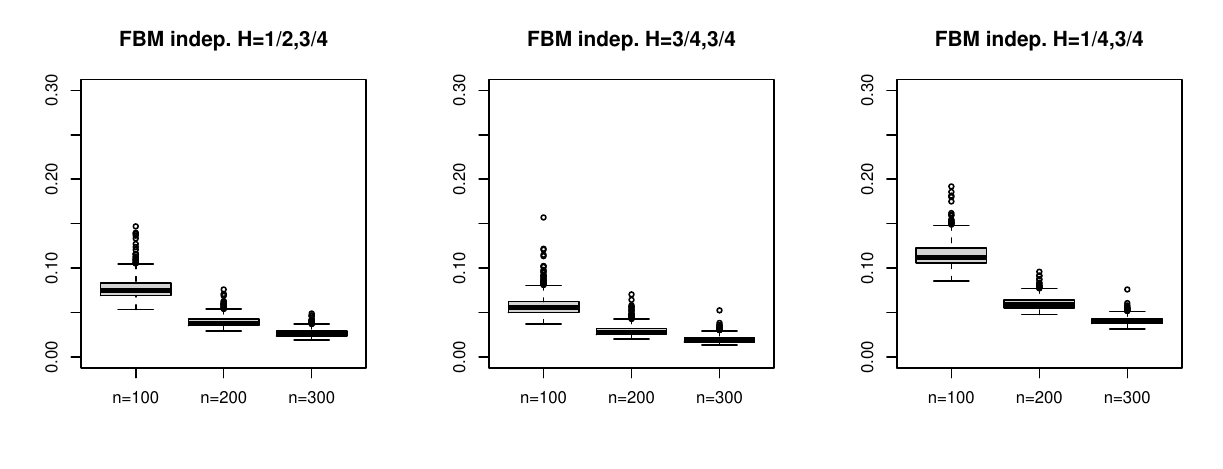}
\includegraphics[width=0.9\textwidth]{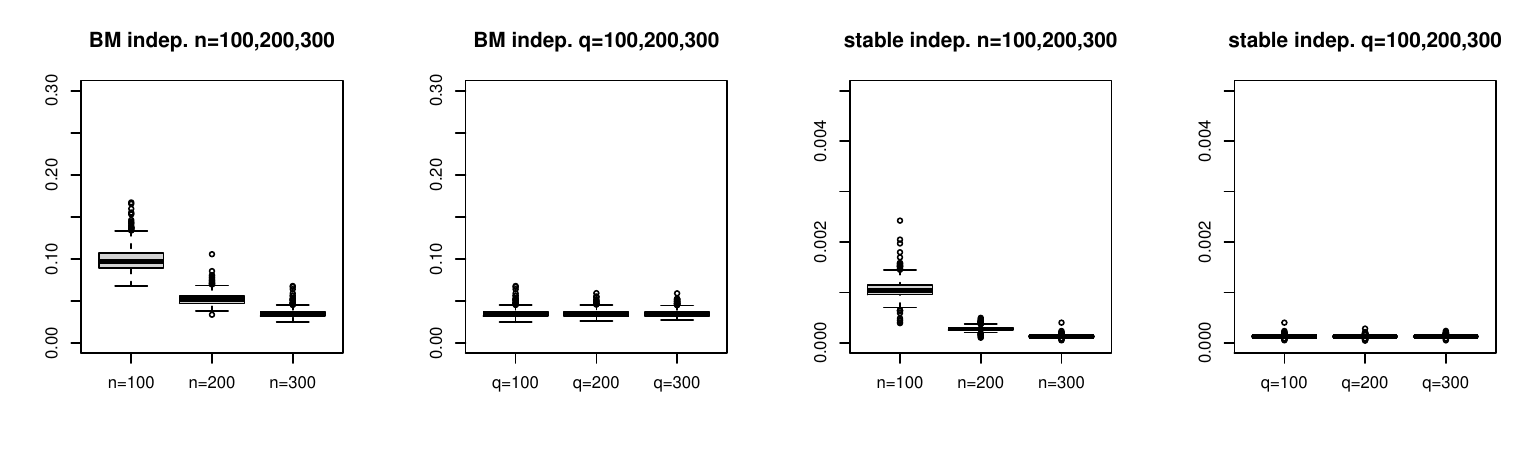}
\end{center}
\bfig\label{fig:1}{\rm \small %{\red description of second graph in middle row is not correct for all graphs} 
Boxplots for $R_n(X^{(p)},Y^{(p)})$ 
based on 500 simulations of independent sheets $X,Y$. 
The sheets are simulated on a $p=q \times q$ 
equidistant lattice on $[0,1]^2$
for $q=30$. {\bf Top and middle row:} fB sheets for different 
choices of Hurst coefficients $H_i\in\{1/4, 1/2,3/4\}$, $i=1,2$ and 
increasing sample size $n$. {\bf Bottom  row:}
 Effects of increasing $n$ (for fixed $q=100$) and increasing 
 $q$ (for fixed $n=300$). The left (right) graphs are based on 
 Brownian (1.8-stable) sheets.}\efig
\end{figure}

\begin{figure}[h!]
\begin{center}
\includegraphics[width=0.9\textwidth]{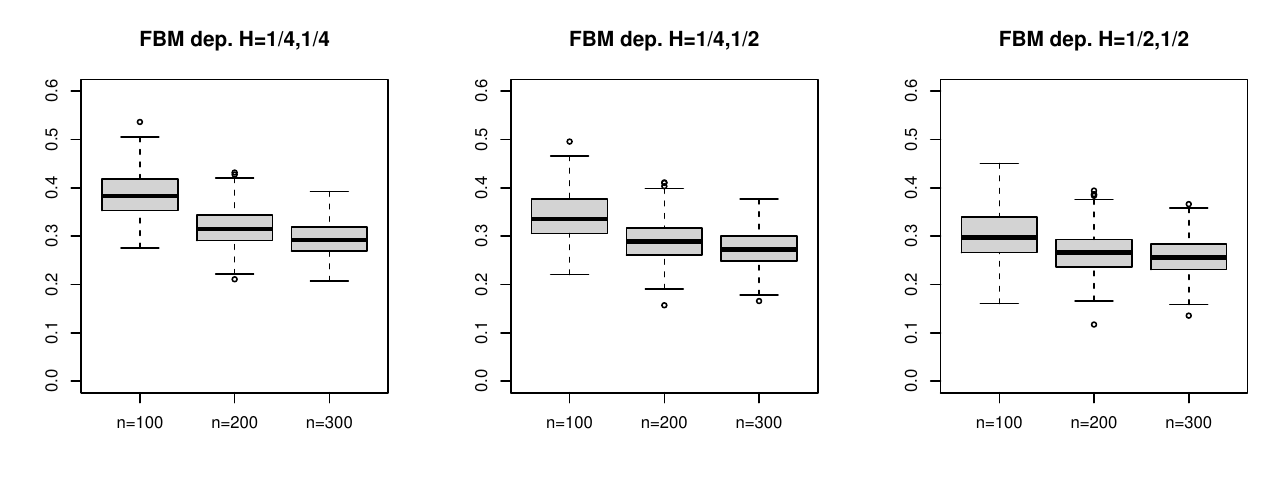}
\includegraphics[width=0.9\textwidth]{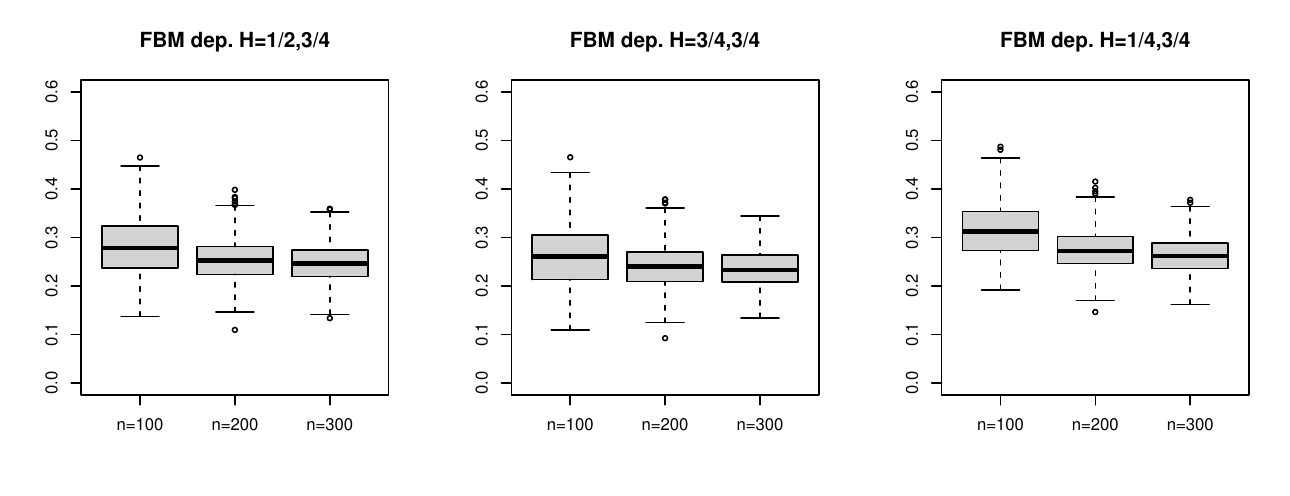}
\includegraphics[width=0.9\textwidth]{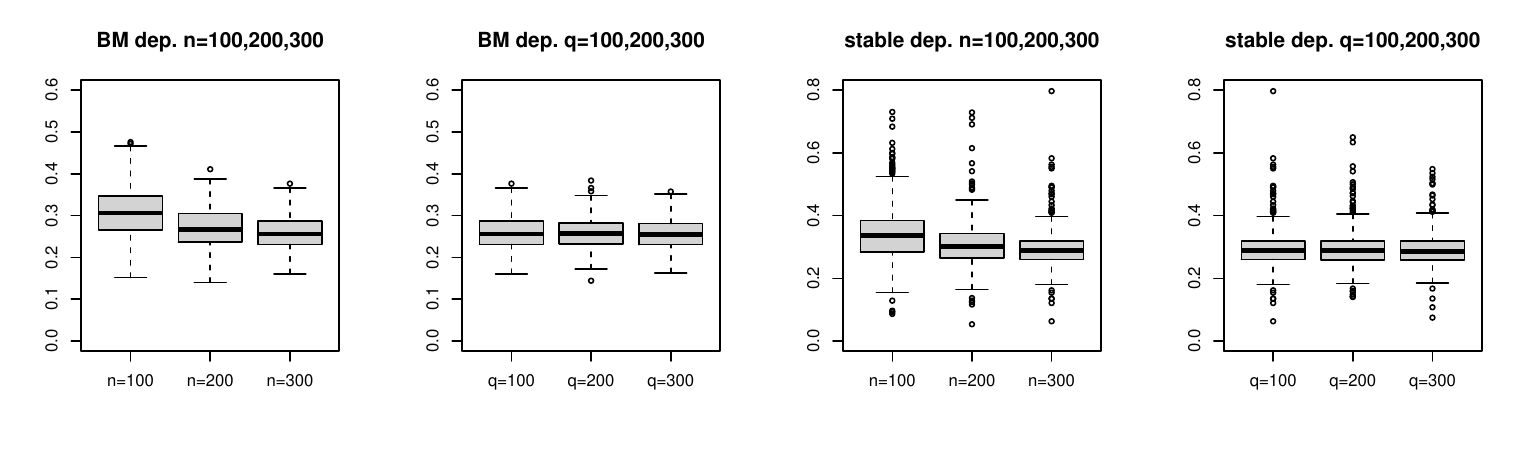}
\end{center}
\bfig\label{fig:2}{\rm \small Boxplots for $R_n(X^{(p)},Y^{(p)})$ based on 500 simulations 
of dependent fB and $1.8$-stable sheets $X,Y$ in the lattice case.
The parameters $H_1$ and $H_2$ and the values $n,p$
are the same as for the corresponding graph at the same location 
in Figure~\ref{fig:1}. The construction of $X,Y$ is based on 
\eqref{def:depBsheets} and 
\eqref{def:depstablesheets} for fB and 1.8-stable sheets for $\rho=0.5$, respectively.}\efig 
\end{figure}
\begin{figure}[h!]
\begin{tabular}{@{}ccc@{}}
\includegraphics[width=.28\textwidth]{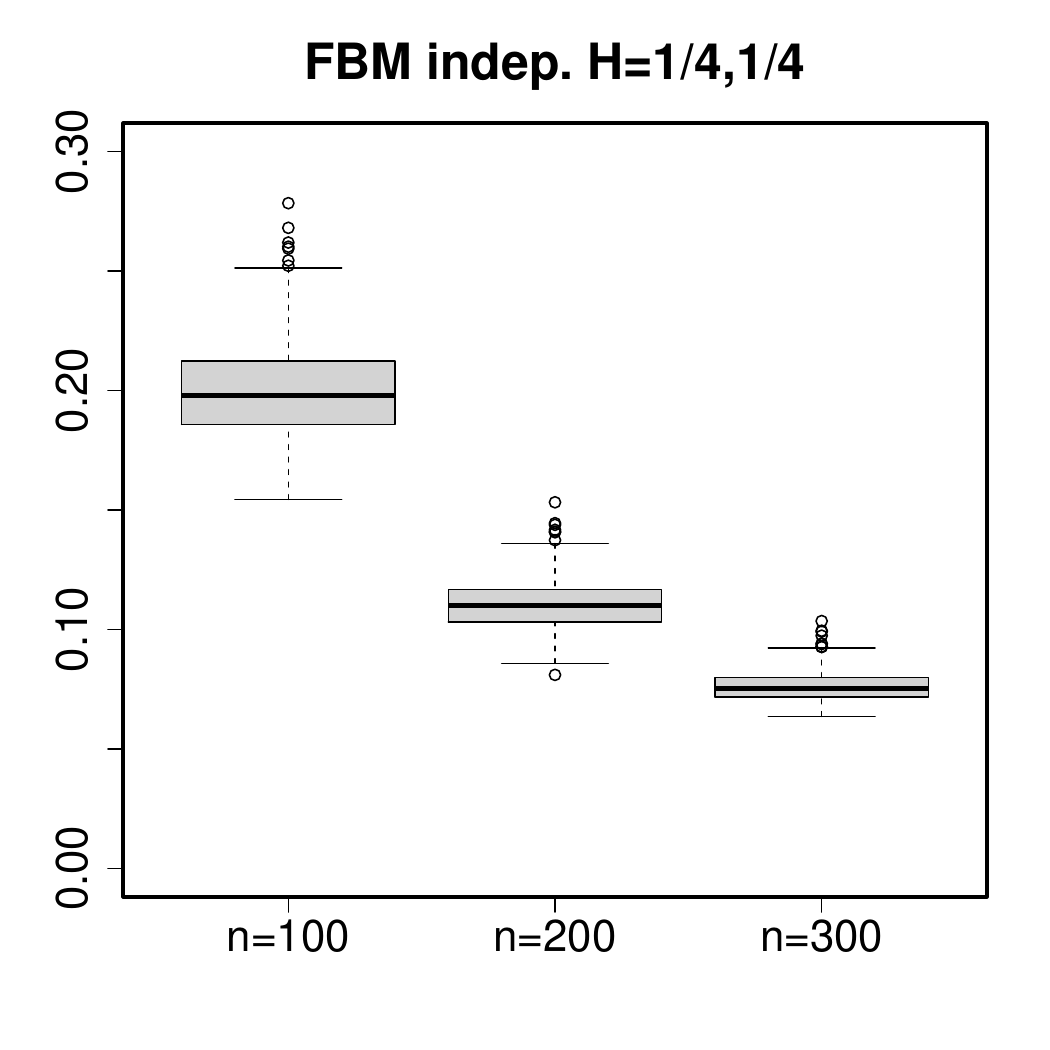}&
 \includegraphics[width=.28\textwidth]{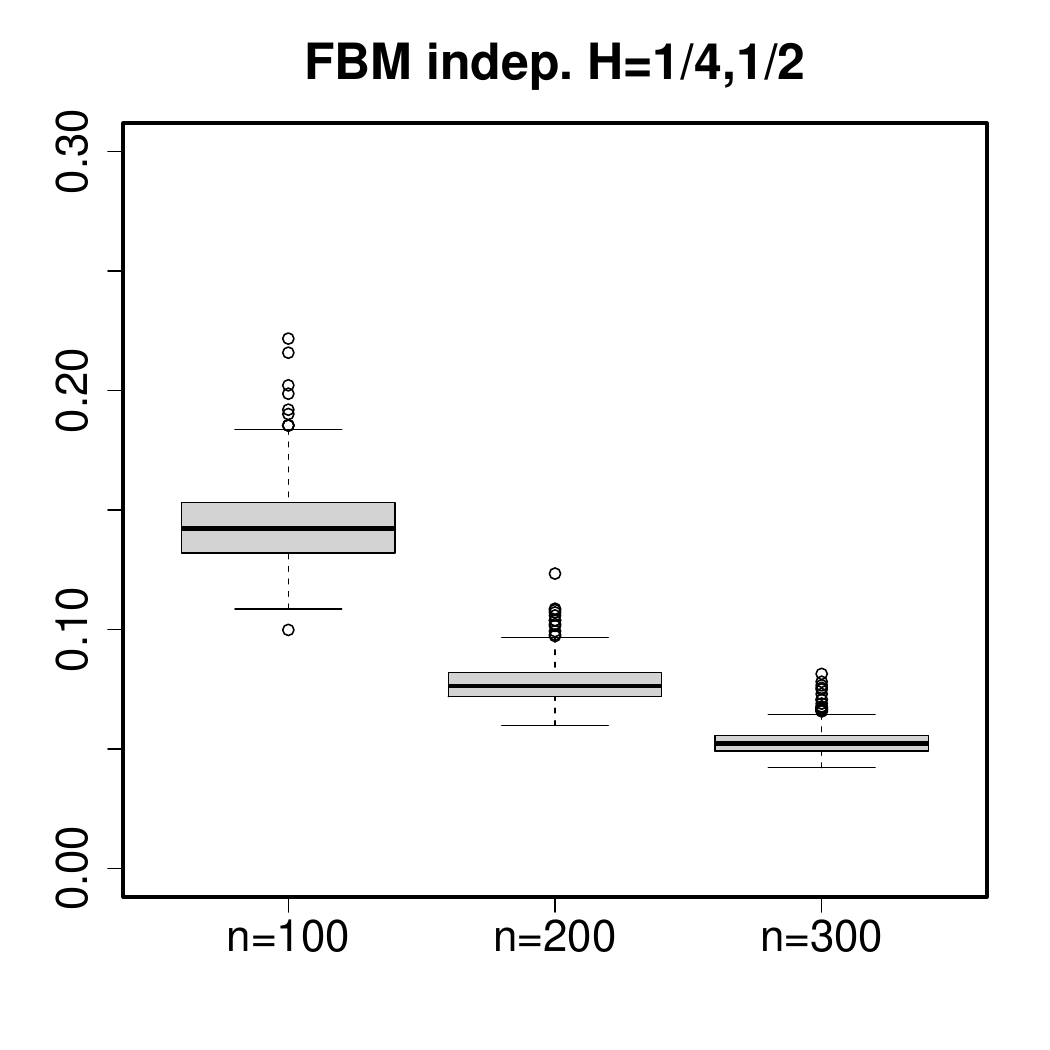}&
\includegraphics[width=.28\textwidth]{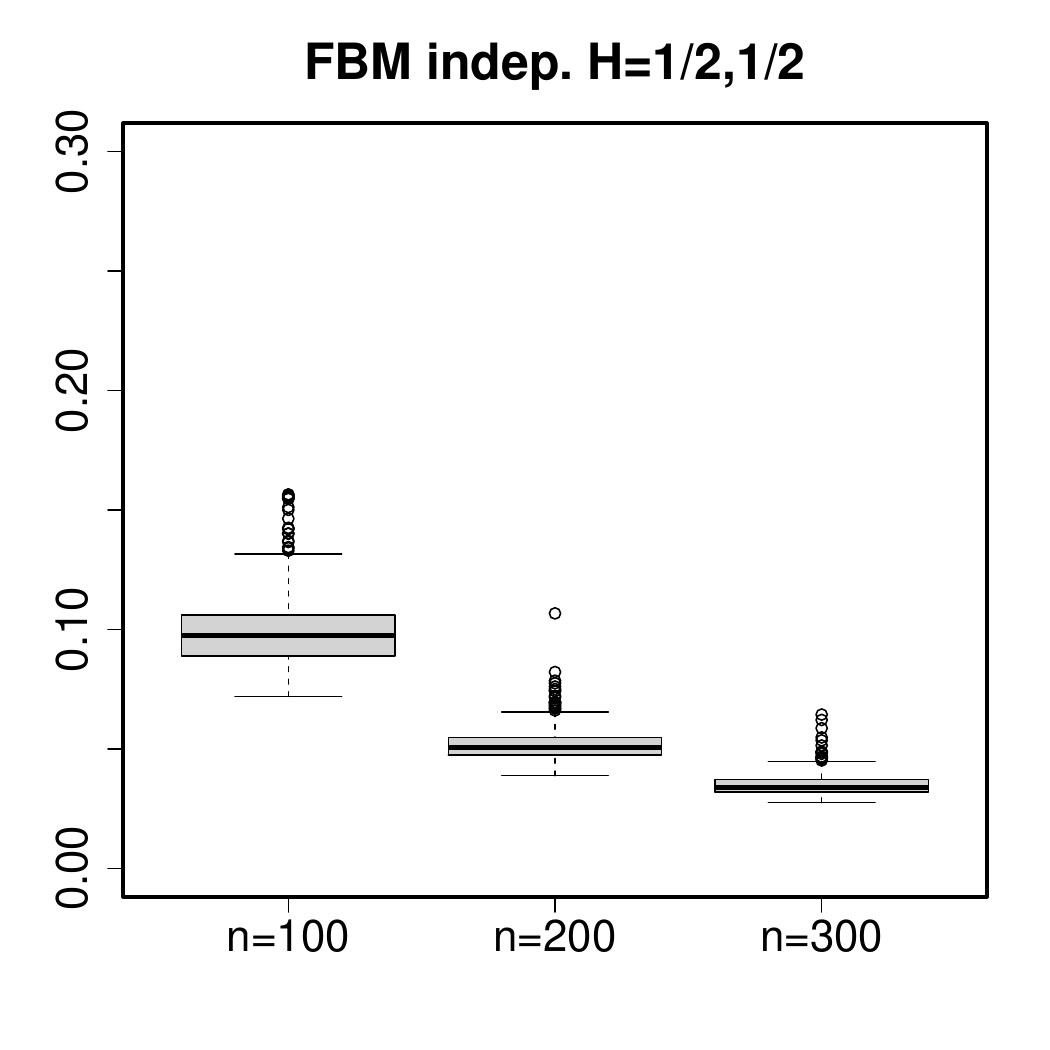}\\
 \includegraphics[width=.28\textwidth]{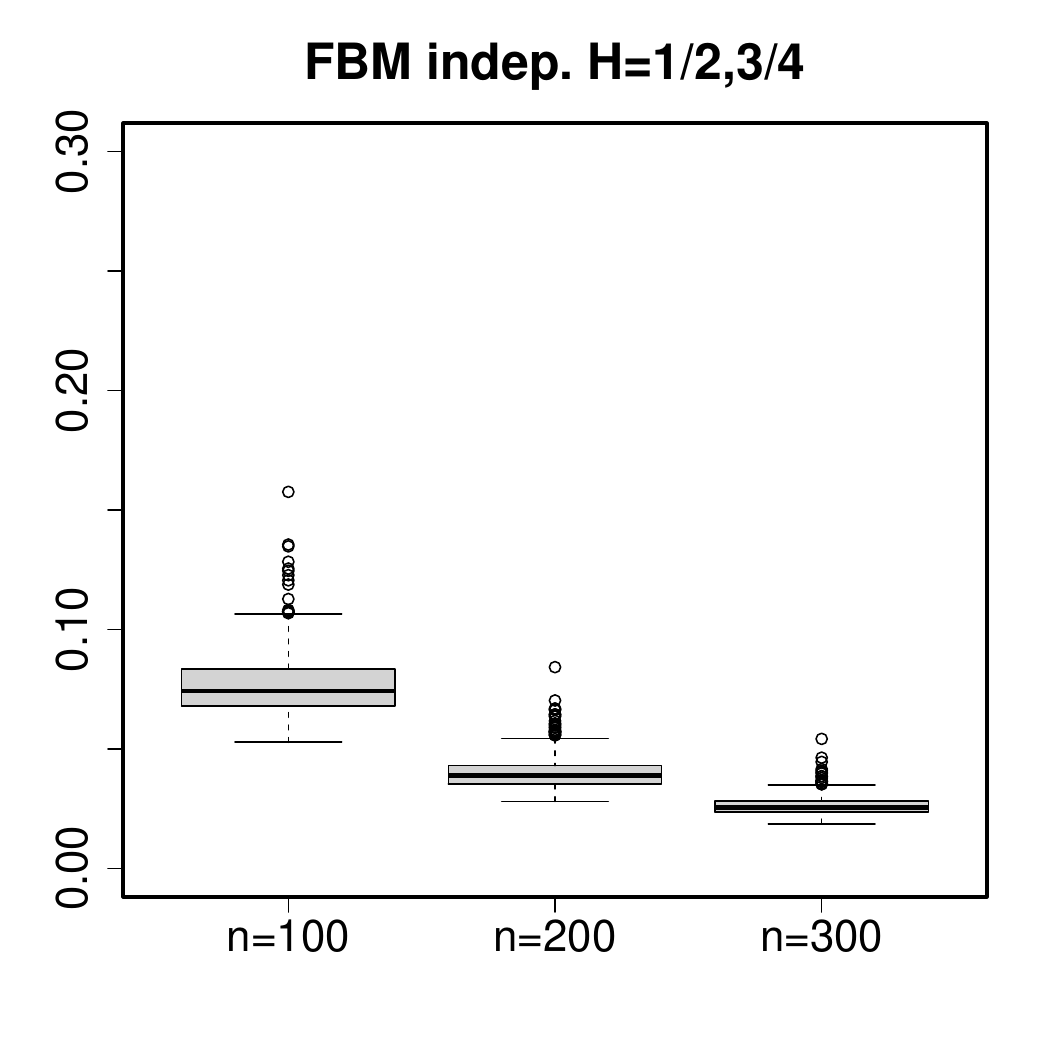}&
 \includegraphics[width=.28\textwidth]{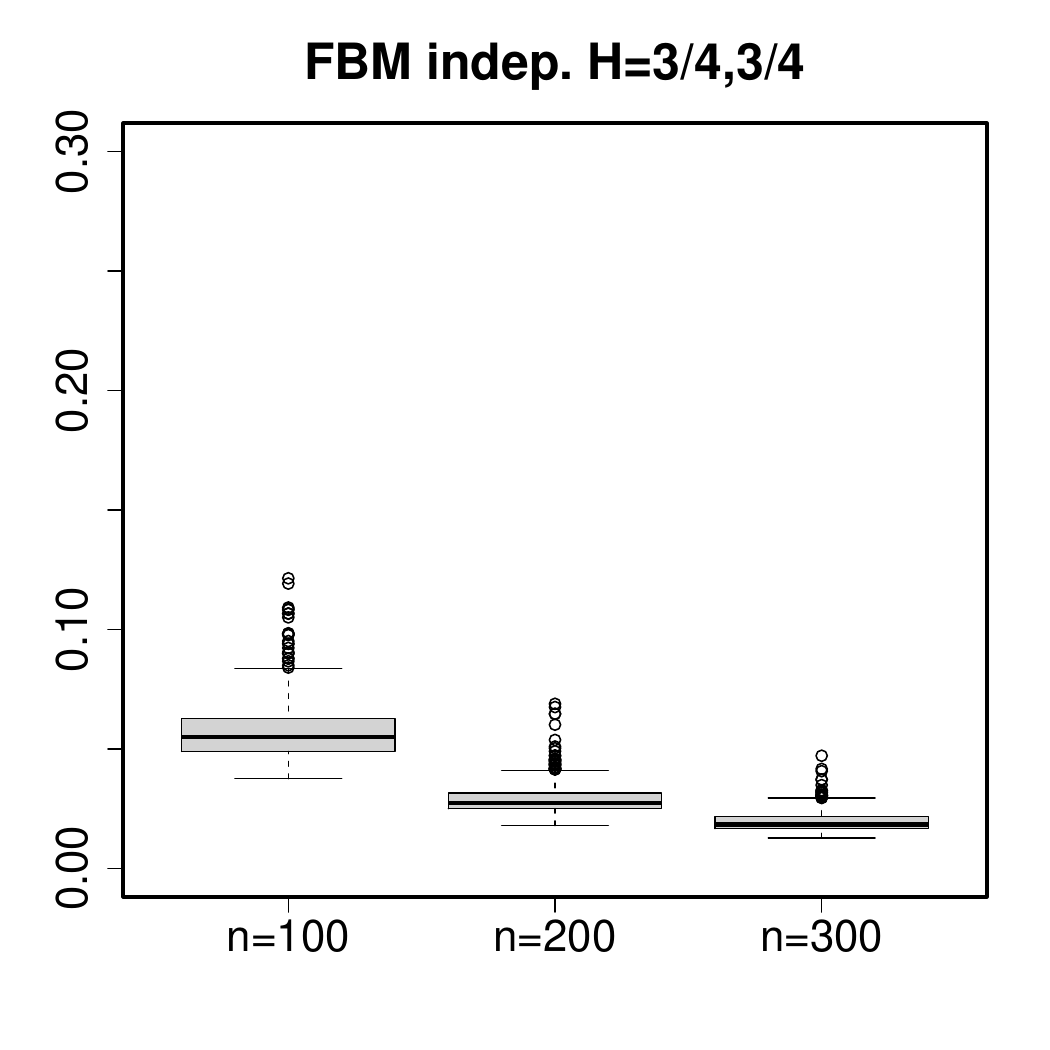}&
\includegraphics[width=.28\textwidth]{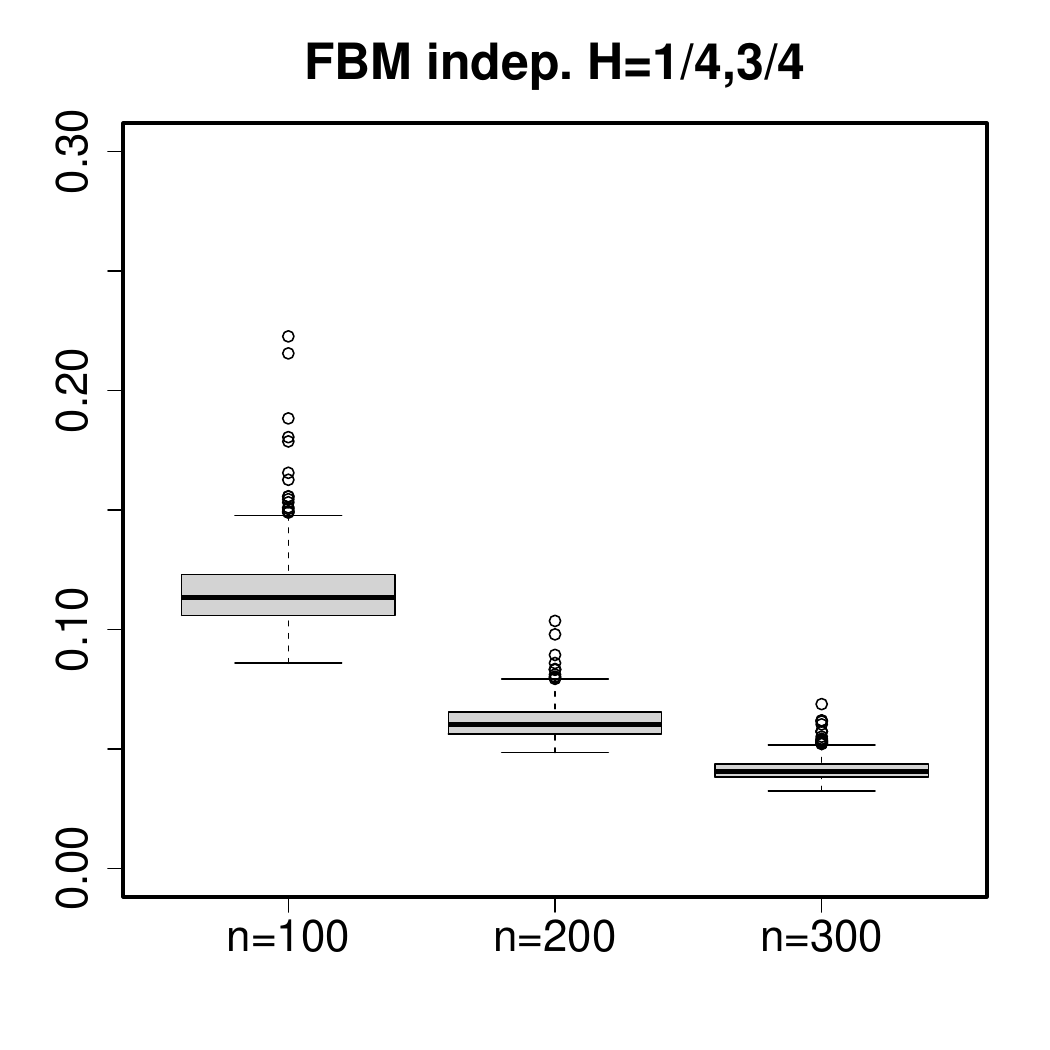}\\
\multicolumn{3}{c}{\includegraphics[width=.89\textwidth]{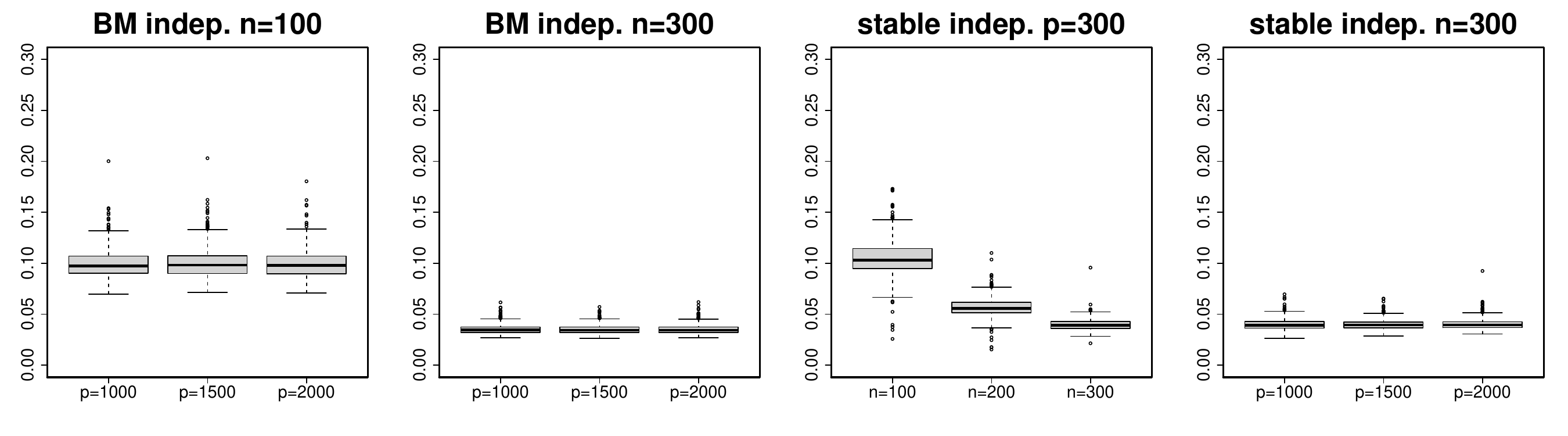}}\\
\end{tabular}
\bfig\label{fig:3}{\rm \small Boxplots for $R_n(X^{(p)},Y^{(p)})$ based on 500 simulations 
of independent fB and 1.8-stable sheets $X,Y$ at $N_p\sim{\Pois}(p)$, $p=1000$, 
uniformly distributed locations in $[0,1]^2$. 
The parameters $H_1$ and $H_2$ and the values $n$
are the same as for the corresponding graph at the same location  
in Figure~\ref{fig:1}. {\bf Top and middle row:}
fB sheets for different choices of Hurst parameters $H_i\in\{1/4, 1/2,3/4\}$,
 $i=1,2$, and increasing sample size $n$.  {\bf Bottom row:}
 Effects of %both $n$ (with fixed $p=1000$) and
 increasing $p$ for independent
 Brownian sheets and $n=100, 300$ (left two). In the two right graphs 
we consider the case of independent 
 $1.8$-stable sheets $X,Y$ with increasing $n$ ($p$) and fixed $p=300$ 
($n=300$).}\efig
\end{figure}
\begin{figure}[h!]
\begin{tabular}{@{}ccc@{}}
\includegraphics[width=.28\textwidth]{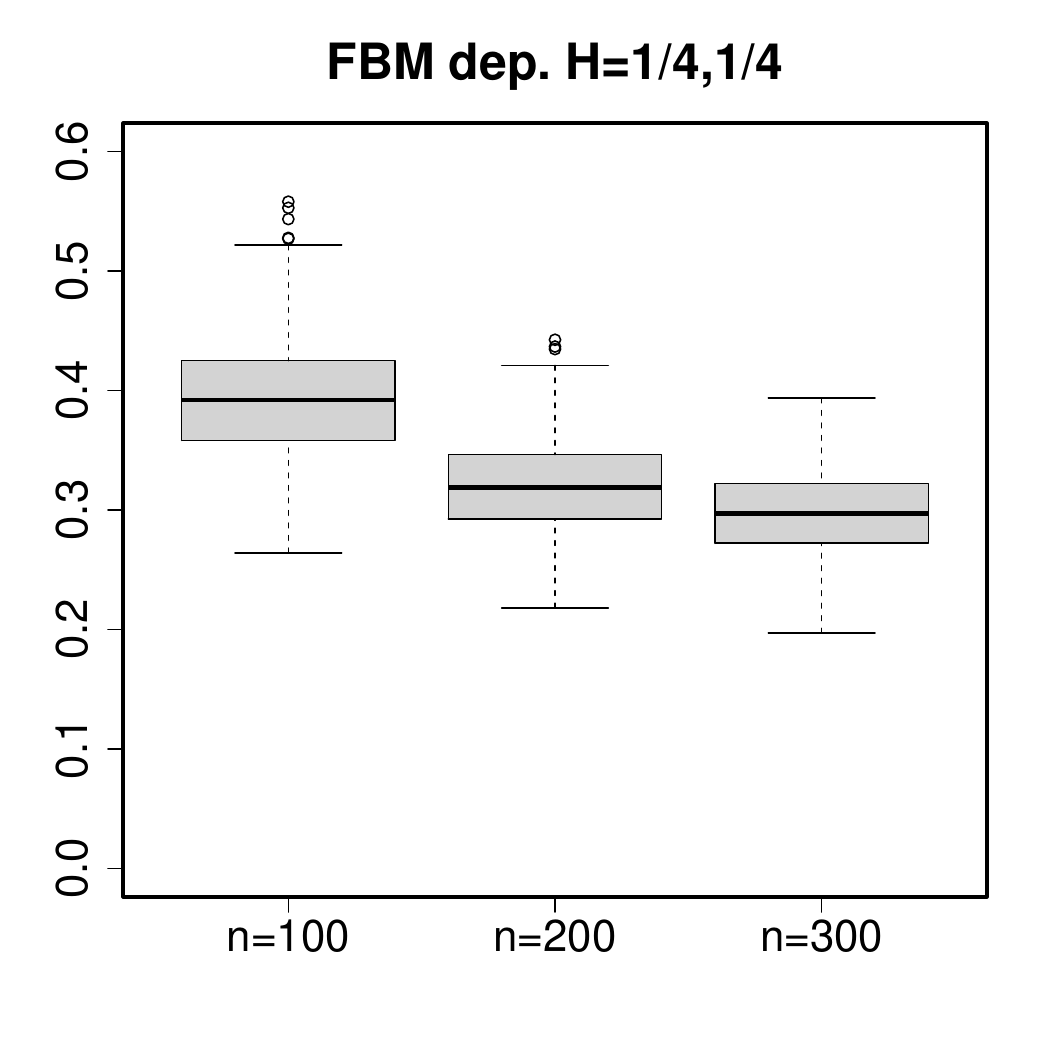}&
 \includegraphics[width=.28\textwidth]{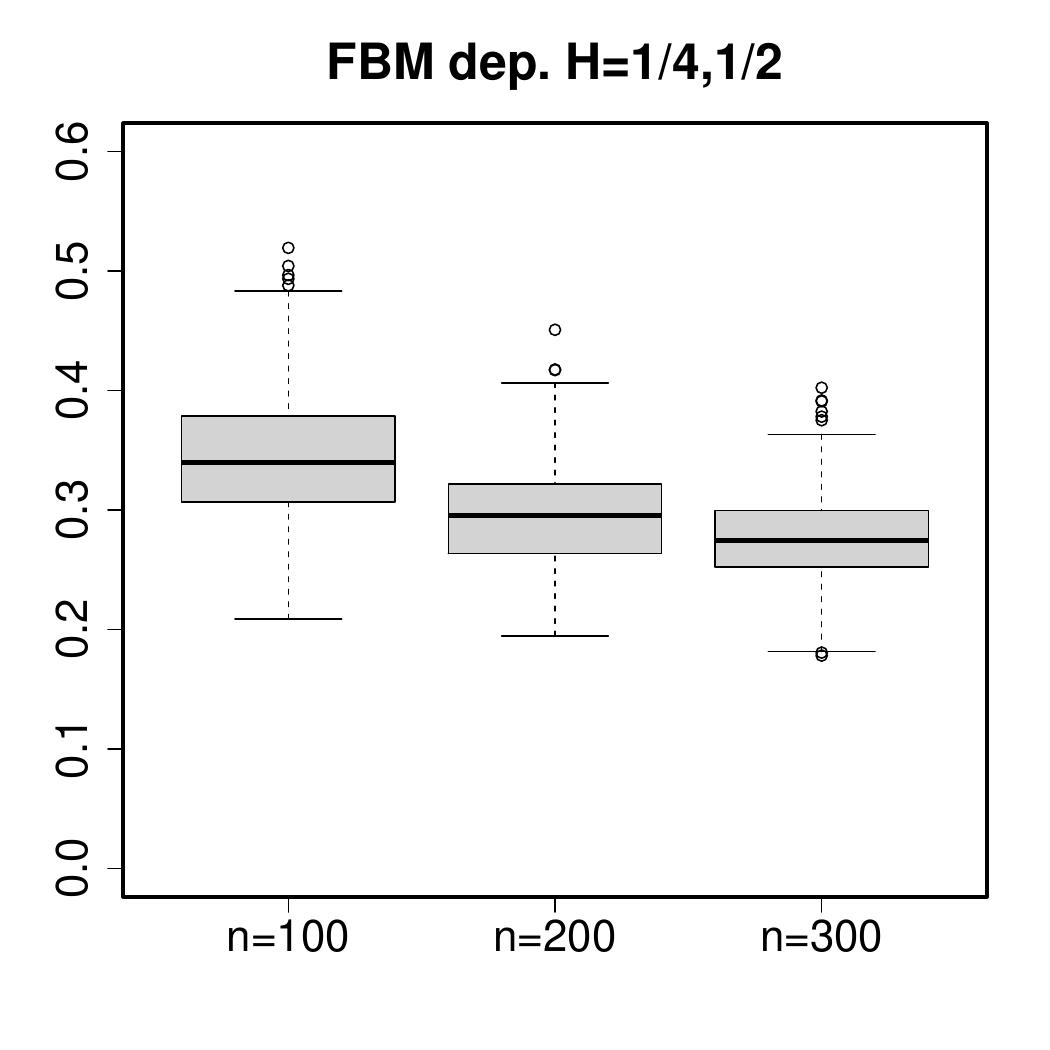}&
\includegraphics[width=.28\textwidth]{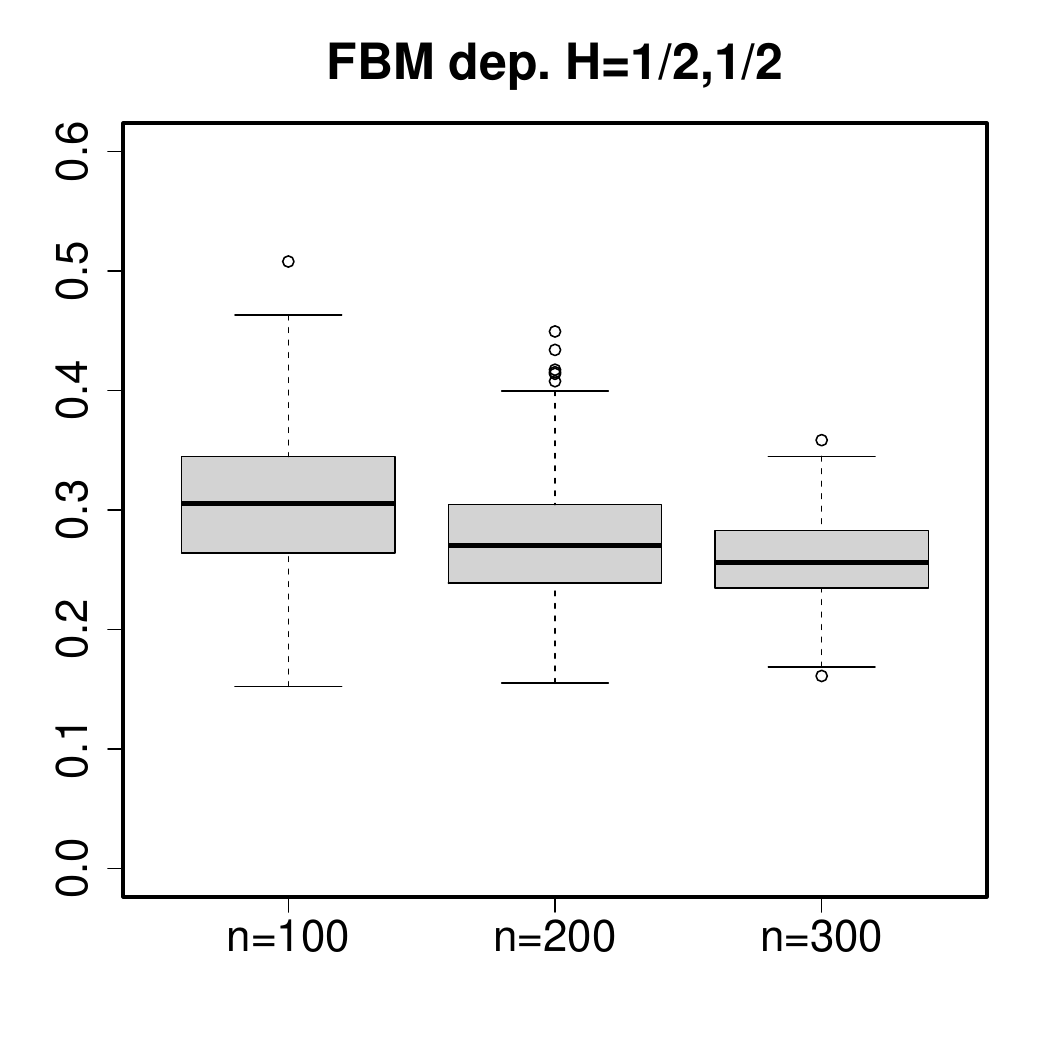}\\
 \includegraphics[width=.28\textwidth]{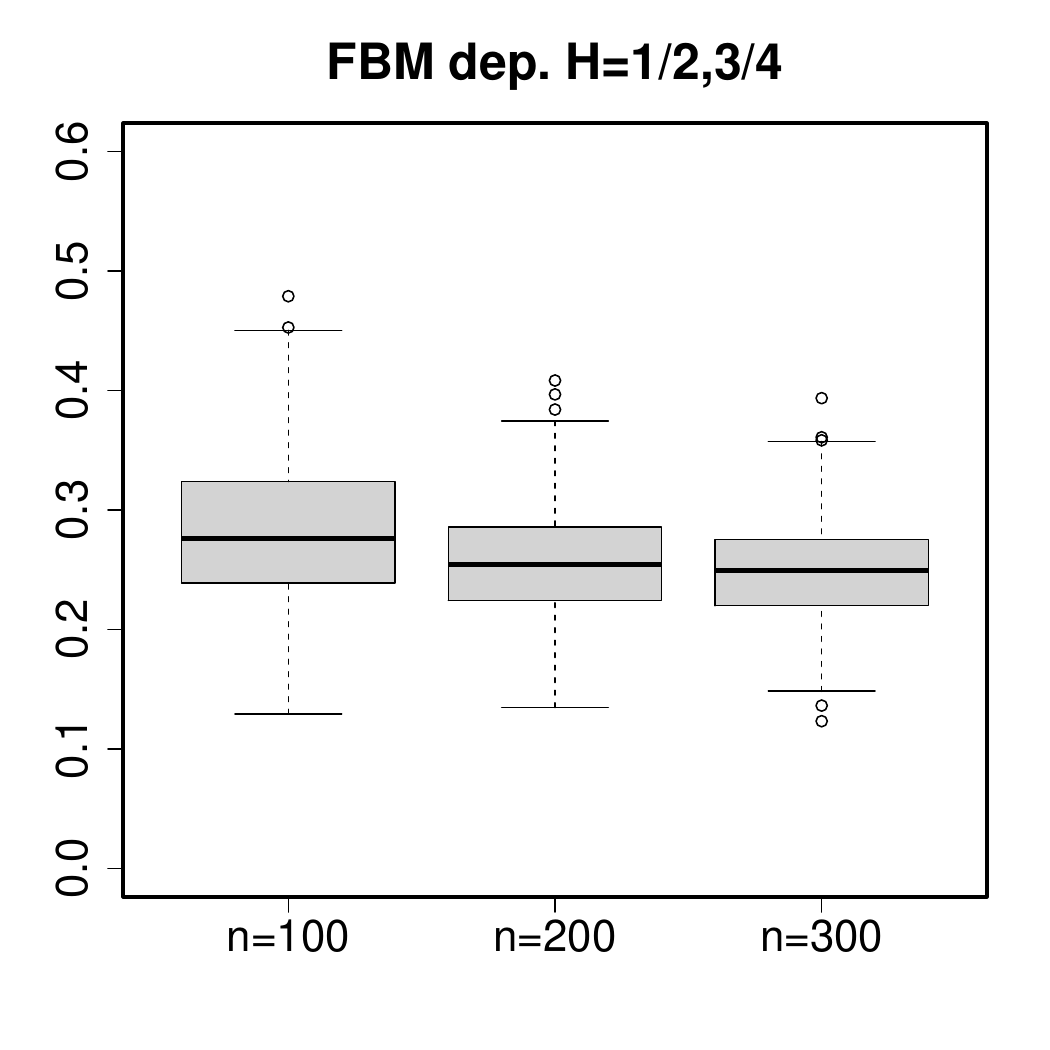}&
 \includegraphics[width=.28\textwidth]{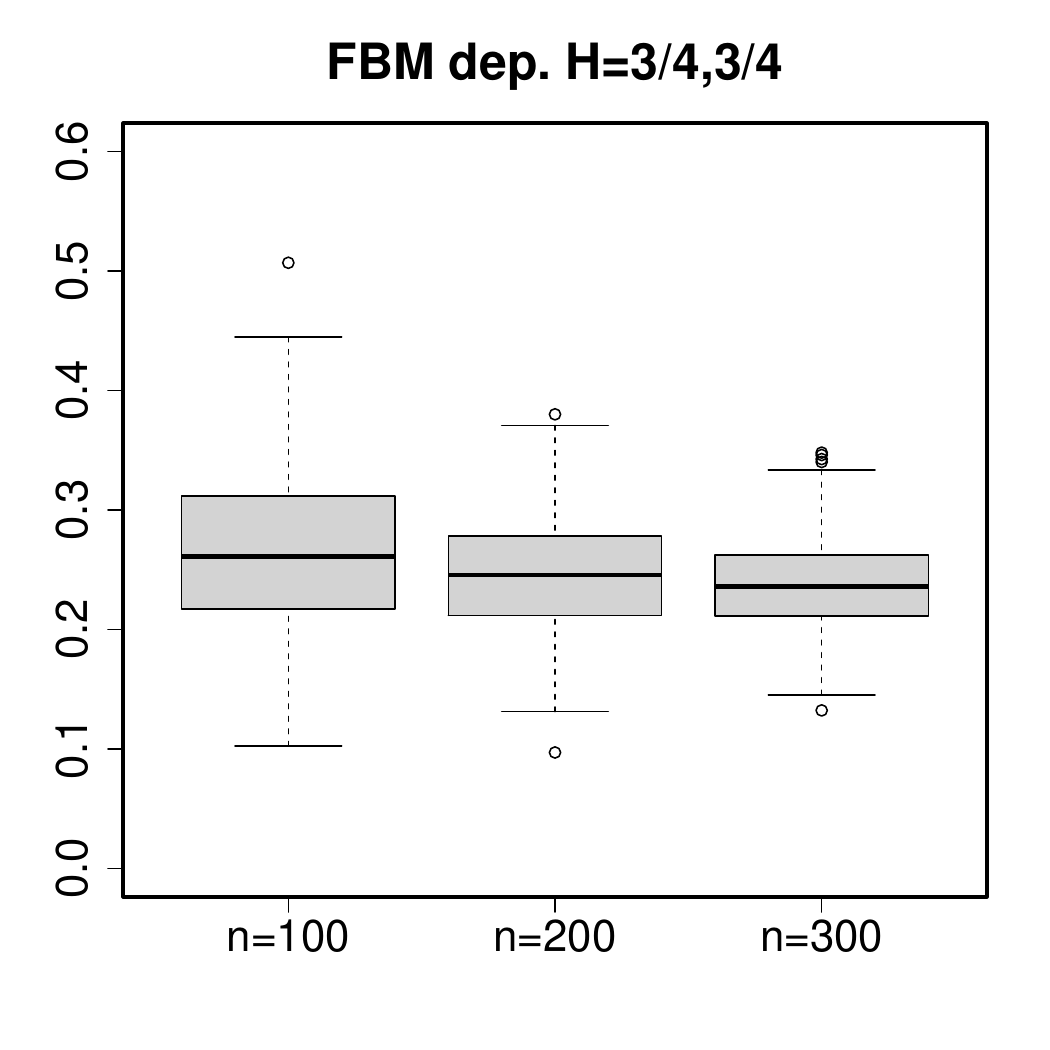}&
\includegraphics[width=.28\textwidth]{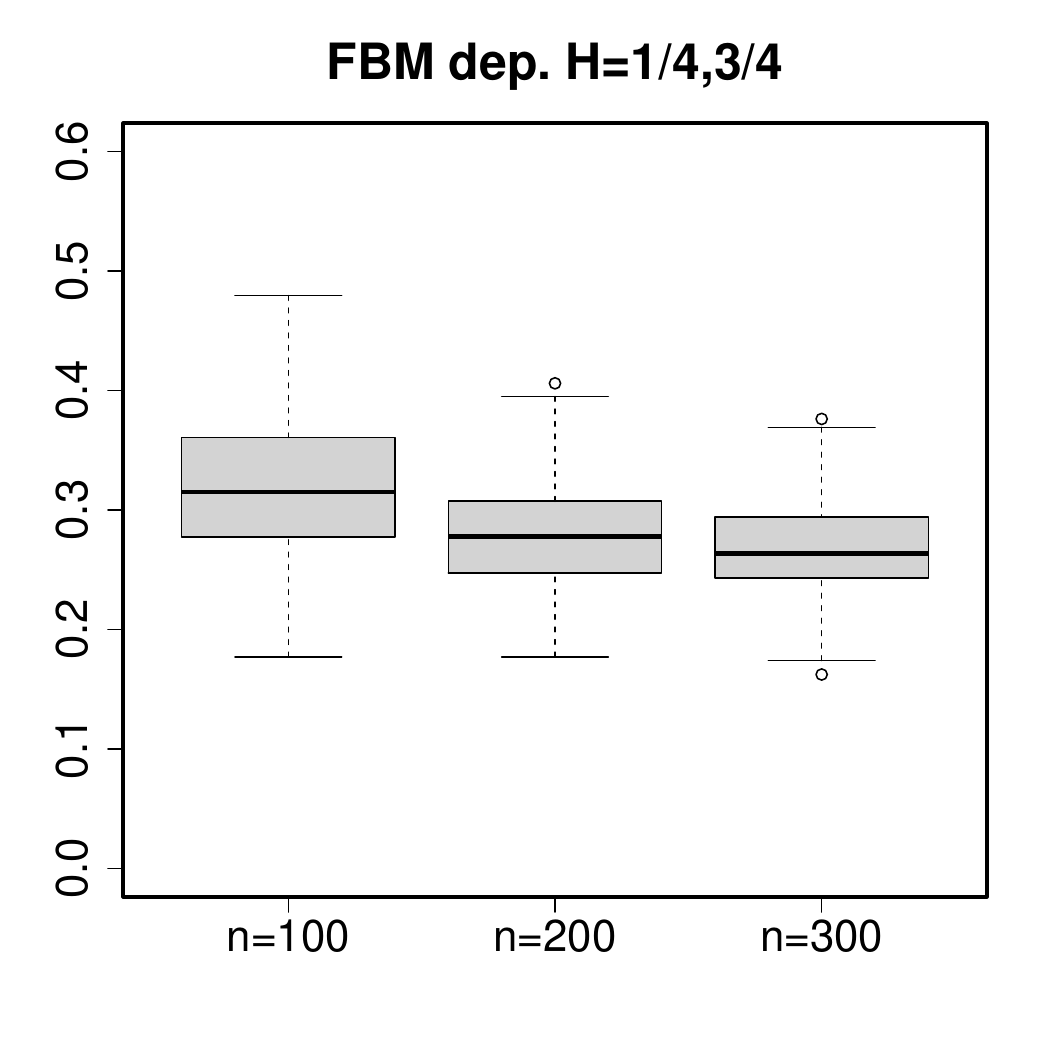}\\
\multicolumn{3}{c}{\includegraphics[width=.89\textwidth]{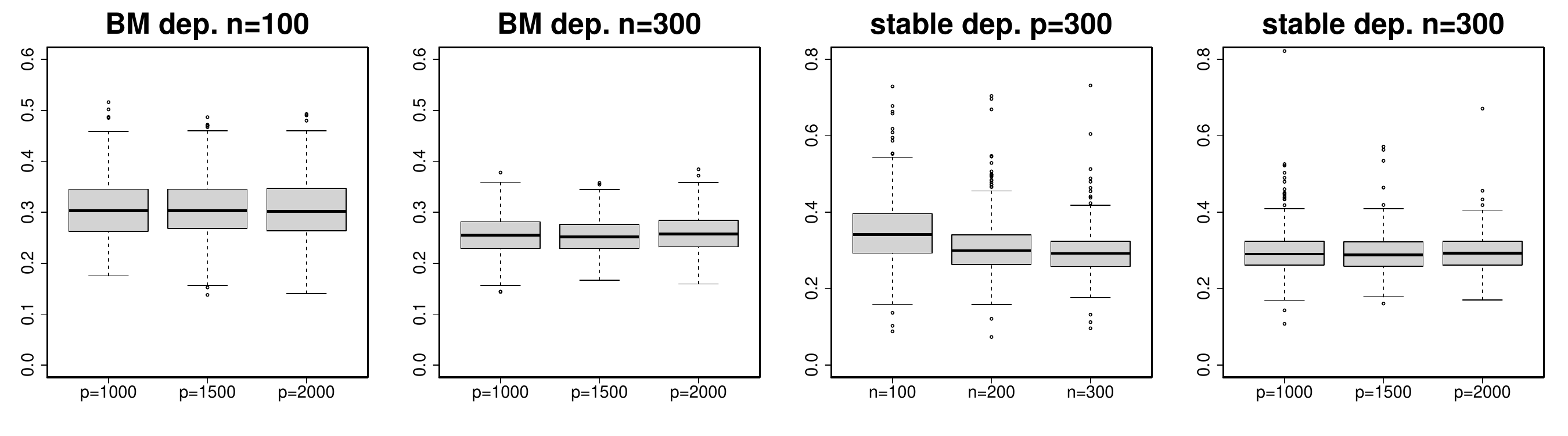}}\\
\end{tabular}
\bfig\label{fig:4}{\rm \small Boxplots for $R_n(X^{(p)},Y^{(p)})$ from 500 replications of 
dependent fB and $1.8$-stable sheets $X,Y$,
 respectively. We consider independent copies of $(X,Y)$ at a random Poisson $N_p$
number (with parameter $p$) of iid uniformly distributed locations $(\bfU_i)$
on $(0,1)^2$. 
The parameters $H_1$ and $H_2$, the values $n,p$ and the dependence parameter
$\rho=0.5$
are the same as for the corresponding graph at the same location 
in Figure~\ref{fig:3}.}\efig
\end{figure}
For the simulation of fB sheets we calculate their covariance function
\begin{align*}
 \cov(B^H(\bfs), B^H( \bft))=\prod_{i=1}^d \frac{1}{2}(
 |s_i|^{2H_i}+|t_i|^{2H_i}-|s_i-t_i|^{2H_i}),\qquad \bfs,\bft\in [0,1]^d\,,
\end{align*}
 at the lattice points and use the multivariate Gaussian random generator
 in R (mvfast) while for Brownian (this is the case $H_1=H_2=0.5$) 
and stable sheets we use the 
independent and stationary
 increment property: for each cell in $[0,1]^2$  we draw 
properly scaled independent Gaussian or stable 
 random variables and sum up, starting from the
 origin until we reach the boundary of $[0,1]^2$. (We generated Brownian sheets
with both methods, leading to very similar results.)   
\par  
Figure \ref{fig:1} shows boxplots for $R_n(X^{(p)},Y^{(p)})$ based on simulations 
from independent  fields $X,Y$ on the described lattice. 
In the top and middle rows, $q=30$, hence $p=900$, and $X,Y$ are fB sheets 
for different choices of $H_1$ and $H_2$. We see the
 influence of the smoothness of the sample paths: the larger $H_i$, the
 smoother the sample paths and the closer $R_n(X^{(p)},Y^{(p)})$ to
 zero. In the bottom row we simulate from  Brownian and $1.8$-stable 
sheets and examine the effect of increasing $n\in \{100,200,300\}$ 
for fixed $q=100$ and of increasing  $q\in \{100,200,300\}$ for fixed $n=300$.
An increase of $n$ apparently improves the performance of $R_n(X^{(p)},Y^{(p)})$:
the larger $n$ the closer it is to zero. On the other hand, if we fix $n$
one hardly sees a change for different values of $p=q\times q$. 
It is surprising for us that the sample distance correlation for independent
$1.8$-stable sheets outperforms the corresponding sample distance correlation 
for independent fB sheets: one detects independence between $X$ and $Y$
already for medium sample sizes.
\par
In Figure \ref{fig:2} we confront the results of Figure~\ref{fig:1} 
with the sample distance correlation for some dependent $X,Y$, again on the lattice. We consider iid fB sheets $X,X'$ and define $Y\eqd X$ by
\beam\label{def:depBsheets}
Y=\rho\, X +(1-\rho^{2})^{1/2}\, X'\,, \qquad \rho\in (0,1)\,.
\eeam
This choice of $X,Y$ leads to the correlation $\rho$ between $X$ and $Y$. In Figure \ref{fig:2} we choose $\rho=0.5$.
We present boxplots for the same choices of $H_1$ and $H_2$ 
as in the top and middle rows of Figure~\ref{fig:1}. 
The bottom row graphs  are based on dependent Brownian/$1.8$-stable sheets  
for the same $n$ and $p$ as at the bottom of  Figure \ref{fig:1}. 
We consider iid $1.8$-stable symmetric \levy\ sheets $X,X'$ and define $Y$ by
\beam
\label{def:depstablesheets}
Y=\rho\, X +(1-\rho^{1.8})^{1/1.8} X'\,,\qquad \rho\in (0,1)\,.
\eeam
Then in particular,  $X\eqd Y$. In  Figure~\ref{fig:2} we choose $\rho=0.5$.
The graphs in Figure~\ref{fig:2} are in stark contrast to 
those in Figure~\ref{fig:1}; they clearly  point at  the dependence
of $X,Y$. Again, the value of $p$ seems irrelevant.
\par
We also examine boxplots in the random location setting. 
In the lattice case in the top and middle 
graphs of Figures \ref{fig:1} and \ref{fig:2} we had $p=900$ points in $[0,1]^2$. 
 In the random location case
we choose uniform locations $\bfU_i$ on $[0,1]^2$ whose number $N_p$ is
Poisson with parameter $p=1000$, i.e., on average there are 1000 points.
In Figures~\ref{fig:3} and \ref{fig:4} we keep the same parameters and sample sizes
as in the previous two figures. 
In Figure~\ref{fig:3} we illustrate the case of independent $X,Y$.
Comparing Figures \ref{fig:3} and \ref{fig:1}, we
observe similar finite sample behavior: the sample size $n$ is more relevant for
convergence to zero than the number $N_p$ and the random locations. The bottom graphs
in Figure~\ref{fig:3} look less convincing (the median of the boxplot
is higher than in the lattice case) but this is due to the fact that 
we do not have the information from 
$100\times 100=10000$ locations but only from $N_p\approx 1000$.
(The choice of a Poisson variable $N_p$ with parameter $p=10000$ leads to 
a complexity which we could not handle on our laptops.)
Figure~\ref{fig:4} corresponds to the setting of dependent $X,Y$ in 
Figure~\ref{fig:2} with random locations. The graphs in both figures show quite
convincingly the difference between dependence and independence of $X$ and $Y$.
There is one significant difference to Figures~\ref{fig:1} and \ref{fig:2}:
the \ds\ of $R_n(X^{(p)},Y^{(p)})$ in Figures~\ref{fig:3} and \ref{fig:4}
is more spread than in the lattice case. This is due to the additional
uncertainty of the random locations.

\subsection{Size and power  of the distance correlation test}
\label{subsxec:simulate:bootstraptest}
We illustrate the performance of 
the bootstrap procedure for the test for independence 
based on distance correlation in the cases of fixed locations on a lattice and
of randomly scattered locations. We focus on independent pairs $X,Y$,
Brownian or $1.8$-stable sheets.  
Given a sample
$(X_1^{(p)},Y_1^{(p)}),\ldots,(X_n^{(p)},Y_n^{(p)})$,  
we draw $500$ bootstrap samples. From each bootstrap sample
we calculate the sample distance correlation and 
from the corresponding bootstrap 
\ds\ the $(1-\xi)$-quantile $q^\ast_{1-\xi}$. Finally, we verify whether 
\beam\label{eq:size}
R_n(X^{(p)},Y^{(p)}) \ge
q^\ast_{1-\xi}\,.
\eeam
Then we repeat this procedure $M\in \{500,1000\}$ times and count the 
successes of \eqref{eq:size}.
\par
In Table \ref{tab:sizel} we choose $q=100$ in the lattice case and observe 
the empirical rejection rates for independent Brownian and 
$1.8$-stable sheets $X,Y$:
each cell of the table corresponds to a given sample size $n$, test level
$\xi$ and iteration number $M$. 
In Table~\ref{table:powerl} we consider the simulation results
in the lattice case for dependent standard Brownian sheets 
with correlation $\rho\in (0,1)$ and $M=500$, where the correlation $\rho$ 
is accomplished through \eqref{def:depBsheets}. 
We also consider dependent $1.8$-stable L\'evy sheets $X,Y$ 
with dependence parameter $\rho$ introduced 
in \eqref{def:depstablesheets}. 
In agreement with the theory, 
the rejection rates increase as $n$ and $\rho$ increase.

\setlength{\tabcolsep}{10pt}
\begin{table}[htpb]
\begin{center}
\begin{tabular}{|c||c|c|c|c||c|c|c|c|}
\hline & \multicolumn{4}{c||}{Brownian sheets} &
 \multicolumn{4}{c|}{1.8-stable sheets} \\
\hline
$M$ & \multicolumn{2}{c|}{$500$} & \multicolumn{2}{c||}{$1000$} &  \multicolumn{2}{c|}{$500$} & \multicolumn{2}{c|}{$1000$} \\
\hline
$n\setminus \xi$   & $0.1$ & $0.05$ & $0.1$ & $0.05$ & $0.1$ & $0.05$ & $0.1$ & $0.05$ \\
\hline
100  & 12.8 & 6.2 &  11.6  & 6.2 & 11.4 &  3.6 & 9.2& 2.6 \\
200  & 9.4 & 4.8 &  9.6  & 4.8 & 8.2 & 3.4 & 8.4 & 4.8 \\
300  & 8.8 & 4.4 & 11.6 & 6.0  & 9.2 & 3.8 & 7.0& 3.2  \\
400  & 12.8 & 7.0 & 12.2 & 7.0  & 8.6 & 4.2 & 9.6 & 4.2 \\
\hline
\end{tabular}
\end{center}
\btab 
\label{tab:sizel} {\rm\small {\bf Bootstrap size of $n\,R_n(X^{(p)},Y^{(p)})$: lattice case.} Empirical rejection rates of bootstrap test based on $nR_n(X^{(p)},Y^{(p)})$ for independent Brownian and $1.8$-stable sheets  $X,Y$ with $M=500,1000$ iterations. We choose  $q=100$.}
%\end{minipage}
\etab 
\end{table}
Table~\ref{tab:sizer} shows the empirical rejection rates 
for independent Brownian and 
$1.8$-stable sheets $X,Y$ in the random observation setting. 
We generate the Poisson number $N_p$ and iid uniform locations
$(\bfU_j)_{j\le N_p}$. To generate the discretized 
Brownian sheets we calculate the correlation at the given random lattice and
 use the multivariate Gaussian random generator in R (mvfast) and 
choose $p=1000$. 
For generating a stable sheet at random locations, we first proceed as 
in the lattice case, generating independent stable random variables for 
 each cell of the random lattice, where the lattice is constructed  by 
cutting $[0,1]^2$ at all marginal points of $(\bfU_j)$. 
Then we calculate the sheet at $\bfU_j$ by summing the cells, starting 
from the origin. The computational complexity for doing this is high 
and therefore
we restrict ourselves to the smaller Poisson parameter $p=500$. 
As discussed in %Section \ref{sec:bootstr} 
Section 4 %in the main document \citep{mmrt:main} 
the bootstrap of $Z^{(p)}=(X^{(p)},Y^{(p)})$ is conducted conditionally on $N_p$. In Tables~\ref{tab:sizer} and \ref{table:powerr} we present results for $N_{500}=492$ and $N_{1000}=982$. 
We have examined the bootstrap procedure for
several realizations of $N_p$ with $p=500,1000$, but we did not find 
differences in the performance.
\par
In Table \ref{table:powerr} we illustrate 
the power of the test in the random observation case. 
The dependence parameter $\rho$ and the sample size $n$ 
are the same as in Table~\ref{table:powerl}. 
Again, $p=1000$ ($p=500$) are the Poisson parameters of  $N_p$ 
for Brownian (stable) sheets.
When comparing Tables~\ref{table:powerr} and \ref{table:powerl}, 
the empirical powers are quite similar, despite the different models for  
$(X^{(p)},Y^{(p)})$.
%{\small
%\btab \label{tab:sizel} {\rm\small {\bf Bootstrap size of $n\,R_n(X^{(p)},Y^{(p)})$: lattice case.}
%%with independent BMs and $\alpha$-stable LMs
%\begin{minipage}{12cm}
%Empirical rejection rates of bootstrap test based on $nR_n(X^{(p)},Y^{(p)})$ for
%independent Brownian and $1.8$-stable sheets 
% $X,Y$ with $M=500,1000$ iterations. We choose  $q=100$.}
%\end{minipage}
%\etab}

{\small 
\setlength{\tabcolsep}{8pt}
\begin{table}[htbp]
\begin{center}
\begin{tabular}{|c||c|c|c|c|c|c||c|c|c|c|c|c|}
\hline & \multicolumn{6}{c||}{Dependent Brownian sheets}
 &\multicolumn{6}{c|}{Dependent 1.8-stable \levy\ sheets} \\
\hline $n$ & \multicolumn{2}{c|}{$100$} & \multicolumn{2}{c|}{$200$} &  
\multicolumn{2}{c||}{$300$} &   \multicolumn{2}{c|}{$100$} &
 \multicolumn{2}{c|}{$200$} & \multicolumn{2}{c|}{$300$} \\
\hline $\rho\setminus \xi$ & $0.05$ &$0.1$ & $0.05$ & $0.1$ & $0.05$ &$0.1$ &$0.05$ &$0.1$  &$0.05$ &$0.1$ &$0.05$ &$0.1$ \\
\hline
0.1  & $15.2$ & $29.2$  & $28.7$  & $38.2$ & $36.8$ & $51.0$ & $14.2$
 & $25.6$  & $27,2$
  & $42.2$ & $40.0$  & $53.6$ \\
0.2  & $50.0$  & $65.0$ & $80.6$  & $87.6$  &$93.2$ & $96.4$  & $46.4$
 & $63.2$  & $76.0$
  & $86.4$ & $87.6$ & $92.8$ \\
0.3  & $85.8$  & $93.6$  & $98.6$  & $99.0$  & $100$ & $100$  & $79.4$
 & $86.8$  & $94.0$
  & $94.8$ & $95.6$ & $96.8$  \\
0.4  & $98.8$  & $99.6$  & $100$  & $100$  & $100$  & $100$  & $90.2$
 & $93.6$
  & $95.8$  & $96.8$ & $97.0$  & $98.0$ \\
0.5  & $100$ & $100$ & $100$ & $100$ & $100$ & $100$ & $94.0$ & $95.4$
 & $96.8$
 & $97.4$ & $98.0$ & $98.6$ \\
\hline
\end{tabular}
\end{center}
\btab\label{table:powerl}
{\small \rm {\bf Bootstrap power of $nR_n(X^{(p)},Y^{(p)})$ with dependent stable sheets: lattice case.}
%\begin{minipage}{12cm}
Empirical rejection rates of bootstrap test based on
 $nR_n(X^{(p)},Y^{(p)})$ for dependent Brownian sheets and
 $1.8$-stable sheets.   We choose $q=100$.
 In each cell of the table the rejection rate is calculated from
 $M=500$ iterations.}
%\end{minipage}
\etab
\end{table}
{\small
\setlength{\tabcolsep}{12pt}
\begin{table}[h]
\begin{center}
\begin{tabular}{|c||c|c||c|c|}
\hline & \multicolumn{2}{c||}{Brownian sheets} &
 \multicolumn{2}{c|}{1.8-stable sheets
} \\
\hline
$n$& $ \xi=0.1$ & $ \xi=0.05$ & $ \xi=0.1$ & $\xi=0.05$ \\
\hline
% $n$  &  %$nR_n \le q_{\alpha}^\ast$ &  
%$ n\,R_n\ge
%	 q_{1-\alpha}^\ast$ & %$nR_n^\ast \le q_{\alpha/2}^\ast$
%		 $n\,R_n \ge q_{1-\alpha}^\ast$  \\
%\hline
100  & 12.2  & 5.4  & 10.6 & 4.0 \\
200  &11.4 & 6.0 &  8.6  & 4.0 \\
300  &10.6 & 3.6 & 8.5   & 4.4 \\
400  & 8.8 & 5.6 & 9.0 & 3.8 \\
\hline
\end{tabular}
\end{center}
\btab\label{tab:sizer}
{\small \rm {\bf Bootstrap size of $nR_n(X^{(p)},Y^{(p)})$: random location case.}
%\begin{minipage}{12cm}
Empirical rejection rates of bootstrap test based on $nR_n(X^{(p)},Y^{(p)})$ for
independent Brownian and $1.8$-stable sheets 
 $X,Y$ with $M=500$ iterations. There is a Poisson number $N_p$ of 
iid uniform locations on $(0,1)^2$, the Poisson parameter  is $p=1000$ ($p=500$)
for Brownian (stable) sheets.
}\etab 
%\end{minipage}
\end{table}}
{\small
\begin{table}[h]
\begin{center}
\begin{tabular}{|c||c|c|c|c|c|c||c|c|c|c|c|c|}
\hline & \multicolumn{6}{c||}{Dependent Brownian sheets}
 &\multicolumn{6}{c|}{Dependent 1.8-stable sheets} \\
\hline $n$ & \multicolumn{2}{c|}{$100$} & \multicolumn{2}{c|}{$200$} &  
\multicolumn{2}{c||}{$300$} &   \multicolumn{2}{c|}{$100$} &
 \multicolumn{2}{c|}{$200$} & \multicolumn{2}{c|}{$300$} \\
\hline $\rho\setminus \xi$ & $0.05$ &$0.1$ & $0.05$ & $0.1$ & $0.05$ &$0.1$ &$0.05$ &$0.1$  &$0.05$ &$0.1$ &$0.05$ &$0.1$ \\
\hline
0.1  & $17.0$ & $27.0$ & $26.2$ & $37.2$  & $38.4$  & $50.2$ & 
 $14.0$ &
 $26.8$  & $26.2$ & $42.0$ & $39.8$  & $52.4$ \\
0.2  & $52.2$  & $67.2$  & $79.6$  & $88.4$ & $95.4$ & $98.8$ & 
 $47.0$ & $63.0$  &
$77.4$  & $85.4$  & $88.4$  & $94.4$ \\
0.3  & $88.2$  & $93.0$ & $99.0$ & $99.4$  & $100$  & $100$ &
 $81.6$ & $88.4$ & $92.6$ 
  & $94.6$  & $96.2$ & $97.4$ \\
0.4  & $99.4$ & $99.6$ & $100$ & $100$ & $100$ & $100$ & 
 $92.8$ & $94.4$
 & $95.6$ & $96.8$ & $97.8$  & $98.0$ \\
0.5  & $100$ & $100$ & $100$  & $100$ & $100$  & $100$  & $94.4$  & $95.4$ 
 & $96.8$
 & $97.4$ & $98.2$ & $98.2$ \\
\hline
\end{tabular}
\end{center}
\btab\label{table:powerr}
{\small \rm {\bf Bootstrap power of $nR_n(X^{(p)},Y^{(p)})$: random location case.}
%\begin{minipage}{12cm}
Empirical rejection rates of bootstrap test based on $nR_n(X^{(p)},Y^{(p)})$ for
dependent Brownian and $1.8$-stable sheets. 
The number $N_p$ is Poisson distributed with 
  $p=1000$ ($p=500$) for Brownian (stable) sheets.
 In each cell of the table the rejection rate corresponds to
 $M=500$ iterations.} 
%\end{minipage}
\etab
\end{table}} 

\section{ An application to Japanese meteorological data}\label{SecEmp}
We apply our results to Japanese meteorological data. 
We choose the 3 most fundamental factors: 
{\em temperature (temp), precipitation (prec) and wind speed (wind)} which have 
been observed for a long time and over a wide range of Japan. 
The data are available in various formats at the web-page of the Japanese Meteorological Agency \mbox{$\sc https://www.data.jma.go.jp/gmd/risk/obsdl$}.
Since daily data include many zeros and can be sparse, especially for precipitation, monthly average data are taken. From January 1980 to January 2021 we have 
$493$ monthly data at $783$ observation stations. There exist more such points in Japan. They are, however,  subject to 
problems such as change of position, many missing data, or only precipitation is observed. We removed these points, but still have plenty of points left; 
see the red dots in the map of Japan in Figure \ref{fig:one}.
Missing values are observed at less than $30$ stations, corresponding to
less than $10$ out of $493$ months in total. A missing value at a station is replaced 
by the average value at the station.  
\begin{figure}[h!]
\begin{tabular}{c}
\begin{minipage}{0.5\hsize}
\begin{center}
\includegraphics[width=55mm]{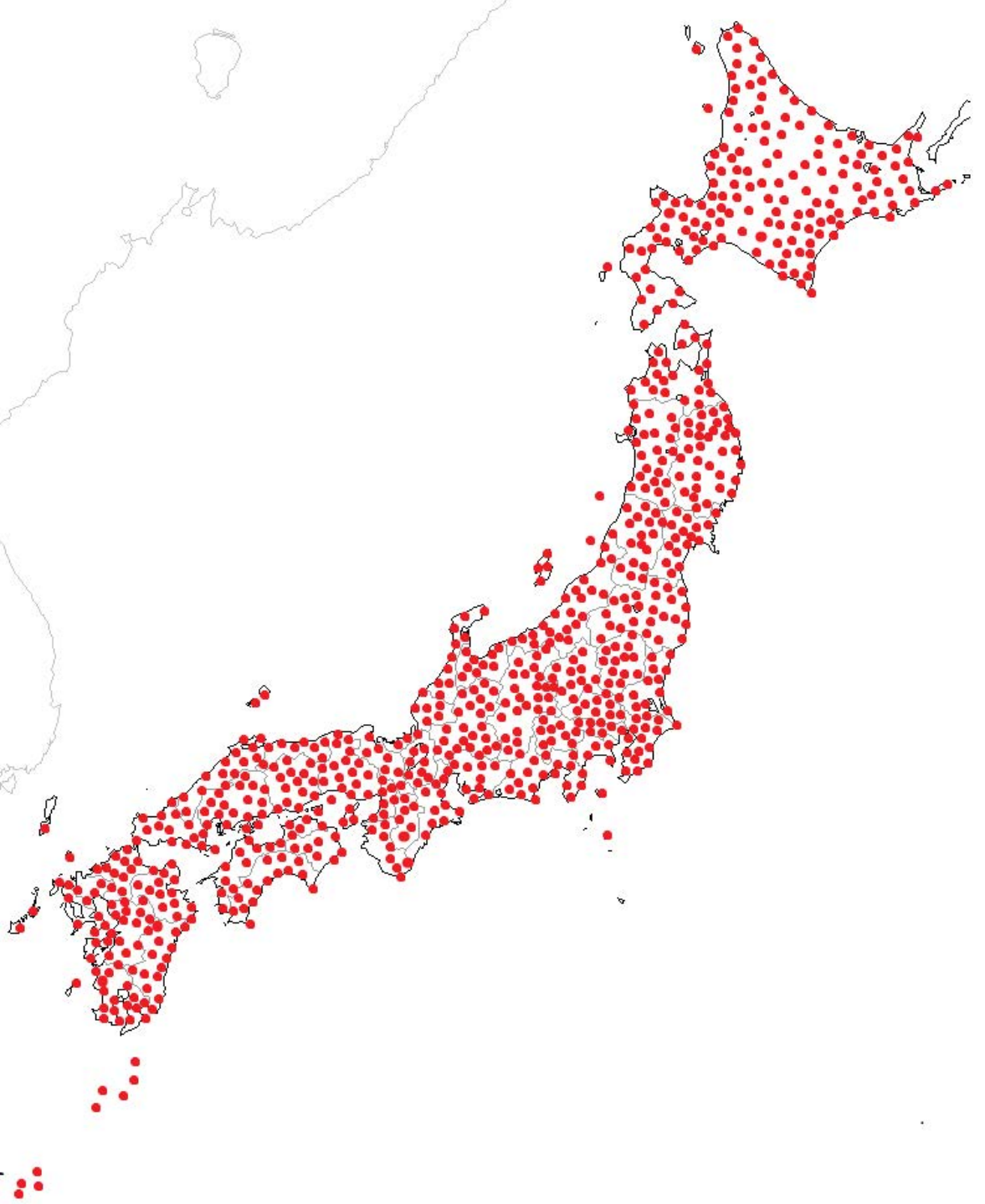}
\end{center}
\bfig\label{fig:one}
{\rm \small Distribution of $783$ meteorological observation points scattered all over Japan.}\efig
\end{minipage}
\begin{minipage}{0.5\hsize}
\begin{center}
\includegraphics[width=62mm]{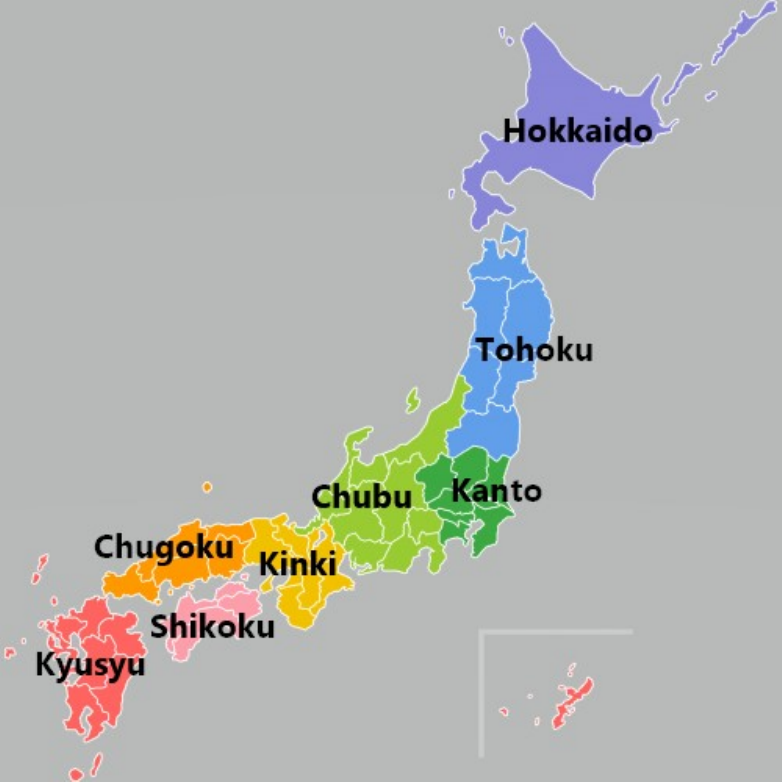}
\end{center}
\bfig\label{fig:two}{\rm \small Japanese $8$ regions}\efig
 \end{minipage}
\end{tabular}
\end{figure}

In our analysis we focus on the distance correlations for the pairs (prec \& temp), (prec \&  wind), (temp \&  wind) denoted by $R_{pt}$, $R_{pw}$, $R_{wt}$. 
In a preliminary analysis we conduct pair-wise independence tests between 
temp, prec and wind  
{\em at each station.} Here we use the classical sample 
distance correlation/correlation coefficients (Peasen, Spearman, and Kendall) between two components of a random vector.
At each station we use the time series of $493$ monthly observations and 
conduct bootstrap tests for pair-wise independence based on $1000$ resamples. 
We follow the procedure described in Section \ref{sec:simulations} with 
significance levels $\xi\in \{0.05,0.01\}$. For other correlation coefficient tests, we use the
``corr.test'' R package. The results are reported in 
Table~\ref{table:japan:each}. Each cell contains the number of rejections
of the hypothesis of  
pair-wise independence between (prec \& temp), 
(prec \& wind), (temp \& wind). 
The hypothesis of independence is overwhelmingly rejected at the majority of
stations while the distance correlation based test 
tends to detect dependence more often than other correlation tests.
{\small 
\setlength{\tabcolsep}{11pt}
\begin{table}[h!]
\begin{center}
\begin{tabular}{|c|c|c|c|c|c|c|}
\hline 
rejection number & \multicolumn{2}{c|}{prec \& temp} & \multicolumn{2}{c|}{prec \& wind}  & \multicolumn{2}{c|}{temp \& wind} \\ \hline 
 methods  $\setminus\, \xi$  & 0.05 & 0.01& 0.05 &0.01 &0.05 & 0.01 \\  \hline  
Peason Corr. & 734 & 720  & 599 & 524& 689& 654 \\
Spearman Corr. & 731  & 710  & 599  & 539 & 691 & 662  \\
Kendall Corr. & 730  & 710  & 599 & 538 & 689 & 661\\
Dist. Corr. & 782  & 773  & 698  & 635 & 769& 748\\\hline 
\end{tabular}
\end{center}
\btab\label{table:japan:each}
{\small\rm Rejection numbers of pair-wise independence for the 783 stations.}\etab
\end{table}
\par

\begin{figure}[h!]
\begin{tabular}{c}
\begin{minipage}{0.95\hsize}
\begin{center}
\includegraphics[width=80mm]{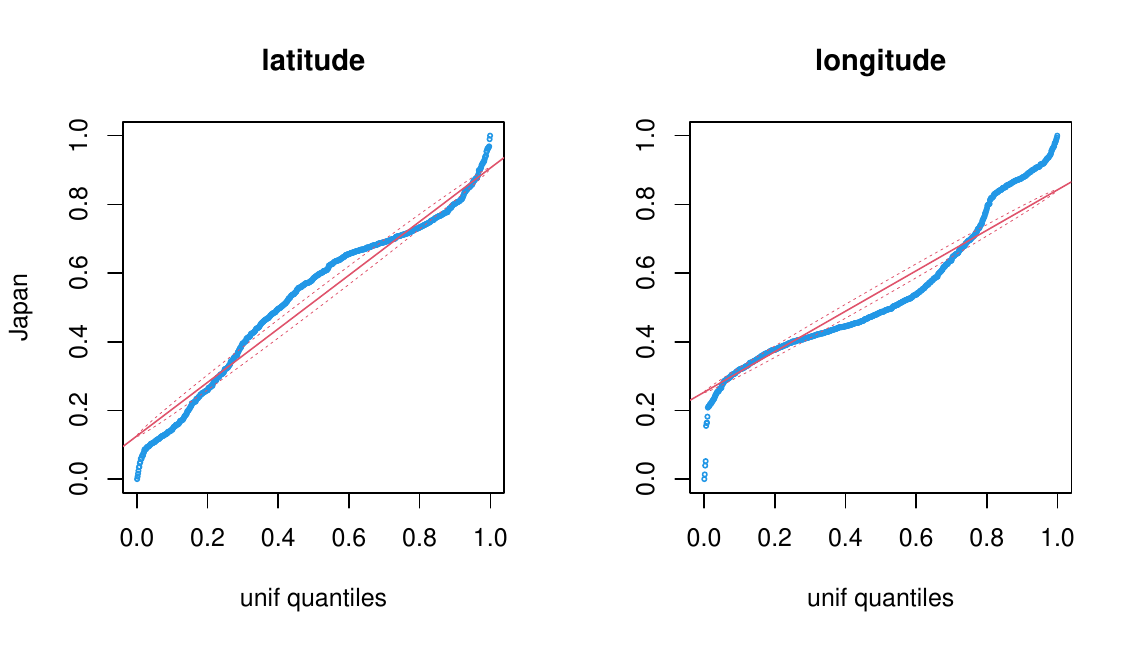}
\end{center}
\bfig\label{fig:three}
{\rm \small QQ-plots of uniform quantiles on $(0,1)$ against
the (re-scaled) latitude and longitude  
of all observation points in Japan. The red dotted lines indicate $95\%$ 
\asy\ confidence bands.}\efig
\end{minipage}
\end{tabular}
\begin{tabular}{c}
\begin{minipage}{0.95\hsize}
\begin{center}
\includegraphics[width=150mm]{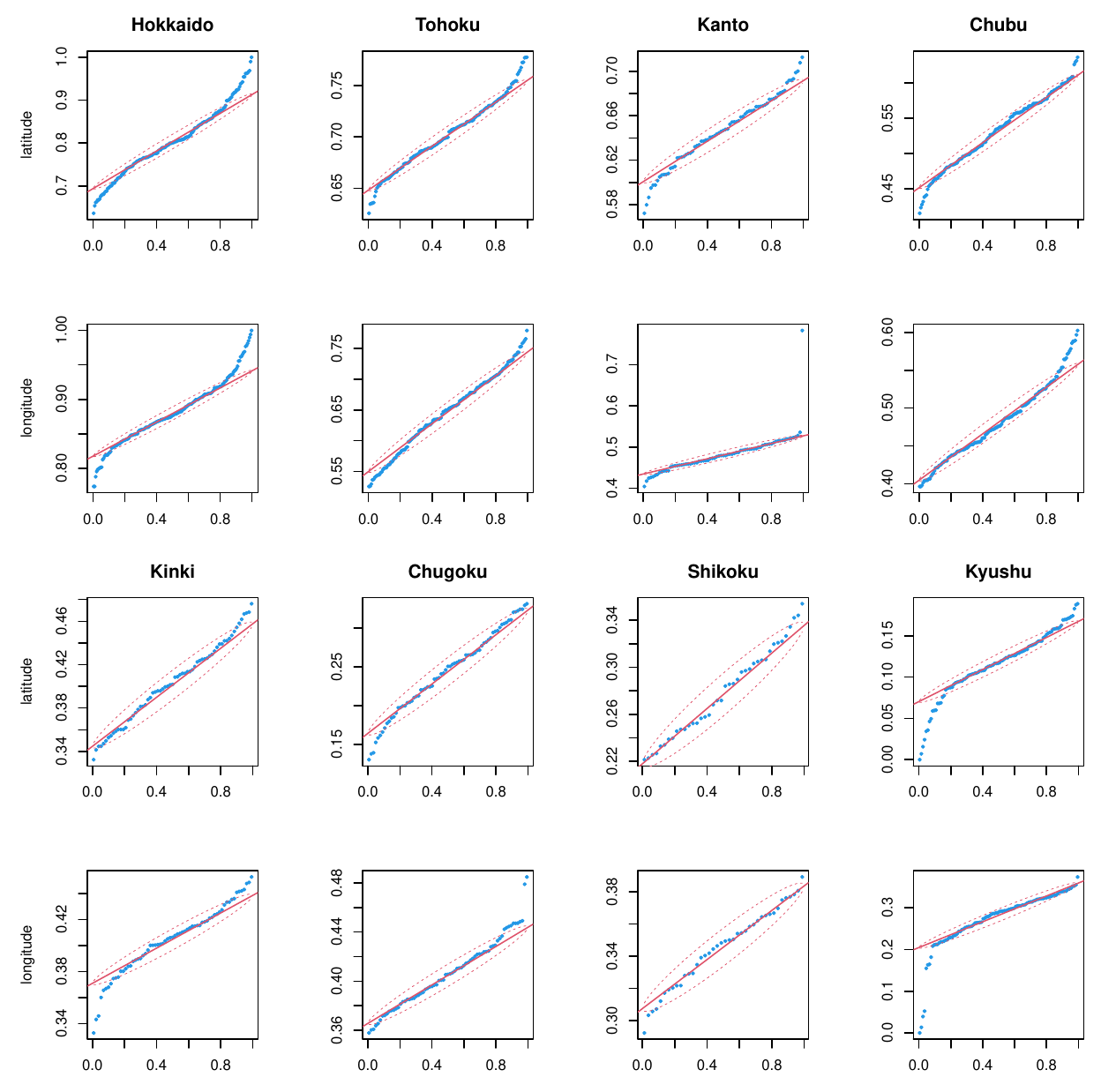}
\end{center}
\bfig\label{fig:four}
{\rm \small QQ-plots of uniform quantiles on $(0,1)$ 
against the (re-scaled) latitude and longitude  
of observation points in 8 Japanese regions. 
%of latitude and longitude of observation points in
%In two up and down cells bellow the name of regions 
%we draw QQ-plots of the latitude and longitude data. 
The red dotted lines indicate $95\%$ \asy\ confidence interval. 
}\efig
\end{minipage}
\end{tabular}
\end{figure}

\begin{figure}[h!]
\begin{tabular}{c}
\begin{minipage}{0.95\hsize}
\begin{center}
\includegraphics[width=150mm]{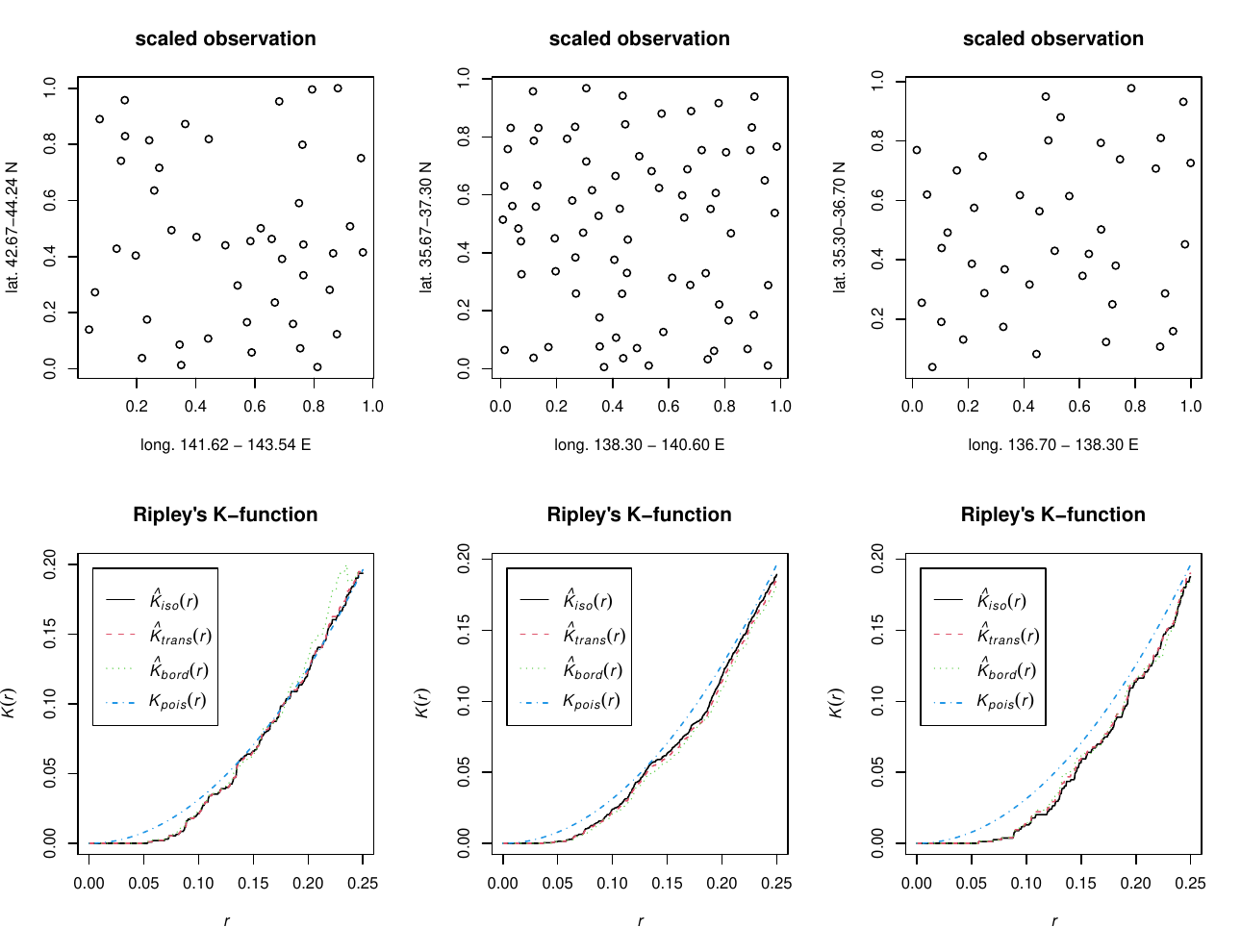}
\end{center}
\bfig\label{fig:five}
{\rm \small Ripley's $K$-function for three retangular regions
on the  Japanese islands, rescaled to the unit square.
Top: plots of observation points from specific ranges of latitude and longitude 
(indicated on the left and bottom sides 
of the graphs). Bottom: Ripley's $K$-functions of the corresponding regions. If data are randomly scattered the curves must be close to their limit
$r^2$ (light blue curves) as the number of points increases. The other tree curves, isotropic, translate and border are edge corrected versions.
%The other three curves are edge corrected versions.  
Notice that the estimation of K is affected by edge effects since points outside the window are unobservable (see \cite{kest:2022}). 
%{\green what do the different K functions mean? Wouldn't be one of them enough?}  Give a reference to the software used.
}\efig
\end{minipage}
\end{tabular}
\end{figure}
We interpret the monthly meteorological observations as realizations of 
a random field on the area of Japan.  
We regard the monthly data as iid random field observations and 
conduct independence tests on them. Figure~\ref{fig:one} shows quite convincingly
that the stations are scattered randomly all over Japan, so the assumption of
uniformly distributed locations is plausible. 
For further confirmation of the uniformity of the locations
we use some exploratory statistical tools.
For the marginal uniformity of all observations on Japan we considered
QQ-plots of the (re-scaled to $(0,1)$) latitude and longitude of all observations 
against the uniform \ds ; see Figure~\ref{fig:three}. 
The fit is not good due to the distorted diagonal shape of the 
Japanese islands, and there are large mountain areas and big lakes
with observation points. 
The marginal fit improves when considering 
QQ-plots in particular regions; see 
Figure~\ref{fig:four}. 
For checking bivariate uniformity we applied Ripley's $K$-functions \cite{ripley:1979} (see also \cite{kest:2022})
which are a convenient tool 
for measuring clustering and randomness of spatial data;
more sophisticated methods were described in  Section~\ref{subsec:randomlocation}.
We chose rectangular areas on the Japanese islands as wide as possible such that
they contain as many as possible points.  Then we re-scaled the rectangles to
unit square size. Three such regions are selected; see Figure ~\ref{fig:five}. 
In view of the plots and shapes of the $K$-function 
we may accept the assumption of uniformly distributed locations.
Therefore, we follow   
the approach of {\em random field at random locations} from 
Sections~\ref{subsec:randomlocation} and \ref{sec:randomlocation}.
We define distance as in  \eqref{eq:lln} and rely on 
Theorem \ref{thm:main} for the theoretical background on the independence test. 
The tests are based on the bootstrap procedure introduced 
in the simulation study (Section \ref{sec:simulations}); the theory for the bootstrap 
was developed in Section~\ref{asymptotic}.
\par
The first line in Table~\ref{table:japandc} 
yields the pair-wise sample distance correlations for the whole of Japan.
They clearly hint at dependence between the 3 factors, and this is also confirmed by bootstrap
tests of the sample distance correlation at different significance levels; see the right part of the table.
Next we study dependence locally in the distinct regions of Japan
and at different seasons of the year.
\par
Most parts of Japan are classified 
as humid subtropical climate area. However, the country 
stretches out from the far north to the far south, and the weather is
strongly influenced by the seasonal monsoon wind. The changes caused by the 
4 seasons are significant; cf. \cite{jmag:2021}. 
Therefore we also test for independence at regional level in each season.
The  seasonal samples consist of $4\times 123=492$ months, discarding the
last month.
Seasons are defined as common in the Northern hemisphere: March--May,  June--August, September--November,
December--February.  
\par
Table~\ref{table:japandc} provides the 
pair-wise (distance) correlations for annual and quarterly 
meteorological data for the  whole of Japan (Ja, 783 stations) as well as 
$8$ main regions (indicated in Figure~\ref{fig:two}):  Hokkaido (Ho, 159), 
Tohoku (To, 137), Kanto (Ka, 73), Chubu (Chb, 140), Kinki (Ki,65),  
Chugoku (Chg, 71), Shikoku (Shi, 40) and Kyusyu (Kyu, 98).
After the names of the regions and seasons
the pair-wise distance correlations 
$R_{\cdot}$ and the 
results of the independence tests based on the bootstrap for
$R_{\cdot}$ are presented for $1000$ resamples, following
the previous section. %in the supplementary material \citep{supplement:2}.
%The test for $r_{\cdot}$ is Pearson's correlation test, assuming
%the observations in the same region are mutually iid.}
%Notations in the right $6$ cells have the following meaning; 
%$\bigcirc$ implies that the independence hypothesis is rejected in both $\xi=0.05,\,0.01$ significance levels, 
%$\bigtriangleup$ implies that the hypothesis is rejected only in $\xi=0.05$, and $\times$ means that there are no rejections. 
{\small 
\setlength{\tabcolsep}{14pt}
\begin{table}[h!]
\begin{center}
\begin{tabular}{|c|l|rrr|ccc|}
\hline 
& (Dist.) Corr.  & $R_{pt}$ & $R_{pw}$ & $R_{tw}$  &$R_{pt}$ & $R_{pw}$ & $R_{tw}$   \\ \hline 
Ja & Year& .567 & .477 & .693 & $\bigcirc$ &$\bigcirc$ & $\bigcirc$  \\
&Spring & .220 & .272 & .354 & $\bigcirc$&$\bigcirc$&$\bigcirc$ \\
&Summer & .329 & .287 & .208 & $\bigcirc$&$\bigcirc$&$\bigcirc$ \\
&Autumn & .493 & .391 & .543 & $\bigcirc$&$\bigcirc$&$\bigcirc$ \\
&Winter & .211 & .327 & .192 & $\bigcirc$&$\bigcirc$&$\bigcirc$ \\ \hline 
Ho& Year & .404 & .272 & .655 & $\bigcirc$ & $\bigcirc$& $\bigcirc$ \\
&Spring & .226 & .217 & .351 & $\bigcirc$&$\bigcirc$&$\bigcirc$  \\
&Summer & .292 & .205 & .201 & $\bigcirc$&$\bigcirc$&$\bigcirc$ \\
&Autumn & .329 & .325 & .644 & $\bigcirc$&$\bigcirc$&$\bigcirc$ \\
&Winter & .259 & .309 & .140 & $\bigcirc$&$\bigcirc$&$\bigcirc$ \\ \hline 
To& Year & .398 & .338 & .668 & $ \bigcirc$&$ \bigcirc$&$ \bigcirc$ \\
&Spring & .121 & .146 & .297 & $\bigcirc$& $\times$ &$\bigcirc$  \\
&Summer & .165 & .222 & .233 & $\bigcirc$& $\bigcirc$ & $\bigcirc$ \\
&Autumn & .407 & .324 & .524 & $\bigcirc$& $\bigcirc$&$\bigcirc$  \\
&Winter & .228 & .240 & .111 & $\bigcirc$& $\bigcirc$& $\bigcirc$ \\ \hline 
Ka& Year & .471 & .308 & .539 & $ \bigcirc$&$ \bigcirc$&$ \bigcirc$  \\
&Spring & .125 & .154 & .344  & $\bigcirc$&$\bigtriangleup$ & $ \bigcirc$  \\
&Summer & .190 & .142 & .158  & $\bigcirc$&$\times$&$\bigcirc$  \\
&Autumn & .387 & .321 & .457 & $\bigcirc$&$\bigcirc$&$\bigcirc$ \\
&Winter & .086 & .176 & .202 & $\bigcirc$&$\bigcirc$&$\bigcirc$  \\ \hline 
Chb& Year & .394 & .366 & .679 & $ \bigcirc$&$ \bigcirc$&$ \bigcirc$  \\
&Spring & .159 & .184 & .371 & $\bigcirc$&$\bigcirc$&$\bigcirc$ \\
&Summer & .153 & .201 & .120 & $\bigcirc$&$\bigcirc$&$\bigcirc$ \\
&Autumn & .379 & .256 & .400 &$\bigcirc$&$\bigcirc$&$\bigcirc$  \\
&Winter & .131 & .308 & .182 &$\bigcirc$&$\bigcirc$&$\bigcirc$  \\ \hline 
Ki& Year& .429 & .211 & .426 &$ \bigcirc$&$ \bigcirc$&$ \bigcirc$  \\
&Spring & .144 & .156 & .301 &$\bigcirc$&$\bigtriangleup$&$\bigcirc$ \\
&Summer & .204 & .188 & .133 &$\bigcirc$&$\bigtriangleup$&$\bigcirc$  \\
&Autumn & .402 & .173 & .190 &$\bigcirc$&$\bigcirc$&$\bigcirc$  \\
&Winter & .135 & .233 & .208 &$\bigcirc$&$\bigcirc$&$\bigcirc$  \\ \hline 
Chg& Year & .311 & .225 & .514 &$\bigcirc$&$\bigcirc$&$\bigcirc$  \\
&Spring & .112 & .127 & .221 & $\bigcirc$&$\bigtriangleup$&$\bigcirc$ \\
&Summer & .165 & .151 & .161 & $\bigcirc$&$\bigcirc$&$\bigcirc$  \\
&Autumn & .322 & .179 & .303 &$\bigcirc$&$\bigcirc$&$\bigcirc$ \\
&Winter & .082 & .194 & .182 & $\bigtriangleup$&$\bigcirc$&$\bigcirc$ \\ \hline 
Shi& Year &.438 & .262 & .542 &$\bigcirc$&$\bigcirc$&$\bigcirc$ \\
&Spring & .146 & .193 & .341 &$\bigcirc$ &$\bigcirc$ &$\bigcirc$ \\
&Summer & .196 & .270 & .157 &$\bigcirc$ &$\bigcirc$&$\bigcirc$ \\
&Autumn & .414 & .195 & .233 & $\bigcirc$&$\bigcirc$&$\bigcirc$ \\
&Winter & .194 & .141 & .261 &$\bigcirc$ &$\bigcirc$&$\bigcirc$ \\ \hline 
Kyu& Year & .460&.211 & .354 &$\bigcirc$&$\bigcirc$&$\bigcirc$   \\
&Spring & .177 & .184 & .274 &$\bigcirc$ &$\bigcirc$&$\bigcirc$ \\
&Summer & .338 & .181 & .134 &$\bigcirc$&$\bigcirc$&$\bigcirc$ \\
&Autumn & .406 & .173 & .172 &$\bigcirc$&$\bigcirc$&$\bigcirc$ \\
&Winter & .193 & .161 & .189 &$\bigcirc$ &$\bigcirc$&$\bigcirc$ \\ \hline 
\end{tabular}
\end{center}
\btab\label{table:japandc}{\rm \small {\bf Correlation and distance correlation of Japanese meteorological data.}
%\begin{minipage}{17cm}
Pair-wise (distance) correlations for annual and quarterly 
meteorological data of the whole of Japan 
and $8$ main regions. Each row presents the corresponding 
region and seasons, followed by $3$ cells with the 
sample distance correlations and $3$ cells with the results of the 
independence tests. The symbols
$\bigcirc,\bigtriangleup,\times$ stand for rejection of the independence hypothesis at both 
significance levels $\xi=0.05,\,0.01$, only at $\xi=0.05$, 
no rejection, respectively.}\etab
%\end{minipage}
\end{table}}
We can extract several implications from the test results for 
the distance correlation and its magnitude which are also in agreement 
with the research of meteorologists; cf. \cite{jmao:2021}.
The pair-wise distance correlations of the $3$ factors in the whole of Japan  
clearly imply dependence of all weather patterns through the $4$ seasons.
Especially,  the values in autumn are rather high.  
The situation is somewhat different  if one  looks at each region. 
The $3$ factors still keep stable pair-wise dependence in autumn and 
winter in most regions  (an exception is the winter of Chugoku area), 
though the values of $R_{pt}$ in winter are 
rather small in the southern region. 
\par
The autumn values are relatively high in most regions, in agreement with
the values for the whole country.
Possible causes include the typhoon (tropical cyclones) 
attacks in August and October,  yielding a lot of precipitation 
and strong wind speed together with changes in temperature. 
Rain fronts in autumn can be another reason. 
While (prec \& temp) and (temp \& wind) indicate dependence in all areas 
through the $4$ seasons, 
the relation between (prec \& wind) shows independence in several regions, 
especially in spring and summer. 
One reason is that the {\em baiu front}, causing the  rainy season, 
is stationary and moves very slowly 
around in May--June. This fact may violate the dependence 
relations between the $3$ factors. 
%However, before drawing conclusions, additional 
%theoretical work from the meteorological side are needed. 
%\par
%{\red Move} {\green As for the analysis by the correlation, 
%since we regard the data at each station as that of iid 
%in the same region, we admit that the results by $r_{\cdot}$ are rough and 
%we could not simply compare the results with those by the distance correlation  
%\footnote{However, we do not have any correlation for bivariate continuous fields}. 
%Similarly, the relation of (prec \& wind) occasionally shows no-correlation in several spots 
%in Table \ref{table:japan:each}, 
%and their values are rather small in all $4$ seasons and in all regions. 
%The values of (prec \& temp) are relatively high in all cases.
%Especially, the negative correlation in Summer time fits our intuition. 
%Although the values are monthly averages, it interesting to note that 
%all the correlations of (temp \& wind) are positive in Winter.}  

%In sum, this empirical study  confirms that the 
%fundamental $3$ meteorological factors in Japan, 
%regarded as random field data, exhibit pair-wise dependence 
%in most patterns. 

\section{Proof of Proposition~\ref{prop::asymptotics1}}\label{sec:proofprop41}
\setcounter{equation}{0}
Write for any Riemann square-integrable random field $Z$ on $B$
independent of $(N^{(p)})_{p>0}$, and any $\beta\in (0,2]$,
\beao
D_{p}^{(\beta)}(Z)=\|Z^{(p)}\|_2^\beta -\|Z\|_2^\beta
= \Big(N_p^{-1} \int_B Z^2  \,dN^{(p)}\Big)^{\beta/2}-
\Big( \int_B Z^2(\bfu)  \,d \bfu\Big)^{\beta/2}\,.
\eeao
Then we have
\beam\label{eq:1a}
|D_p^{(\beta)}(Z)| \le |D_p^{(2)}(Z)|^{\beta/2}\,.
\eeam
We start by proving \eqref{app1distacecov1}. We have
\beam\label{eq:may8f}
 T_{n,\beta}(X^{(p)},Y^{(p)}) -T_{n,\beta}(X,Y)=I_1+I_2-2\,I_3, 
\eeam
where 
\beao
I_1 &=& \dfrac{1}{n^2} \sum_{k,l=1}^n
 \Big(\|X_k^{(p)}-X_l^{(p)}\|_2^\beta \,\|Y_k^{(p)}-Y_l^{(p)}\|_2^\beta
 -\|X_k-X_l\|_2^\beta\, \|Y_k-Y_l\|_2^\beta\Big)\,, \\
I_2 & =& \dfrac{1}{n^2} \,\sum_{k,l=1}^n \|X_k^{(p)}-X_l^{(p)} \|_2^\beta\, 
\dfrac{1}{n^2}\,
 \sum_{k,l=1}^n \|Y_k^{(p)}-Y_l^{(p)} \|_2^\beta- \dfrac{1}{n^2}\, 
\sum_{k,l=1}^n \|X_k-X_l\|_2^\beta \dfrac{1}{n^2} \,
\sum_{k,l=1}^n \|Y_k-Y_l\|_2^\beta, \\
I_3 &=&\dfrac{1}{n^3} \sum_{k,l,m=1}^n\Big(
 \|X_k^{(p)}-X_l^{(p)}\|_2^\beta\, \|Y_k^{(p)}-Y_m^{(p)}\|_2^\beta-\|X_k-X_l\|_2^\beta \|Y_k-Y_m\|_2^\beta\Big). 
\eeao
We have
\beao
 |I_1| &\le& \frac{1}{n^2} \sum_{k,l=1}^n |D_p^{(\beta)}(X_k-X_l)|
\, \|Y_k^{(p)}-Y_l^{(p)}\|_2^\beta  +\frac{1}{n^2} \sum_{k,l=1}^n |D_p^{(\beta)}(Y_k-Y_l)|\,
 \|X_k-X_l\|_2^\beta\\& =:&I_{11}+I_{12}. 
\eeao
Since the conditions on $X$ and $Y$ are symmetric  
it suffices to consider $I_{11}$. It will be convenient to set $Z=X_1-X_2$.
Since $((X_i,Y_i))$ are iid 
and independent of $N^{(p)}$ in view of \eqref{eq:1a}  we get the bound
\beao
 \E [I_{11}\mid N^{(p)} ]
&\le & \E\big[|D_p^{(\beta)}(Z)| \mid N^{(p)}\big]
\E\big[\|Y_1^{(p)}-Y_2^{(p)}\|^{\beta}_2\mid N^{(p)}\big]\\
&\le & \E \big[ |D_p^{(2)}(Z)|^{\beta/2}\mid N^{(p)}\big]\, \E\big[\|Y_1^{(p)}-Y_2^{(p)}\|^{\beta}_2\mid N^{(p)}\big]\,.
\eeao
Assume that $X,Y$ have finite variance. Then, by the \slln ,
\beao
\E \big[ |D_p^{(2)}(Z)|^{\beta/2}\mid N^{(p)}\big]
&\le &  \Big( 
\dfrac{1}{N_p} \sum_{i=1}^{N_p} \E[|Z|^2 (\bfU_i)\mid \bfU_i]
+  \int_B\E[|Z|^2 (\bfu)]\,d\bfu \Big)^{\beta/2}\\
&\stas & \Big( 2\,\int_B\E[|Z|^2 (\bfu)]\,d\bfu \Big)^{\beta/2}\,,\qquad p\to\infty\,.
\eeao
The \rhs\ is finite by assumption (1a).
If $X,Y$ have infinite variance and $\beta\in (0,2)$ then
\beao
 \E[|D_p^{(2)}(Z)|^{\beta/2} \mid N_p ] \le  2 \,\E\big[\sup_{\bfu \in B}|Z|^\beta(\bfu)\big]. 
\eeao
The \rhs\ is finite by assumption (1b).
%We observe that
%\textcolor{red}{
%\beao
%\E \big[ |D_p^{(2)}(Z)|^{\beta/2}\mid N^{(p)}\big]
%&\le &  \dfrac{1}{N_p} \sum_{i=1}^{N_p} \E[|Z|^\beta(\bfU_i)\mid \bfU_i]
%+  \int_B\E[|Z|^\beta(\bfu)]\,d\bfu\\
%&\stas & 2\,\int_B\E[|Z|^\beta(\bfu)]\,d\bfu\,,\qquad p\to\infty\,.
%\eeao}
%\textcolor{red}{
%``I could not derive the first inequality, can we put the power $\beta/2$ 
%inside the sum or integral ?''}
%In the last step we used the \slln .
%In both cases, the \rhs\ is finite by assumption. 
Conditional on $(N^{(p)})_{p>0}$
we have the \con\ %$|D_p^{(2)}(Z)|^{\beta/2} \stas 0\,,\, p\to\infty\,,$}
\beao
|D_p^{(2)}(Z)|^{\beta/2} \stas 0\,,\qquad p\to\infty\,,
\eeao
due the \con\ of the Riemann sums $\|Z^{(p)}\|_2^2\to 
\|Z\|_2^2$ for a.e. realization of $Z$. Then an application
of dominated \con\ yields %\textcolor{red}{$\E \big[ |D_p^{(2)}(Z)|^{\beta/2}\mid N^{(p)}\big]\stas 0\,.$}
\beao
\E \big[ |D_p^{(2)}(Z)|^{\beta/2}\mid N^{(p)}\big]\stas 0\,.
\eeao
A similar argument shows that $(\E\big[\|Y_1^{(p)}-Y_2^{(p)}\|^{\beta/2}\mid N^{(p)}\big])$ is bounded a.s. We conclude that
$\E \big[ |I_1|\mid N^{(p)}\big]\stas 0\,,\, p\to \infty\,.$
%\beao
%\E \big[ |I_1|\mid N^{(p)}\big]\stas 0\,,\qquad p\to
%\infty\,.
%\eeao
We can deal with $I_2,\,I_3$ in the same way by
 observing that 
\beao
 |I_2| &\le& \frac{1}{n^2} \sum_{k,l=1}^n |D_p^{(\beta)}(X_k-X_l)|\,
 \dfrac{1}{n^2} \sum_{k,l=1}^n \|Y_k^{(p)}-Y_l^{(p)}\|_2^\beta \\&& +\dfrac{1}{n^2} \sum_{k,l=1}^n 
  \|X_k-X_l\|_2^\beta 
  \frac{1}{n^2} \sum_{k,l=1}^n
 |D_p^{( \beta)}(Y_k-Y_l)|, \nonumber
\eeao
and 
\beao
 |I_3| &\le& \frac{1}{n^3} \sum_{k,l,m=1}^n |D_p^{(\beta)}(X_k-X_l)|
\, \|Y_k^{(p)}-Y_m^{(p)}\|_2^\beta + \frac{1}{n^3} \sum_{k,l,m=1}^n  
 |D_p^{(\beta)}(Y_k-Y_m)|\,\|X_k-X_l\|_2^\beta. \nonumber\\
\eeao
We omit further details.
\par
Next we deal with the case $X=Y$. 
Again using the decomposition \eqref{eq:may8f} and writing $Z=X_1-X_2$,
 we observe that 
\beao
\E[|I_{1}|\,\mid\,N^{(p)}]&\le&  \E[|D_p^{(\beta)}(Z)|\,(\|Z^{(p)}\|_2^\beta+\|Z\|_2^\beta)\,\mid\, N^{(p)}]\\
&\le &c\,\E\big[\|Z\|_2^{2\beta}+ \|Z^{(p)}\|_2^{2\beta}
%+ \|Z\|_2^\beta\|Z^{(p)}\|_2^\beta 
\,\mid \,N^{(p)}\big]\,.
\eeao 
Since we already know that $D_p^{(\beta)}(Z)\stas 0$ conditionally on $(N^{(p)})_{p>0}$  we intend to use dominated \con\ to show that the \lhs\ converges to zero a.s. Therefore we will bound the right-hand terms. First assume that $X$ has finite variance. 
We observe that
\begin{align*}
\E\big[\|Z\|_2^{2\beta} \mid N^{(p)}\big]= 
\E\Big[\Big(\int_B Z^2(\bfu)\,d\bfu\Big)^\beta\Big]
\le \left \{
\begin{array}{ll}
\Big(
\int_B \E[Z^2(\bfu) \,d \bfu]
\Big)^\beta& \text{for}\quad \beta\le 1\,, \\
\int_B \E[|Z|^{2\beta}](\bfu)\, d\bfu &\text{for}\quad \beta > 1\,. 
\end{array}
\right.
\end{align*}
The \rhs\ is finite by assumption (2a). If $X$ has infinite variance 
and $\beta\in (0,1)$ we have $\E\big[\|Z\|_2^{2\beta} \big]\le \E\big[\sup_{\bfu \in B}|Z|^{2\beta}(\bfu)\big]\,,$
%\beao
%\E\big[\|Z\|_2^{2\beta} \big]\le \E\big[\sup_{\bfu \in B}|Z|^{2\beta}(\bfu)\big]\,,
%\eeao
and the \rhs\ is finite by assumption (2b). The same bounds apply to 
%Similary, if $X$ has finite variance,  
\beao
\E\big[\|Z\|_2^{2\beta} \mid N^{(p)}\big] & = \E\big[
\big(
N_p^{-1} \sum_{i=1}^{N_p} Z^2(\bfU_i)
\big)^\beta \mid N^{(p)}
\big]\,.
% \\
%& \le \left \{
%\begin{array}{ll}
%\Big(
%\int_B \E[Z^2(\bfu) d \bfu]
%\Big)^\beta \wedge \E[\sup_{\bfu \in B}|Z|^{2\beta}(\bfu)] & \text{for}\quad \beta<1 \\
%\int_B \E[|Z|^{2\beta}(\bfu)] d\bfu &\text{for}\quad \beta \ge 1. 
%\end{array}
%\right.
\eeao
%We have \textcolor{red}{
%\beao
%\E\big[\|Z\|_2^{2\beta} \mid N^{(p)}\big]= 
%\E\Big[\Big(\int_B Z^2(\bfu)d\bfu\Big)^\beta\Big]
%\le  \Big(\int_B \E[|Z|^{2\beta}(\bfu)]d\bfu\Big)^{\beta\wedge 1} <\infty\,.
%\eeao}
%\textcolor{red}{How can we use concavity $\beta \le 1$, for this inequality ?}
%by assumption. Here we used concavity of the \fct\ $f(x)=|x|^\beta$ 
%for $\beta\le 1$ and Lyapunov's inequality of order $\beta$ for
%$\beta\in(1,2]$.
%Similarly, 
%\beao
%\E\big[\|Z^{(p)}\|_2^{2\beta} \mid N^{(p)}\big]&=&
%\E\Big[\Big(N_p^{-1}\sum_{i=1}^{N_p} Z^2(\bfU_i)\Big)^{\beta}\mid N^{(p)}\Big]\\
%&\le & \Big(N_p^{-1}\sum_{i=1}^{N_p} \E[|Z|^{2\beta}(\bfU_i)\mid \bfU_i]\Big)^{1\wedge \beta}\\
%&\stas &\Big(\int_B \E[|Z|^{2\beta}(\bfu)]\,d\bfu\Big)^{1\wedge \beta}<\infty\,,
%\eeao
%and 
%Moreover, by the Cauchy-Schwarz inequality,
%\beao
%\E\big[ \|Z\|_2^\beta\|Z^{(p)}\|_2^\beta \mid N^{(p)}\big]&\le &
%\Big(\E\Big[\Big(\int_B Z^2(\bfu) d \bfu \Big)^\beta\Big]\Big)^{1/2}\,\big(\E\b%ig[\|Z^{(p)}\|_2^{2\beta} \mid N^{(p)}\big]\big)^{1/2}\,.
%\eeao
%By definition all three terms are bounded a.s.
%whose \rhs\ is bounded a.s. } 
Thus for any $\vep>0$, $\P(|I_1|>\vep\,\mid\, N^{(p)})\to 0$ 
as $\nto$ along a.e. sample path of $(N^{(p)})$.
\par
We also have 
\beao
|I_2| & \le & \frac 1 {n^2}
\sum_{k,l=1}^n  |D_p^{(\beta)}(X_k-X_l)|\frac 1 {n^2}
\sum_{k,l=1}^n
\Big(\|X_k^{(p)}-X_l^{(p)}\|_2^\beta + \|X_k-X_l\|_2^\beta\Big)=:I_{21}\,I_{22}.
\eeao
We have already proved %\textcolor{red}{$\E[I_{21}^{(\beta)} \mid N^{(p)}] \le \E[|D_p^{(\beta)}(Z)|\mid N^{(p)}]\to 0\,.$}
\beao
\E[I_{21}^{(\beta)} \mid N^{(p)}] &\le& \E[|D_p^{(\beta)}(Z)|\mid N^{(p)}]\to 0\,.
\eeao
Moreover,
\beao
\E[I_{22}\mid N^{(p)}]\le \E[\|Z^{(p)}\|_2^\beta\mid N^{(p)}]+ \E[\|Z\|_2^\beta]\,.
\eeao
We have already shown that the \rhs\ is bounded for a.e. sample path
of $(N^{(p)})_{p>0}$. Therefore we have for any $\vep>0$ and a.e. sample
path of $(N^{(p)})_{p>0}$,
\beao
\P\big(|I_2|>\vep \mid N^{(p)}\big)\to 0\,,\qquad \nto\,.
\eeao
We have 
\begin{align*}
|I_3| \le \frac{1}{n^3} \sum_{k,l,m=1}^n |D_p^{(\beta)}(X_k-X_l)|
\, \big(\|X_k^{(p)}-X_m^{(p)}\|_2^\beta + \|X_k-X_l\|_2^\beta\big),
\end{align*}
hence by Cauchy-Schwarz,
\beao
\E\big[|I_3|\mid N^{(p)}\big]&\le & \E\big[|D_p^{(\beta)}(X_1-X_2)| \big(\|X_1^{(p)}-X_3^{(p)}\|_2^\beta+ \|X_1-X_3\|_2^\beta\big)\mid N^{(p)}\big]\\
&\le &\E\big[|D_p^{(2)}(X_1-X_2)|^{\beta/2} \big(\|X_1^{(p)}-X_3^{(p)}\|_2^\beta+ \|X_1-X_3\|_2^\beta\big)\mid N^{(p)}\big]\\
&\le &\big(\E\big[|D_p^{(2)}(Z)|^{\beta}\mid N^{(p)}\big]\big)^{1/2}
 \Big( \big(  \E\big[\|Z^{(p)}\|_2^{2\beta}\mid N^{(p)}\big] \big)^{1/2}+ 
\big(\E\big[\|Z\|_2^{2\beta}\big]\big)^{1/2}\Big)\,.
\eeao
The \rhs\ converges to zero by arguments similar to those above.
We finally conclude that
\[
 \P\big(|I_3|>\vep \mid N^{(p)}\big)\to 0\,,\qquad \nto\,.
\]
This finishes the proof of \eqref{eq:august19a}.

\section{Proof of Proposition \ref{prop:2}}\label{App:B}\setcounter{equation}{0}
\begin{lemma}
\label{lem:momentcondiap2}
$(1)$ Assume $\beta\in (0,2)$ and, if $X$ has finite second or 
$\beta$th moment, we assume {\rm (A1)} or {\rm (B1)}, 
respectively.
Then there exists a constant $c>0$ such that 
\begin{align*}
\sup_p \E [\|X^{(p)}\|^\beta_2\mid \bfU^{(p)}]\le c \qquad \as
\end{align*}
where $\bfU^{(p)}=(U_1,\ldots,U_p)$.\\
$(2)$ If $\beta\in (1,2)$ and 
$\max_{\bft\in B}\E[|X(\bft)|^{2(2\beta-1)}]<\infty$ 
then we have
$
\E[\|X\|^{2(2\beta-1)}_2]<\infty,
$
and there exists a constant $c>0$ \st
\begin{align*}
\sup_p \E[\|X^{(p)}\|^{2(2\beta-1)}_2\mid \bfU^{(p)}]\le c\qquad \as
\end{align*}
\end{lemma}

\begin{proof}
$\rm (1)$ {\em Assume {\rm (A1).}} Then $\sup_{\bft \in B}\E[X^2(\bft)]=c_0<\infty$. 
In view of \eqref{def:normapp2} and by Jensen's inequality we have
\beao
\E [\|X^{(p)}\|_2^\beta\mid \bfU^{(p)}] &=&
\E\Big[\big(\tilde p^{-1}\mbox{$\sum'$} 
\ov X_j^2 \big)^{\beta/2} \Big| \bfU^{(p)}\Big]\\
&\le &\E\Big[\Big(\tilde p^{-1}\mbox{$\sum'$} \sum_{\ell:\bfU_\ell\in \Delta_j}
 X^2(\bfU_\ell)/\#\Delta_j 
\Big)^{\beta/2}\Big|\bfU^{(p)}\Big] \\
&\le  &\Big(\tilde p^{-1} 
 \mbox{$\sum'$}\sum_{\ell:\bfU_\ell \in \Delta_j} \E[X^2(\bfU_\ell)\mid \bfU_\ell] /{\#\Delta_j} 
\Big)^{\beta/2} \le  c_0^{\beta/2}\qquad \as 
\eeao
{\em Assume} (B1). Then by similar arguments 
\beao
\E [\|X^{(p)}\|_2^\beta\mid \bfU^{(p)}] &\le &c\,\E\big[
\big(\sup_{\bft\in B} X^2(\bft)\big)^{\beta/2}
\big]
\le c\, \E\big[
\sup_{\bft\in B}|X(\bft)|^\beta
\big]<\infty\qquad \as 
\eeao
$(2)$ The first statement is immediate. Using Jensen's inequality, we also
have
\beao
\E[\|X^{(p)}\|^{2(2\beta-1)}_2\mid \bfU^{(p)}] & \le & 
\tilde p^{-1}\mbox{$\sum'$}
\sum_{k:\bfU_k \in \Delta_j} 
\E[|X(\bfU_k)|^{2(2\beta-1)}\mid \bfU_k ]/\#\Delta_j \\
& \le& c \,\max_{ \bft\in B} \E [|X(\bft)|^{2(2\beta-1)}] <\infty \qquad \as  
\eeao
\end{proof}
Now we proceed with the proof of the proposition.
 In what follows, it will be convenient
to write $Z=X_1-X_2$.
We use the decomposition \eqref{eq:may8f} %(6.2) of the main document \citep{mmrt:main} 
adjusted to the present situation.\\
{\bf 1.} {\em We assume} (A1).\\ 
{\em First we study $I_{1}$.} By a symmetry argument it suffices to 
consider $I_{11}$.\\
{\em Assume $\beta\in(0,1]$.}
Using the independence of $X,Y$ and [Lemma \ref{lem:momentcondiap2} (1)], we have\begin{align}
 \E[I_{11}] &= \E[ \E[I_{11}\mid \bfU^{(p)}] ] \nonumber \\
            &\le \E \big[ \E \big[|D_p^{(\beta)}(Z)| \,\big|\, \bfU^{(p)}
%\|X_1^{(p)}-X_2^{(p)}\|_2^\beta-\|Z\|_2^\beta| \mid (\bfU_k)
\big] \, \E[\|Y_1^{(p)}-Y_2^{(p)}\|_2^\beta \,\mid \,\bfU^{(p)}] \big] \nonumber \\
& \le c\, \E [|D_p^{(\beta)}(Z)
%\|X_1^{(p)}-X_2^{(p)}\|_2^\beta-\|X_1-X_2\|_2^\beta
|] 
\le c \,\E [\|Z^{(p)}-Z\|_2^\beta]\,, \label{ineq:I11betaless1}
\end{align}
where we used $\beta\in (0,1]$ in the last step. We have
\beam
\E [\|Z^{(p)}-Z\|_2^\beta]&=&\E \Big[
\Big(
\sum_{j=1}^{\tilde p} \int_{\Delta_j}(
\ov Z_j - Z(\bft))^2\,d\bft 
\Big)^{\beta/2}\Big] \label{secondmethodpfineq1} \\
& \le &
\E \Big[
\Big(
\mbox{$\sum'$} \int_{\Delta_j}(
\ov Z_j - Z(\bft))^2\,d\bft 
\Big)^{\beta/2}\Big]+\E \Big[ 
\Big( \sum_{j\le \tilde p: \#\Delta_j=0}\int_{\Delta_j}Z^2(\bft)\,d\bft 
\Big)^{\beta/2}\Big]\,.\nonumber\\ %\label{secondmethodpfineq2}\\
&=:&J_1+J_2\,.\nonumber 
\eeam
Then an application of Jensen's inequality yields
\beao
J_2&\le & 
\Big(\E\Big[ \sum_{j\le \tilde p: \#\Delta_j=0}\int_{\Delta_j}Z^2(\bft)\,d\bft 
\Big]\Big)^{\beta/2}\\
&=& \Big(\E\Big[ \sum_{j\le \tilde p: \#\Delta_j=0} \E\big[\int_{\Delta_j}Z^2(\bft)\,d\bft 
 \big| \bfU^{(p)} \big] \Big] \Big)^{\beta/2} \\
&\le &c\, \Big(\max_{\bft\in B} \var(X(\bft))\Big)^{\beta/2}
\,\Big(\E\Big[ \tilde p^{-1} \sum_{j\le \tilde p} \1(\#\Delta_j=0) \Big] \Big)^{\beta/2}\\
&\le &c\,\Big(\P\Big(\sum_{k=1}^p \1_{\Delta_1}(\bfU_k)=0 \Big)\Big)^{\beta/2}\\
&=& c\,\big((1- \tilde p^{-1})^p\big)^{\beta/2}\le c\,p^{-\beta/2}\,.
\eeao
Now we turn to $J_1$. Then applications of 
Jensen's inequality and (A1) yield
\beao
J_1
&\le &
\Big(\E\Big[
\mbox{$\sum'$} \int_{\Delta_j}(
\ov Z_j - Z(\bft))^2\,d\bft \Big]
\Big)^{\beta/2}\\
&\le & \Big(\E\Big[
\mbox{$\sum'$}  (\#\Delta_j)^{-1}\sum_{k:\bfU_k\in \Delta_j} \int_{\Delta_j}  
(Z(\bfU_k)- Z(\bft))^2 d\bft \Big]
\Big)^{\beta/2}\\
&=&\Big(\E\Big[
\mbox{$\sum'$}(\#\Delta_j)^{-1} \sum_{k:\bfU_k\in \Delta_j} \int_{\Delta_j}  
\E\big[(Z(\bfU_k)- Z(\bft))^2\mid \bfU_k\big] d\bft \Big]
\Big)^{\beta/2}\\
%&\le &\Big( 
%\max_{j=1,\ldots,\tilde p}\max_{\bfv\in \Delta_j}\,\E\big[(Z(\bfv)- Z(\bft))^2\big] d\bft
%\Big)^{\beta/2}\\
&\le &c\,
\max_{j=1,\ldots,\tilde p}\max_{\bfv,\bft\in \Delta_j}\,\big(\E\big[(X(\bfv)- X(\bft))^2\big] \big)^{\beta/2}\\
&\le &c\,|{\bf1}/\tilde p^{1/d}|^{\gamma_X\beta/2}= (\sqrt{d/\tilde p^{2/d}})^{\gamma_X\beta/2}=c\,\tilde p^{-\beta\gamma_X/(2d)}\,.
\eeao
We conclude that
\beao
\E[I_{11}]\le c\, (J_1+J_2)\le c\,\big(p^{-1}+  \tilde p^{-\gamma_X/d}\big)^{\beta/2}\,,
\eeao
and in turn
\beam\label{eq:nov23}
\E[|I_1|]\le c\,\big(p^{-1}+  \tilde p^{-(\gamma_X\wedge \gamma_Y)/d}\big)^{\beta/2}\,.
\eeam
{\em Assume $\beta \in (1,2)$.} 
Applying $|x^\beta-y^\beta|\le (x\vee y)^{\beta-1}|x-y|$ for positive $x,y$ and 
H\"older's inequality, we obtain
\beam
\E \big[ | D_p^{(\beta)}(Z) 
| \big] 
&\le& c\, \E\big[ \big(\|Z^{(p)}\|_2^{\beta-1} \vee \|Z\|_2^{\beta-1}\big)\, 
| D_p^{(1)}(Z) |
\big] \label{ineq:holder} \\
&\le & c\, \big(\E\big[ \|Z^{(p)}\|_2^{2} \vee \|Z\|_2^{2} 
| \big]\big)^{(\beta-1)/2} \times \big( \E \big[
\|Z^{(p)}-Z\|_2^{2/(3-\beta)} 
\big] \big)^{(3-\beta)/2} =: c\, P_1 P_2. \nonumber 
\eeam
Since we assume finite second moments via (A1) we have  $P_1<\infty$.
Moreover,
\beao
P_2^{2/(3-\beta)}&=&\E \big[
\|Z^{(p)}-Z\|_2^{2/(3-\beta)}\big] =
\E \Big[
\Big(
\sum_{j=1}^{\tilde p} \int_{\Delta_j}(
\ov Z_j - Z(\bft))^2\,d\bft \big)
\Big)^{1/(3-\beta)}\Big]\,.
\eeao
Observing that $(3-\beta)^{-1}<1$ we can use the same arguments as for
$\beta\in(0,1]$ to obtain the bound 
\beao
 P_2^{2/(3-\beta)} \le c\, \big(p^{-1}+\tilde p^{-\gamma_X/d}\big)^{1/(3-\beta)}. 
\eeao
Under (A1),
we conclude that we have the following upper bound for $\beta\in (1,2)$:
\beao
 \E[|I_1|] \le c\,(\tilde p^{-(\gamma_X \wedge \gamma_Y)/d}+p^{-1})^{1/2}. 
\eeao
Combining this bound with \eqref{eq:nov23}, we finally have for any $\beta\in (0,2)$,
\beao
\E[|I_1|] \le \, c\,(\tilde p^{-(\gamma_X \wedge \gamma_Y)/d}+p^{-1})^{(1\wedge \beta)/2}\,.
\eeao
{\em Now we turn to $I_2,I_3$.} We have
\beao
 |I_2| &\le& \frac{1}{n^2} \sum_{k,l=1}^n |D_p^{(\beta)}(X_k-X_l)|\,
  \frac{1}{n^2} \sum_{k,l=1}^n \| Y_k^{(p)}-Y_l^{(p)}\|_2^\beta 
 +\frac{1}{n^2} \sum_{k,l=1}^n 
  \|X_k-X_l\|_2^\beta 
  \frac{1}{n^2} \sum_{k,l=1}^n |D_p^{(\beta)}(Y_k-Y_l)| 
\eeao
and a similar bound holds for $|I_3|$. The same arguments as for $\E[|I_1|]$ 
yield 
\[
 \E[|I_2+I_3|] \le c \big(
\tilde p^{-(\gamma_X\wedge \gamma_Y)/d}+p^{-1}
\big)^{(1 \wedge \beta)/2}.
\]
We omit further details.\\[2mm]
{\bf 2.} Now we assume that $X,Y$ have finite $\beta$th moment for some $\beta\in(0,2)$
and (B1), (B2) hold.\\
{\em First we study $I_1$ and focus on $I_{11}$.} 
We follow the patterns of the proof under condition (A1) in the 
finite variance case.\\
{\em Assume $\beta\in (0,1]$.} We again write $Z=X_1-X_2$. In view of 
\eqref{ineq:I11betaless1} and [Lemma \ref{lem:momentcondiap2} (1)] it suffices to 
bound $J_1,J_2$.
If (B1) holds we can proceed as in part {\bf 1}.:
\beao
J_2&\le & 
\Big(\E\Big[ \tilde p^{-1} \sum_{j=1}^{\tilde p} \1(\#\Delta_j=0)\Big]\Big)^{\beta/2}\, \E\Big[\max_{\bft\in B}|Z(\bft)|^\beta\Big] \le c\,p^{-\beta/2}\,.
\eeao
We also have by (B2)
\beao
J_1&\le & c\,\tilde p^{-\beta/2}\,\sum_{j=1}^{\tilde p}\E\Big[\max_{\bfv,\bft\in\Delta_j}|X(\bft)-X(\bfv)|^\beta\Big]
\le  c\,\wt
p^{1-\gamma_X/d-\beta/2}\,.
\eeao
We conclude that
\begin{align}
\label{betamoment:I1}
 \E[|I_1|] \le c p^{1-\beta/2}(\tilde p^{-(\gamma_X\wedge \gamma_Y)/d} +p^{-1}).
\end{align}
Similar bounds hold for $I_2,I_3$. We omit details.\\
{\em Assume $\beta\in (1,2)$.} We start from \eqref{ineq:holder}:
\beao
\E [ | D_p^{(\beta)}(Z) %\|X_1^{(p)}-X_2^{(p)}\|_2^\beta-\|X_1-X_2\|_2^\beta 
| ] 
&\le& c\, \E\big[\big( \|Z^{(p)}\|_2^{\beta-1} \vee \|Z\|_2^{\beta-1}\big) 
| D_p^{(1)}(Z) | %\|X_1^{(p)}-X_2^{(p)}\|_2-\|X_1-X_2\|_2 |
\big] \\
& \le& c \, \big( \E\big[ \|Z^{(p)}\|_2^{\beta} \vee \|Z\|_2^{\beta}\big] \big)^{(\beta-1)/\beta}
\big( \E \big[
\|Z^{(p)}-Z\|_2^\beta \big]\big)^{1/\beta}  =: c\,\wt P_1 \wt P_2. 
\eeao
The quantity $\wt P_1$ is bounded by a constant for all $p$.
Proceeding as for $\beta\in (0,1]$, we obtain
\beao
 \wt P_2 \le c \big(p^{1-\beta/2}(\tilde p^{-\gamma_X/d}+p^{-1} )\big)^{1/\beta},
\eeao
and we finally have
\[
 \E [|I_1|] \le c \big(p^{1-\beta/2}(\tilde p^{-(\gamma_X\wedge \gamma_Y)/d}+p^{-1})\big)^{1/\beta}.
\]
The quantities $\E[|I_i|],i=2,3$ can be bounded in a similar way. This proves part {\bf 2}.\\
{\bf 3.} We will show that $\E[|T_{n,\beta}(X,X)-T_{n,\beta}(X^{(p)},X^{(p)})|]\to 0$
as $\nto$.
We again use the decomposition \eqref{eq:may8f} %(6.2) of the main document \citep{mmrt:main} 
for $X=Y$.\\
%\begin{align*}
%I_1 &= \frac{1}{n^2} \sum_{k,l=1}^n
%\big( \|X_k^{(p)}-X_l^{(p)}\|_2^{2\beta}  -\|X_k-X_l\|_2^{2 \beta} \big), \\
%I_2 & = \Big( \frac{1}{n^2} \sum_{k,l=1}^n \|X_k^{(p)}-X_l^{(p)}
% \|_2^\beta \Big)^2
% - \Big(\frac{1}{n^2}
%\sum_{k,l=1}^n \|X_k-X_l\|_2^\beta \Big)^2, \\
%I_3 &=\frac{1}{n^3} \sum_{k,l,m=1}^n \Big(
% \|X_k^{(p)}-X_l^{(p)}\|_2^\beta
% \|X_k^{(p)}-X_m^{(p)}\|_2^\beta-\|X_k-X_l\|_2^\beta
% \|X_k-X_m\|_2^\beta\Big) 
%\end{align*}
{\em The finite variance case. Assume {\rm (A1), (A2).}}\\ 
{\em Assume $\beta\in(0,1]$.} 
We 
observe that 
\begin{align}
 \E[|I_1|] &\le \E \big[
\|Z^{(p)}-Z\|_2^\beta
(\|Z^{(p)}\|_2^{\beta}+\|Z\|_2^\beta)
\big] \nonumber \\
& \le \big(\E [
\|Z^{(p)}-Z \|_2^{2\beta}]
\big)^{1/2} \big(
(\E[\|Z^{(p)}\|_2^{2\beta}])^{1/2}+ (\E[\|Z\|_2^{2\beta}])^{1/2}
\big). \label{ineq:dvfinite:betale1}
\end{align}
We can re-use the argument for bounding
 $\E[\|Z^{(p)}-Z\|^\beta_2]$, replacing $\beta/2$ by $\beta$
in the derivation. Thus the first expectation in \eqref{ineq:dvfinite:betale1} is
 bounded by $c\big(
\tilde p^{-\gamma_X/d}+p^{-1}
\big)^\beta$. The remaining two expectations are finite thanks to [Lemma
 \ref{lem:momentcondiap2} (1)]. Thus we have
\[
 \E[|I_1|]\le c \big(
\tilde p^{-\gamma_X/d}+p^{-1}
\big)^{\beta/2}\to 0\,,\qquad \nto\,.
\]
{\em Assume $\beta \in (1,2)$.} 
We may proceed similarly as for the distance covariance. We have
\beao
\E[|I_1|]&\le & \E[\big|
\|Z^{(p)}\|_2^{2\beta} -\|Z\|_2^{2\beta} 
\big|] \\
& \le& c \E \big[
(\|Z^{(p)}\|_2^{2\beta-1} \vee \|Z\|_2^{2\beta-1})\,
\|Z^{(p)}- Z\|_2
\big] \\
& \le& c \big(
\E [\|Z^{(p)}\|_2^{2(2\beta-1)} \vee \|Z\|_2^{2(2\beta-1)} ]
\big)^{1/2} \big(\E[
\|Z^{(p)}- Z\|_2^2] 
\big)^{1/2} \\
&\le & c \,  \big(\E[
\|Z^{(p)}- Z\|_2^2] 
\big)^{1/2}\,.
\eeao
In the last step we used  [Lemma~\ref{lem:momentcondiap2} (2)] and (A2). 
Now we may proceed as for the bound of \eqref{secondmethodpfineq1} to obtain 
\[
\E[|I_1|]\le c \big(
\tilde p^{-\gamma_X/d} +p^{-1}
\big)^{1/2}\to 0\,,\qquad \nto\,.
\]
We can deal with $\E[|I_2|],\E[|I_3|]$ in the same way by observing that 
\begin{align*}
 |I_2| \le& \frac{1}{n^4} \sum_{j,k,l,m=1}^n 
%\big| \|X_j^{(p)}-X_k^{(p)}\|_2^\beta-\|X_j-X_k\|_2^\beta \big| 
|D_p^{(\beta)}(X_j-X_k)|
\| X_l^{(p)}-X_m^{(p)}\|_2^\beta +\frac{1}{n^4} \sum_{j,k,l,m=1}^n 
  \|X_j-X_l\|_2^\beta |D_p^{(\beta)}(X_l-X_m)| \\
% \big| \|X_l^{(p)}-X_m^{(p)}\|^\beta-\|X_l-X_m\|_2^\beta \big| \\
&=: I_{21}+I_{22}
\end{align*}
and 
\begin{align*}
 |I_3| \le& \frac{1}{n^3} \sum_{k,l,m=1}^n |D_p^{(\beta)}(X_k-X_l)|
%\big| \|X_k^{(p)}-X_l^{(p)}\|_2^\beta-\|X_k-X_l\|_2^\beta \big| 
\| X_l^{(p)}-X_m^{(p)}\|_2^\beta +\frac{1}{n^3} \sum_{k,l,m=1}^n 
  \|X_k-X_l\|_2^\beta  %\big| \|X_l^{(p)}-X_m^{(p)}\|^\beta-\|X_l-X_m\|_2^\beta \big|
|D_p^{(\beta)}(X_l-X_m) |
,\\
&=: I_{31}+I_{32}.
\end{align*}
This means that
 $\E[I_{21}],\E[I_{22}],\E[I_{31}],\E[I_{32}]$ can be  bounded in a similar way
and these expectations converge to zero as $\nto$.  We illustrate this for $I_{32}$.\\
{\em Assume $\beta\in(0,1]$.} By the Cauchy-Schwarz
 inequality and using similar bounds as above, we have 
\begin{align}
\label{ineqi32}
\begin{split}
 \E[I_{32}] &\le \big(
\E [\|X_1-X_2\|_2^{2\beta}]
\big)^{1/2} \big(
\E[|D_p^{(\beta)}|^2]
%\big|
%\|X_1^{(p)}-X_3^{(p)}\|_2^\beta -\|X_1-X_3\|_2^\beta
%\big|^2
\big)^{1/2} \\
& \le \big(
\E [\|X_1-X_2\|_2^{2\beta}]
\big)^{1/2} \big( \E 
\big[ \|X_1^{(p)}-X_3^{(p)} -(X_1-X_3)\|_2^{2\beta} \big] 
\big)^{1/2} \\
& \le c\big(
\tilde p^{-\gamma_X/d} +p^{-1}
\big)^{\beta/2}\to 0\,,\qquad \nto\,.
\end{split}
\end{align}
For $\beta\in (1,2)$, multiple use of H\"older's inequality yields
\beao
  \E[I_{32}] &\le &
c\E[
\|(X_1^{(p)}-X_2^{(p)})-(X_1-X_2)\|_2
\big(
\|X_1^{(p)}-X_2^{(p)}\|_2^{\beta-1}\vee \|X_1-X_2\|_2^{\beta-1}
\big) \|X_1-X_3\|_2^{\beta}
] \\
& \le& 
c \big(
\E[\|Z^{(p)}-Z\|_2^{2}]
\big)^{1/2} \big(
\E [\|Z\|_2^{2(2\beta-1)} ]
\big)^{\frac{\beta}{2(2\beta-1)}}\,\big(
\E[
\|Z^{(p)}\|_2^{2(2\beta-1)}\vee \|Z\|_2^{2(2\beta-1)}
]
\big)^{\frac{\beta-1}{2(2\beta-1)}} \\
&\le& 
c \,\big(
\E [\|Z^{(p)}-Z \|_2^{2} ]
\big)^{1/2} \big(
\E [\|Z\|_2^{2(2\beta-1)} ]
\big)^{\frac{\beta}{2(2\beta-1)}}
\,\\
&&\times \Big(
\big(
\E[
\|Z^{(p)}\|_2^{2(2\beta-1)} ]\big)^{\frac{\beta-1}{2(2\beta-1)}}
 + \big(
\E[ \|Z\|_2^{2(2\beta-1)}
]
\big)^{\frac{\beta-1}{2(2\beta-1)}} \Big).
\eeao
The first quantity is bounded by 
$c \big(
\tilde p^{-\gamma_X/d}+p^{-1}
\big)^{1/2}$ and the other quantities are finite by virtue of [Lemma
 \ref{lem:momentcondiap2} (2)]. Hence $\E[I_{32}]\to 0$, $\nto$.
\\[1mm]
\noindent {\em  The infinite variance case. We assume 
a finite $2\beta$th moment of $X$, $\beta\in(0,1)$, and {\rm (B3).}}\\
{\em Let $2\beta\in(0,1]$.} First we follow the inequality 
\eqref{ineq:dvfinite:betale1} of distance variance 
and then we use the bound of $\E[\|Z^{(p)}-Z\|_2^{2\beta}]$ 
for the distance covariance in the infinite variance case; see \eqref{betamoment:I1} with $\beta$ replaced by $2\beta$. 
By [Lemma \ref{lem:momentcondiap2} (1)], $\E[
\|Z^{(p)}\|_2^{2\beta}]$ and  $\E\|Z\|_2^{2\beta }
]$ are bounded by a constant. As a result we have by (B3),
\begin{align}
 \E[|I_1|]  \le c \big( \E[\|Z^{(p)}-Z\|_2^{2\beta}]\big)^{1/2}  \le c \big( p^{1-\beta} \big(
 \tilde p^{-\gamma_X'/d}+p^{-1} 
\big) \big)^{1/2}.\label{I1:dv:betamoment1}
\end{align}
The \rhs\ converges to zero as $\nto$ provided $1<\beta+\gamma_X'/d$.
For $2\beta \in (1,2)$ we have by H\"older's inequality 
\begin{align*}
 \E[|I_1|] & \le 
\E \big[
\big(\|Z^{(p)}\|_2^{2\beta-1} \vee \|Z\|_2^{2\beta-1}\big)
\|Z^{(p)}-Z\|_2
\big] \\
& \le c \big(\E[
\|Z^{(p)}\|_2^{2\beta} \vee \|Z\|_2^{2\beta}
]
\big)^{(2\beta-1)/(2\beta)} \big(
\E[ \|Z-Z^{(p)}\|_2^{2\beta}]
\big)^{1/(2\beta)} \\
&= c\widehat P_1 \widehat P_2. 
\end{align*}
We can bound $\wh P_2$ in the same way as for $\beta\in(0,1/2]$:
\beao
 \widehat P_2 \le c \big(
p^{1-\beta}\big(
\tilde p^{-\gamma_X'/d}+p^{-1}
\big)
\big)^{1/(2\beta)} .
\eeao
The quantity $\widehat P_1$ is finite by virtue of 
[Lemma \ref{lem:momentcondiap2} (1)] and (B3). Thus
we have the bound for general $\beta\in (0,1)$: 
\beao
 \E[|I_1|] \le c \big(
p^{1-\beta} 
(\tilde p^{-\gamma_X'/d}+p^{-1}
) \big)^{(1\wedge \beta^{-1})/2}\to 0\,,\qquad \nto .
\eeao
For $I_2$ and $I_3$, recall the bounds $|I_2|\le I_{21}+I_{22}$ and
 $|I_3|\le I_{31}+I_{32}$, and take expectations on both sides of the 
inequalities. We illustrate how to deal with $\E[I_{32}]$. 
Since $2\beta \in (0,2)$ inequality
 \eqref{ineqi32} holds.
Then for $2\beta \in (0,1]$ it follows from \eqref{I1:dv:betamoment1} that $\E[I_{32}]\le c \big(p^{1-\beta}(\tilde p^{-\gamma_X'/d}+p^{-1})\big)^{1/2}$, 
and for $2\beta \in (1,2)$ we obtain the same bound as for $\wh P_2$ by the similar calculations.\\
%similer to the quantity $\wt P_2$, we obtain the same bound. 
%$\E[I_{32}]\le c \{\tilde p^{1-\beta}(\tilde p^{-\gamma_X}+p^{-1})\}^{1/2}$.  
This finishes the proof.

\appendix

\section{Negative moments of a Poisson \rv}\setcounter{equation}{0}
\label{Sec4supp}
We consider a family $(N_p)_{p>0}$ of Poisson \rv s with $\E[N_p]=p$ and 
are interested in the \asy\ behavior of $\E[N_p^{-\gamma}\1(N_p>0)]$ for $\gamma>0$ as $p\to\infty$.
\ble
\label{lem:poisson}
We consider a family $(N_p)_{p>0}$ of Poisson \rv s with $\E[N_p]=p$. Then for
any $\gamma>0$, $p^\gamma\E[N_p^{-\gamma}\1(N_p>0)]\to 1$. 
\ele
\begin{proof} We observe that
\beao
\E[N_p^{-\gamma}\1(N_p>0)]&=&\Big(\sum_{k\in A_p}+ \sum_{k\in A_p^c}\Big) k^{-\gamma}\ex^{-p} \dfrac{p^k}{k!}\,,
\eeao
where $A_p=\{k\ge 1:|k-p|>\vep_p\,p\}$ and 
$A_p^c=\{k\ge 1:|k-p|\le \vep_p\,p\}$  and $\vep_p= \sqrt{ (C\log p)/p}$
for some constant $C=C(\gamma)>0$. We will show that 
\beam\label{eq:march10a}
p^\gamma\,\P(|N_p-p|>\vep_p\,p)\to 0\,,\qquad p\to\infty\,.
\eeam
Then $\P(|N_p-p|\le \vep_p\,p)\to 1$ and 
\beao
p^{\gamma}\,\sum_{k\in A_p^c} k^{-\gamma}\ex^{-p} \dfrac{p^k}{k!}
&=&\sum_{k\in A_p^c} (p/k)^{\gamma}\ex^{-p} \dfrac{p^k}{k!}\\
&=&(1+o(1))\,\P(|N_p-p|\le \vep_p\,p)\to 1\,.
\eeao
Next we turn to the proof of \eqref{eq:march10a}. By Markov's inequality
and a Taylor expansion we have for $h=\vep_p$, 
\beao
p^\gamma\,\P(N_p>(1+\vep_p)\,p)&\le & p^\gamma\,\ex^{-h\,(1+\vep_p)\,p}\,\E[\ex^{h\,N_p}]
=p^\gamma\,\exp(-h\,(1+\vep_p)\,p -p(1-\ex^{h}))\\
&=&p^\gamma\,\exp\big(-(\vep_p\,p+\vep_p^2\,p) + (\vep_p\,p + 0.5 \vep_p^2\,p(1+o(1)))\big)\\
&=&\exp\big(-0.5 \vep_p^2\,p(1+o(1)) + \gamma\,\log p\big)\\&=& 
\exp\big(-0.5 C \log p (1+o(1))+\gamma\,\log p\big)\to 0\,,\qquad p\to\infty,
\eeao 
provided $C >2\gamma$. On the other hand, by monotonicity of the Poisson
\pro ies for $k\le p$ we have 
\beao
\ex^{-p} \dfrac{p^k}{k!}\le 
\ex^{-p} \dfrac{p^{[p(1-\vep_p)]}}{[p(1-\vep_p)]!}\,,\qquad k\le [p(1-\vep_p)]\,.
\eeao
Hence by Stirling's formula, $(\sqrt{2\pi n}(n/\ex)^n\le n!)$,
\beao
p^\gamma\,\P(p-N_p>\vep_p\,p) 
&=& p^\gamma\,\P(N_p< p(1-\vep _p))
\le p^{\gamma+1}\ex^{-p} \dfrac{p^{[p(1-\vep_p)]}}{[p(1-\vep_p)]!}\\
&\le &c\,p^{\gamma+0.5}\ex^{-p+[p\,(1-\vep_p)]}\Big(\dfrac{p}{[p(1-\vep_p)]}\Big)^{[p(1-\vep_p)]}\\
&\le & c\exp\Big( (\gamma+0.5)\log p -p\,\vep_p+  [p(1-\vep_p)]\,\log \dfrac{p}{[p(1-\vep_p)]} \Big)=:c\,\ex^{v(p)}\,.
\eeao
We have by a Taylor expansion, for large $p$,
\beao
v(p)-(\gamma+0.5)\,\log p&=& -p\,\vep_p+  [p(1-\vep_p)]\Big( \dfrac{p}{[p(1-\vep_p)]}-1 - \dfrac 1 2 \Big(\dfrac{p}{[p(1-\vep_p)]}-1\Big)^2(1+o(1))\Big)\\
&\le & - 0.25\,p\,\vep_p^2= -0.25\,C \log p\,.
\eeao
Choosing $C>4\,(\gamma+0.5)$, we have $\ex^{v(p)}\to 0$.
Combining these bounds, we proved the lemma.
\end{proof}

\noindent {\bf Acknowledgments}
Muneya Matsui's research is partly supported by the JSPS Grant-in-Aid for Scientific Research C
(19K11868). Thomas Mikosch's research is partially supported by Danmarks Frie Forskningsfond Grant No 9040-00086B

\end{document}